%% file: arxiv.tex
\documentclass[11pt]{article}

\usepackage[a4paper,margin=1in]{geometry}
\usepackage{amsthm}
\usepackage{graphicx,subfigure}

\theoremstyle{plain}
\newtheorem{theorem}{Theorem}
\newtheorem{corollary}{Corollary}
\newtheorem{lemma}{Lemma}
\newtheorem{proposition}{Proposition}

\theoremstyle{definition}
\newtheorem{definition}{Definition}
\newtheorem{example}{Example}

\theoremstyle{remark}
\newtheorem{remark}{Remark}

\theoremstyle{plain}
\input{shared/macro.tex}

\colorlet{blue}{black}
\usepackage{hyperref}

\title{From Manifold Identification to Newton Acceleration on Intersections: Sparse Stiefel Optimization}

\author{
Shixiang Chen\thanks{School of Mathematical Sciences, Key Laboratory of the Ministry of Education for Mathematical Foundations and Applications of Digital Technology, University of Science and Technology of China, Hefei, Anhui, China. Email: \texttt{shxchen@ustc.edu.cn}.}
\and
Wen Huang\thanks{School of Mathematical Sciences, Xiamen University, Xiamen, China. Email: \texttt{wen.huang@xmu.edu.cn}.}
}

\date{}

\begin{document}

\maketitle

\begin{abstract}
We study a Newton acceleration for sparse composite optimization
on the Stiefel manifold. The main difficulty is geometric: the active
manifold identified by the nonsmooth regularizer may fail to intersect the
Stiefel manifold transversely, which obstructs a Riemannian Newton step
on the identified manifold. In the transverse case, we prove local
identification of the ManPG tangent proximal mapping. For nontransverse
cases, we introduce an off-diagonally perturbed Stiefel family that
generically restores the identification geometry while yielding an
\(O(\|\Delta\|_F)\)-KKT guarantee for the original problem. We also derive
verifiable support-level conditions for clean intersection, which cover
nontransverse sparse patterns and yield the smooth moving local models
used by the Newton correction. Based on these results, we propose MIX, a safeguarded ManPG/Newton-CG
method on moving identified intersections. In the general clean-intersection
setting, we prove global descent and KKT-residual guarantees for MIX.
In the transverse or generically perturbed cases,
{\color{blue} if the sequence has an accumulation point satisfying certain
	regularity assumptions and the second-order sufficient condition (SOSC), then the full sequence converges
	to that point, with finite active-manifold identification and a local
	Q-superlinear rate.}  Numerical experiments on compressed modes and sparse PCA show
that MIX substantially improves efficiency while preserving solution quality.  Beyond the Stiefel manifold, we also outline how the safeguarded global-convergence mechanism of MIX extends to general smooth equality-constrained manifolds.
\end{abstract}

\noindent\textbf{Keywords:} Stiefel manifold;  Manifold
identification; Partial smoothness; Clean intersection; Transversality;  Newton-CG

\setcounter{tocdepth}{2}
\compacttableofcontents

\input{shared/section_introduction}
\input{shared/section_preliminary}
\input{shared/section_identification_ManPG}
\input{shared/section_geometry_intersection}
\input{shared/section_optimization_on_intersection}

\input{shared/section_algorithm_intersection}
\input{shared/section_num_exp}

\section*{AI Use Disclosure}

The central research questions and methodological framework—including the formulation of the sparse Stiefel problem, the perturbation strategy, the design of MIX, and the use of Newton corrections on identified intersections—were formulated by the authors. ChatGPT assisted with specific searches of clean-intersection counterexamples.

In preparing Version 3 of this manuscript, the authors used ChatGPT
(GPT-5.6 Sol) to assist in developing and checking the revised local
convergence analysis, in particular the proof that removes the
sequence convergence  assumption from the local convergence analysis of MIX.
ChatGPT was additionally used for language
polishing and implementation-level code refactoring. All AI-assisted mathematical arguments and examples were independently
checked and verified by the authors, who take full responsibility for
all mathematical statements, proofs, algorithms, numerical results,
and conclusions in this manuscript.
\bibliographystyle{unsrt}
\bibliography{manifold}
\addtocontents{toc}{\protect\setcounter{tocdepth}{1}}
\input{shared/appendix}

\end{document}

%% file: shared/macro.tex
\usepackage{amsmath}
\usepackage{amssymb}
\usepackage{mathrsfs}
\usepackage{bbm}
\usepackage{algpseudocode}
\usepackage{algorithm}
\usepackage{placeins}
\usepackage{xcolor}
\usepackage{tikz,caption}
\usetikzlibrary{arrows.meta,calc}
 \usepackage{enumitem}

\usepackage{multirow}

\newtheorem{assumption}{Assumption}

\numberwithin{equation}{section}
\newcounter{manualtaggedequation}
\newenvironment{manualtaggedequation}{%
	\stepcounter{manualtaggedequation}%
}{}

\numberwithin{lemma}{section}
\numberwithin{proposition}{section}
\numberwithin{corollary}{section}
\numberwithin{definition}{section}
\numberwithin{remark}{section}

\usepackage{booktabs}

\makeatletter
\newif\ifmm@svjour
\@ifclassloaded{svjour3}{\mm@svjourtrue}{\mm@svjourfalse}
\newcommand{\compacttableofcontents}{%
	\ifmm@svjour
		\tableofcontents
	\else
		\begingroup
		\setlength{\parskip}{0pt}%
		\renewcommand*\l@section[2]{%
			\ifnum \c@tocdepth >\z@
				\addpenalty\@secpenalty
				\addvspace{0.35em \@plus\p@}%
				\setlength\@tempdima{1.5em}%
				\begingroup
					\parindent \z@ \rightskip \@pnumwidth
					\parfillskip -\@pnumwidth
					\leavevmode \bfseries
					\advance\leftskip\@tempdima
					\hskip -\leftskip
					##1\nobreak\hfil \nobreak\hb@xt@\@pnumwidth{\hss ##2}\par
				\endgroup
			\fi}%
		\tableofcontents
		\endgroup
	\fi
}
\makeatother

\newcommand{\M}{\mathcal{M}}
\newcommand{\N}{\mathcal{N}}

\newcommand{\rmN}{\mathrm{N}}

\newcommand{\rmD}{\mathrm{D}}

\newcommand{\Ob}{\mathrm{Ob}}

\newcommand{\U}{\mathbf{U}}

\newcommand{\R}{\mathbb{R}}
\newcommand{\E}{\mathbb{E}}

\newcommand{\be}{\begin{equation}}
\newcommand{\ee}{\end{equation}}
\newcommand{\ba}{\begin{array}}
\newcommand{\ea}{\end{array}}
\newcommand{\bad}{\begin{aligned}}
\newcommand{\ead}{\end{aligned}}

\newcommand{\normfro}[1]{\| #1 \|_{\text{F}}}

\newcommand{\inp}[2]{\left\langle #1, #2 \right\rangle}
\newcommand{\argmin}{\mathop{\rm argmin}}

\newcommand{\ri}{\mathrm{ri}}              
\newcommand{\aff}{\mathrm{aff}}            
\newcommand{\para}{\mathrm{par}}            

\newcommand{\ip}[2]{\left\langle #1,#2\right\rangle}

\newcommand{\dist}{\mathrm{dist}}

\newcommand{\rank}{\mathrm{rank}}
\newcommand{\range}{\mathrm{range}}

\newcommand{\diag}{\mathrm{diag}}

\newcommand{\Retr}{\mathrm{Retr}}

\newcommand{\Tr}{\mathrm{Tr}}

\newcommand{\T}{\mathrm{T}}
 
\newcommand{\st}{\mathrm{s.t. }}

\newcommand{\vet}{\mathrm{vec}}
\newcommand{\St}{\mathrm{St}}

 \newcommand{\grad}{\mathrm{grad}}
 \newcommand{\Hess}{\mathrm{Hess}}
\providecommand{\vvec}{\mathrm{vec}}
\newcommand{\svec}{\overline{\mathrm{vec}}}

\newcommand{\Proj}{\mathcal{P}}

\newcommand{\Exp}{\mathrm{Exp}}

\newcommand{\supp}{\operatorname{supp}}

%% file: shared/section_introduction.tex
\section{Introduction}
In this paper, we consider the following composite optimization problem on the
Stiefel manifold:
\begin{manualtaggedequation}
\begin{equation}\label{prob:nonsmooth_stiefel}\tag{$\mathrm{P}_0$}
	\min_X F(X):=f(X)+g(X),
	\quad \st \quad
	X\in\St(n,r),
\end{equation}
\end{manualtaggedequation}
where
\[
\St(n,r):=\{X\in\R^{n\times r}\mid X^\top X=I_r\}.
\]
Here \(f:\R^{n\times r}\to\R\) is smooth with Lipschitz continuous Euclidean
gradient, and \(g:\R^{n\times r}\to\R\) is convex and Lipschitz continuous.
Representative examples include the entrywise \(\ell_1\) regularizer and the
group-sparsity \(\ell_{2,1}\) regularizer,
\[
g(X)=\lambda\|X\|_1
=
\lambda\sum_{i=1}^n\sum_{j=1}^r |X_{ij}|,
\qquad
\text{or}
\qquad
g(X)=\lambda\|X\|_{2,1}
=
\lambda\sum_{i=1}^n \|X_{i,:}\|_2,
\]
where \(X_{i,:}\) denotes the \(i\)-th row of \(X\). In these cases,
problem~\eqref{prob:nonsmooth_stiefel} seeks an orthogonal matrix with
entrywise or row-wise sparsity. Applications arise in sparse PCA \cite{Zou-spca-2006,d2007direct,journee2010generalized,cai2013sparse}, compressed modes \cite{Ozolins2013,Lai2014}, and other structured matrix optimization models; see \cite{chen2024nonsmooth}.

A number of algorithms have been developed for nonsmooth
optimization on Riemannian manifolds. ManPG~\cite{chen2020proximal} extends
the classical proximal-gradient method to embedded submanifolds. At each
iteration, it computes a tangent proximal step
\begin{equation}\label{eq:manpg_subprob_intro}
	V_k
	=
	\arg\min_{V\in \T_{X_k}\St(n,r)}
	\big\langle \nabla f(X_k),V\big\rangle
	+\frac{1}{2t}\|V\|_{\mathrm F}^2
	+g(X_k+V),
\end{equation}
and then retracts along this direction. General Riemannian
proximal-gradient frameworks and accelerated variants based on
retraction-type models were developed by Huang and
Wei~\cite{HuangWei2022RPG}. Other first-order approaches include dynamic
smoothing methods~\cite{BeckRosset2023DynamicSmoothing}, which apply
Riemannian gradient steps to smooth approximations with decreasing smoothing
parameters, and primal-dual methods~\cite{LiMaSrivastava2024RADMM}\cite{deng2025oracle}\cite{xu2026riemannian}.

Several second-order   methods have
also been proposed. ManPQN~\cite{wang2023proximal} incorporates second-order
information of the smooth term into the ManPG subproblem, while
RPN~\cite{si2023riemannian} and RPNCG~\cite{huang2024riemannian} develop
Riemannian proximal Newton-type methods based on second-order manifold
geometry and semismooth analysis. Augmented-Lagrangian methods with semismooth
Newton subproblem solvers have also been studied for nonsmooth optimization on
matrix manifolds~\cite{ZhouBaoDingZhu2023ALMSSN}. These works demonstrate that
Newton and semismooth-Newton techniques can be effective, and in some settings
can yield local superlinear convergence under certain regularity assumptions.
However, for sparse Stiefel problems, the local fast-convergence conditions
depend on a constraint qualification for the intersection between the sparse
active manifold and the Stiefel constraint. Existing Newton-type frameworks do
not verify this intersection geometry for the sparse Stiefel model. 

This gap leads us to the manifold-identification viewpoint adopted
in this paper. Identifiable surfaces and active constraints were studied in
\cite{burke1988identification,wright1993identifiable}, and partial smoothness
was developed in~\cite{lewis2002active}. In Euclidean composite optimization,
proximal-gradient-type methods identify the active manifold in finite time
under partial smoothness and nondegeneracy; see
\cite{liang2017activity,bareilles2023newton,bareilles2023harnessing}. 
{\color{blue}More recently,  \cite{curtis2025proximal}
developed a proximal-gradient method for regularized optimization with
general nonlinear constraints and established a manifold identification
result under partial smoothness. Their identification theorem additionally
assumes convergence of the associated Lagrange multiplier estimates along
the subsequence converging to the KKT point.}
For
Stiefel-constrained problems, however, the  ManPG proximal step is computed in the
affine tangent space and is then followed by a retraction. Thus the tangent
proximal predictor
\(Y_k=\mathcal T_{t_k}(X_k)=X_k+V_k\) may identify the active manifold, whereas
the retracted iterate need not preserve the identified sparsity pattern. We
therefore study identification at the level of the ManPG tangent proximal
predictor.

Let \(X^\star\) be a nondegenerate critical point and let
\(\M_g(X^\star)\) denote the active manifold of \(g\) at \(X^\star\). Once this
manifold is identified, the nonsmooth problem is locally represented by the
smooth restricted problem on
\[
\St(n,r)\cap \M_g(X^\star).
\]
To apply a Newton-type correction on this restricted problem, one needs a
constraint qualification for the intersection. The standard condition is
transversality. Under transversality, the ManPG tangent proximal mapping
identifies \(\M_g(X^\star)\) locally, and {\color{blue}finite identification
follows once the iterates remain in the corresponding neighborhood.}

For the entrywise \(\ell_1\)-regularized Stiefel problem with \(r=1\),
transversality is verified in \cite{si2023riemannian}. However, we show that  it   may fail for common sparse patterns with $r>1$. To
handle such nontransverse cases, we perturb only the off-diagonal
orthogonality constraints and consider
\begin{manualtaggedequation}
\begin{equation}\label{prob:P-Delta}
	\tag{$\mathrm{P}_\Delta$}
	\min_{X \in \St_\Delta(n,r)} F(X):=f(X)+g(X),
\end{equation}
\end{manualtaggedequation}
where
\begin{equation}\label{def:St-Delta}
	\St_\Delta(n,r)
	:=
	\bigl\{
	X \in \mathbb{R}^{n\times r} :
	X^\top X = I_r+\Delta
	\bigr\}.
\end{equation}
Here
\(\Delta\in\mathbb S_0^r
:=\{\Delta\in\mathbb S^r:\operatorname{diag}(\Delta)=0\}\) and
\(I_r+\Delta\succ0\). The zero-diagonal restriction preserves the unit-norm
constraints on the columns and perturbs only the off-diagonal orthogonality
equations. Treating these off-diagonal equations as parameters allows us to
use Sard-type and parametric transversality arguments to avoid exceptional
degeneracies~\cite{sard1942measure,tang2024feasible}. We also prove that
every KKT point of \eqref{prob:P-Delta} is an
\(O(\|\Delta\|_{\mathrm F})\)-KKT point of
\eqref{prob:nonsmooth_stiefel}.

The off-diagonal perturbation is introduced to obtain a generic
transversality result for the auxiliary perturbed problems~\eqref{prob:P-Delta}.
 Under a finite-family assumption on active manifolds,
we prove that for almost every  \(\Delta\), and around any
nondegenerate critical point \(X^\star\in\St_\Delta(n,r)\), the ManPG tangent
proximal mapping \(\mathcal T_t\) is locally single-valued, \(C^1\), and
identifies   the active manifold associated with $X^*$,
\(\mathcal T_t(X)\in \M_g(X^\star)\)
for all \(X\) sufficiently close to \(X^\star\). If the original
transversality condition already holds, the same conclusion applies with
\(\Delta=0\). Thus no perturbation is needed in transverse cases, such as the
sphere case and the row-sparse \(\ell_{2,1}\)-regularized model.

After identification, the predictor
\(Y_k=\mathcal T_{t_k}(X_k)=X_k+V_k\) already carries the active structure.
It is therefore natural to build the Newton correction at \(Y_k\). The fixed
intersection \(\St_\Delta(n,r)\cap\M_g(Y_k)\), however, may be empty. We
instead use the moving intersection
\[
\M_k
:=
\St_{\Delta_k}(n,r)\cap\M_g(Y_k),
\qquad
\Delta_k:=\Delta+V_k^\top V_k .
\]
This set is nonempty by construction: since
\(V_k\in \T_{X_k}\St_\Delta(n,r)\), one has
\(Y_k^\top Y_k=I_r+\Delta_k\). Moreover, the ManPG safeguard gives
\(V_k\to0\), and hence
\(\Delta_k-\Delta=V_k^\top V_k\to0\). Thus the moving Stiefel constraint is
only a vanishing second-order perturbation of the fixed one.

It remains to ensure that the moving intersection provides a smooth local
model for Newton correction. For sparse active manifolds, this amounts to a
rank condition for the Stiefel constraint restricted to the active manifold,
or equivalently to a clean-intersection condition, which is weaker than
transversality~\cite{drusvyatskiy2015transversality}. We analyze this
condition through the support graph of the identified sparse pattern. In low-rank cases this condition can be verified deterministically: it is
automatic for \(r=1\), holds for \(r=2\) under general perturbations, and holds
for the unperturbed \(r=3\) case. For \(r\ge4\), it is not automatic, and we
provide sufficient conditions, in particular the edge-wise \(2/3\)-row
cover condition, together with a probabilistic estimate for its occurrence.
 Under
these rank/clean-intersection conditions, the resulting intersection manifolds
admit computable second-order retractions based on alternating-projection
methods~\cite{chen2026alternating}.

These ingredients lead to MIX, a safeguarded ManPG/Newton-CG method on
identified intersection manifolds. The perturbation and transversality analysis
serve to obtain finite identification and a well-defined local model in
nontransverse cases. If the Newton-CG correction is not accepted, MIX
falls back to the ManPG step, which provides the global descent safeguard and
KKT-residual guarantee.

Our contributions can be summarized as follows.
\begin{itemize}
	\item We introduce the perturbed problem \eqref{prob:P-Delta} on
	\(\St_\Delta(n,r)\). We prove that its KKT points are
	\(O(\|\Delta\|_{\mathrm F})\)-KKT points of the original problem and that,
	for almost every \(\Delta\), the ManPG tangent proximal mapping identifies the active manifold associated with a nondegenerate critical point $X^*$; see
		Theorem~\ref{thm:id-strong-crit}. If transversality already holds, the same
		identification holds for problem~\eqref{prob:nonsmooth_stiefel}.

		\item For sparse Stiefel problems, we characterize the local
	intersection geometry between the Stiefel constraint and the active
	manifold. We prove deterministic clean-intersection/rank results for low
	ranks, including the sphere case, general \(r=2\), and the unperturbed
	\(r=3\) case, and give support-level and probabilistic verification
	conditions for \(r\ge4\); see
		Theorem~\ref{thm:Mg-intersection-submanifold} and
		Proposition~\ref{prop:row-cover-bernoulli}.

		\item We develop MIX, a safeguarded ManPG/Newton-CG method on moving
		identified intersection manifolds. In the general clean-intersection setting,
		we prove global descent and an iteration-complexity bound for the ManPG
		residual; Theorem~\ref{thm:global_convergence_mix} further yields a
			vanishing KKT residual for the original problem when \(\Delta=0\), and an
			\(O(\|\Delta\|_F)\)-KKT guarantee when \(\Delta\neq0\). In the transverse or generically perturbed local regime, if the
			sequence has an accumulation point satisfying the stated regularity
			assumptions and the SOSC, then the full sequence converges to that
			point, the ManPG predictors identify the active manifold in finitely
			many iterations, and the convergence is locally Q-superlinear; see Lemma~\ref{lem:alpha_one_eventually} and
			Theorem~\ref{thm:local_superlinear_full_chain}.

		\item We further show that the global-safeguard mechanism of MIX extends
	to general \(C^2\) equality-constrained embedded manifolds, provided the
	corresponding local intersection, retraction, and restoration steps are
	available; see Appendix~\ref{app:mix-general-equality}.
\end{itemize}
The paper is organized as follows. Section~\ref{sec:prelim} introduces notation and preliminary material from Riemannian optimization and nonsmooth analysis. Section~\ref{sec:id_manpg} establishes the manifold identification result for ManPG under a finite-family assumption on active manifolds. Section~\ref{sec:intersection_property} studies the geometry of the intersection between the Stiefel manifold and sparse active manifolds. Section~\ref{sec:opt_on_intersection} develops the geometric ingredients for optimization on the intersection manifold. Section~\ref{sec:alg_on_intersection} presents the hybrid algorithm MIX and its convergence analysis. Numerical results are reported in Section~\ref{sec:numerical_exp}.  The appendices collect technical proofs, examples for Section~\ref{sec:intersection_property}, derivations for projection and Hessian formulas, and basic computations on $\St_\Delta$.

%% file: shared/section_preliminary.tex
\section{Preliminaries}\label{sec:prelim}
\subsection{Riemannian optimization}
Let us review some standard results   for  Riemannian geometry and Riemannian optimization. These can also be found in \cite{lee2012smooth,Absil2009,boumal2023introduction}.

Let $\M,\N$ be smooth manifolds, and let $h:\M\rightarrow\N$ be any map. We say that $h$ is a smooth map if for every $p\in \M$, there exists smooth charts $(U,\varphi)$ containing $p$ and $(V,\phi)$ containing $h(p)$ such that $h(U)\subset V$ and the composite map  $\phi\circ h \circ\varphi^{-1}$ is smooth from $\varphi(U)$ to $\phi(V).$
\begin{proposition}\cite[Theorem 5.12]{lee2012smooth}\label{prop:submanifold}
    Let $\M$ and $\N$ be smooth manifolds, and let $h:\M\rightarrow\N$ be a smooth map with constant rank $c$.  Then each nonempty level set of $h$ is a properly embedded submanifold of codimension $c$ in $\M.$
\end{proposition}
\begin{definition}\label{def:local_def}
    A \emph{local defining function} for an embedded submanifold $\mathcal{S} \subset \M$ is a smooth map $\Phi: \mathcal{U}\to \mathbb{R}^k$, defined on an open neighborhood $\mathcal{U}$ of $p \in \mathcal{S} $, such that:
\begin{enumerate}
\item $\mathcal{S} \cap \mathcal{U} = \Phi^{-1}(0)$, and
\item $\mathrm{D}\Phi_q: \T_q\M \to \mathbb{R}^k$ has constant rank $k$ (in particular, is surjective) for all $q \in \mathcal{S}  \cap \mathcal{U} $.
\end{enumerate}
\end{definition}

By \cite[Prop.~5.38]{lee2012smooth}, if $\Phi$ is a local defining function for $\mathcal{S}$, then for every $p \in \mathcal{S}  \cap \mathcal{U} $,
\begin{equation}\label{eq:tangent-kernel}
\T_p \mathcal{S}  =  \ker\big( \mathrm{D}\Phi_p\big), ~\text{where }~  \mathrm{D}\Phi_p: \T_p \M \to \mathbb{R}^k .
\end{equation}
In other words, the tangent space of the level set $S$ at $p$ coincides with the kernel of the differential of its defining function at $p$. %
For the Stiefel manifold, define $h:\R^{n\times r}\rightarrow \mathbb{S}^r$ with
\(h(X)= X^\top X-I_r\), where    $\mathbb{S}^r$  denotes the set of $r\times r$ symmetric matrices.
Then, one has $\St(n,r)=h^{-1}(0)$ and the Jacobian matrix $\mathrm{D}h(X)$ is full rank for any $X\in\St(n,r)$; see \cite[Section 7.3]{boumal2023introduction}.
Therefore, by Proposition~\ref{prop:submanifold}, $\St(n,r)$ is  a smooth embedded submanifold of dimension $nr-\frac{r(r+1)}{2}$ in $\R^{nr}.$
The tangent space and normal space are given by
\begin{align}
    \T_X\St(n,r)&=\{\eta\in\R^{n\times r}\mid X^\top \eta+\eta^\top X=0\}\\
    \rmN_X\St(n,r)& = \{X\Lambda\mid \Lambda\in\mathbb{S}^r\}.\label{eq:normal_Stiefel}
\end{align}
A common condition to ensure that the intersection of two manifolds forms a manifold is transversality, which is also assumed in \cite{si2023riemannian}. The following is a summary of  Theorem 6.30 and Problem 6-10 in \cite{lee2012smooth}.

\begin{definition}[Transversality]\label{def:tranverse}
	Let $\M_1$ and $\M_2$ be two $C^k$submanifolds embedded in $\R^d$. We say that $\M_1$ and $\M_2$ intersect transversely if at each  $X\in\M_1\cap\M_2,$  $\T_X\M_1 + \T_X\M_2 = \R^d,$ or equivalently, $\rmN_X\M_1\cap \rmN_X\M_2=\{0\}$. In this case, the intersection $\M=\M_1\cap \M_2$ is a $C^k-$ embedded submanifold of $\R^d$ whose codimension is the sum of the codimension of $\M_1$ and $\M_2$. Moreover, there holds $\T_X\M=\T_X\M_1\cap \T_X\M_2.$
\end{definition}

Transversality can be too strong: for manifolds given locally by equality constraints, it corresponds to a LICQ-type condition on the Jacobians of the defining functions. The clean-intersection condition below is weaker and, in the equality-defined setting, corresponds to a constant-rank-type qualification rather than LICQ \cite{drusvyatskiy2015transversality}.
\begin{definition}[Clean intersection]
\label{def:M1andM2_intersect_cleanly}
	Let $\M_1,\M_2\subset \mathbb{R}^n$ be two smooth embedded submanifolds, and fix
	$\bar x\in \M_1\cap \M_2$.
	We assume that $\M_1$ and $\M_2$ intersect \emph{cleanly} at $\bar x$, namely,
	$\M:=\M_1\cap \M_2$ is a smooth embedded submanifold in a neighborhood of $\bar x$ and
	there exists a neighborhood $\U_{\bar x}\subset \M$ of $\bar x$ such that
	\[
	\T_{x}\M=\T_{x}\M_1\cap \T_{x}\M_2,
	\qquad \forall x\in \U_{\bar x}.
	\]
\end{definition}
\paragraph{Basic notation.}
Let $\rmN_X\St(n,r)$ denote the $\frac{r(r+1)}{2}$-dimensional normal space at point $X\in \St(n,r).$ When $r=1$, we denote the sphere by $\mathrm{S}^{n-1}$. For a manifold $\M$, we denote $\T\M=\{(X,\eta)\mid X\in\M  \text{ and } \eta\in\T_X\M\}$ as the tangent bundle and $\rmN\M=\{(X,\eta)\mid X\in\M  \text{ and } \eta\in\rmN_X\M\}$ as the normal bundle. We denote $\mathfrak{X}(\M)$ as the set of smooth vector fields on $\M.$ That is, any $V\in\mathfrak{X}(\M)$ is a smooth map satisfying $V:\M\rightarrow\T\M.$ For two matrices $X,Y\in\R^{n\times r},$ their Hadamard product is denoted by $X\odot Y.$ For a nonempty set $S\subset\mathbb{R}^m$,
$\aff(S)$
is the \emph{affine hull} of $S$, and \(\para(S):=\mathrm{span}\{x-y:\ x,y\in S\}=\aff(S)-x_0\) for any $x_0\in \aff(S)$ is the subspace parallel to $S$. In particular, \(\aff(S)=x_0+\para(S)\) for all $x_0\in \aff(S)$.

\paragraph{Vectorization.}
For matrix manifolds, it is convenient to represent the ambient space as a vector space and use an orthogonal basis. Throughout this paper, we use lowercase letters to denote vector representations and uppercase letters to denote matrix representations. For example, a matrix $X \in \mathbb{R}^{n \times r}$ is represented in its vectorized form by stacking all its column vectors into \(\vet(X)=x \in \mathbb{R}^{nr}\).
We fix a linear isomorphism $\overline{\vet}:\mathbb S^r\to\R^{\frac{r(r+1)}{2}}$ that stacks
the upper-triangular entries of a symmetric matrix into a vector:
\begin{equation}\label{def:vet_symm}
\overline{\vet}(\Lambda) =\lambda \in \R^{\frac{r(r+1)}{2}}, \quad \forall \Lambda\in\mathbb{S}^r.
\end{equation}

\paragraph{Metric and norms.}
We consider the Riemannian metric on $\M$ that is induced from the Euclidean inner product; i.e., for any $\xi,\eta \in \T_X\M$, we have $\inp{\xi}{\eta} = \Tr(\xi^\top\eta)$. The Frobenius norm of a matrix $X$ is denoted by $\normfro{X}$, the operator norm of $X$ is denoted by $\|X\|_2$, and the $\ell_2$ norm of a vector $x$ is denoted by $\|x\|.$ We use $B_x\in\R^{nr\times \frac{r(r+1)}{2}}$ to denote a basis matrix of the normal space $\rmN_X\St(n,r)$, where the lowercase vector representation is given by $x=\vet(X)$. The Euclidean ball centered at $x$ with radius $\delta$ in $\R^d$ is denoted by $\mathbb{B}(x,\delta)=\{y\in\R^d: \|y-x\|\leq \delta \}.$ 
For a point $a$ and a set $B$, we define the distance by \(\mathrm{dist}(a,B) = \inf_{b\in B}\|a-b\|\).

\begin{definition}[smooth function on manifold and differential]\label{def:smooth_differential}
  Let $\M$ and $\M'$ be  embedded submanifolds in $\E$ and $\E'$, respectively. Define a  map $f:\M\rightarrow\M'$.
  \begin{enumerate}
      \item  $f$ is smooth at $X\in\M$ if there exists a function $\bar{f}:\mathcal{U}\rightarrow\E'$ which is smooth on a neighborhood $\mathcal{U}$ of $X$ in $\E$ and such that $f$ and $\bar{f}$ coincide on $\M\cap \mathcal{U},$ that is, $f(y)=\bar{f}(y)$ for any $y\in \M\cap \mathcal{U}$. We call $\bar{f}$ a smooth extension of $f$ around $X.$ The map $f$ is smooth if it is smooth at all $X\in\M.$
      \item The differential of $f$ at point $X\in\M$ is the linear map $\mathrm{D}f(X):\T_X\M\rightarrow \T_{f(X)}\M'$ defined by
      \(\mathrm{D}f(X)[\eta] = \mathrm{D}\bar{f}(X)[\eta]\) for all $\eta\in\T_X\M$,
      where $\mathrm{D}\bar{f}$ is the classical differential of the smooth extension function $\bar{f}$.
  \end{enumerate}
\end{definition}
In our problem \eqref{prob:nonsmooth_stiefel}, $f(X)$ is smooth in $\R^{n\times r}$ and hence smooth on $\St(n,r).$ We would directly use the notation $\mathrm{D}f$ to represent the differential on the manifold. The  Riemannian gradient of a smooth function $f:\M\rightarrow\R$ at $X$, denoted by $\grad f(X)$, is the smooth vector field satisfying \(\mathrm{D} f(X)[\eta] = \langle \eta, \grad f(X)\rangle\) for all $\eta\in\T_X\M$.

By our choice of  Riemannian metric, we have
\(\grad f(X)=\Proj_{\T_X\M}\nabla f(X)\),
where    $\Proj_{\T_X\M}$ denotes the orthogonal projection onto the tangent space $\T_X\M$ and $\nabla f(X)$ is the Euclidean gradient.
The Riemannian Hessian of $f$ at $X$ is the linear map $\Hess f(X):\T_X\M\rightarrow\T_X\M$ defined as $\Hess f(X)[\eta]=\nabla_\eta \grad f$.  Here $\nabla$ denotes the Levi-Civita connection induced by the Euclidean metric and is defined by
\begin{equation}\label{def:Levi-Civita}
    \nabla_\eta V = \Proj_{\T_X\M}(\mathrm{D}\bar{V}(X)[\eta]),
\end{equation}
  where $V\in \mathfrak{X}(\M)$ is a smooth vector field on $\M$, and $\bar{V}$ is a smooth extension of $V$.
We also know that  the Riemannian Hessian is given by \({\Hess f(X)}[\eta]=\Proj_{\T_X\M}(\mathrm{D} \grad f(X)[\eta])\).

In Riemannian optimization, a descent direction is defined in the tangent space, followed by a mapping back to the manifold along a curve. While a geodesic curve serves as the natural generalization of a straight line in Euclidean space, it is often computationally expensive to compute. To address this, the notion of a retraction is introduced, providing a structured and efficient approximation of the Riemannian exponential map.
To develop second-order Riemannian methods, we further introduce the concept of a second-order retraction. We first need to  present the definition of acceleration of a smooth curve on a manifold.
\begin{definition}
	Let $\M$ be a Riemannian submanifold embedded in $\R^n$ and  $c \colon I \to \mathcal{M}$ be a smooth curve, where $I$ is an open interval in $\R$. Its velocity is the vector field $c' \in \mathfrak{X}(c)$. The acceleration of $c$ is the smooth vector field $c'' \in \mathfrak{X}(c)$ defined by \(c'' = \frac{\mathrm{D}}{\mathrm{d}t} c'\),
	where $\frac{\mathrm{D}}{\mathrm{d}t}: \mathfrak{X}(c)\rightarrow\mathfrak{X}(c)$ is a covariant derivative on $\M$.
	We also call $c''$ the intrinsic acceleration of $c$.
\end{definition}
When $\mathcal{M}$ is embedded in a linear space $\E$, a smooth curve $c$ on $\mathcal{M}$ is also smooth in $\E$.  We write the classical (extrinsic) acceleration of $c$ in the embedding space as \(\ddot{c} = \frac{\mathrm{d}^2}{\mathrm{d}t^2} c\).
When $ \frac{\mathrm{D}}{\mathrm{d}t}$ is the covariant derivative induced by the Riemannian connection \eqref{def:Levi-Civita}, one has (see \cite[(5.23)]{boumal2023introduction}):
\begin{equation}\label{eq:intrinsic_acc_and_ex_acc}
	c''(t)= \Proj_{\T_{c(t)}\M}(\ddot{c}(t) ).
\end{equation}
In that spirit, we use notations $\dot{c}$ and $c'$ interchangeably for velocity since the two notions coincide.

\begin{definition}[Retraction]\cite{Absil2009}
	\begin{enumerate}
		\item A retraction \textit{on a manifold} $\mathcal{M}$ is a smooth map
		\[
		\mathrm{Retr} \colon \mathcal U \to \mathcal{M} \colon (X, \eta) \mapsto \mathrm{Retr}_X(\eta),
		\]
		where $\mathcal U$ is an open neighborhood of the zero section in $\T\mathcal M$,
		\textit{such that each curve $c(t):=\mathrm{Retr}_X(t\eta)$ satisfies
		$c(0)=X$ and $c'(0)=\eta$ whenever $(X,t\eta)\in\mathcal U$ for all sufficiently small $t$.}

		\item A second-order retraction $\mathrm{Retr}$ on a Riemannian manifold $\mathcal{M}$
		is a retraction such that, for all $(X,t\eta)\in\mathcal U$, the curve $c(t)=\mathrm{Retr}_X(t\eta)$,
		defined for all sufficiently small $t$, has zero intrinsic acceleration at $t=0$,
		that is, $c''(0)=0$.
	\end{enumerate}
\end{definition}

\subsection{Nonsmooth function}
Next, we present some results for nonsmooth functions on embedded submanifold in $\E.$ We say that a function $F$ is locally Lipschitz continuous if for any $X\in\M$ it is Lipschitz continuous in a neighborhood of $X$ in $\E$. Conversely, if $F$ is locally Lipschitz continuous in the Euclidean space $\E$, then it is also locally Lipschitz continuous when restricted to the embedded submanifold $\M$ of $\E$.

\begin{definition}[{generalized Clarke subdifferential \cite{hosseini2011generalized}}]
For a locally Lipschitz function $F$ on $\M$, the Riemannian generalized directional derivative of $F$ at $X\in\M$ in the direction $\eta$ is defined by
\[ 
F^{\circ}(X,\eta) =\limsup\limits_{Y\rightarrow X,t\downarrow 0}\frac{F\circ \phi^{-1}(\phi(Y)+t\mathrm{D}\phi(X )[\eta])-F\circ \phi^{-1}(\phi(Y))}{t},
\] 
where $(\phi,\mathcal{U})$ is a coordinate chart at $X$.
The generalized gradient or the Clarke subdifferential of $F$ at $X\in\M$, denoted by $\partial_C F$, is given by
\[ 
\partial_C F(X)=\{\xi\in \T_X\M :\inp{\xi}{\eta}\leq F^{\circ}(X,\eta) \ \forall \eta\in \T_X\M \}.
\] 
\end{definition}

The following definition generalizes the subdifferential (Clarke) regularity \cite{Rockafellar2009} to the embedded submanifold of $\E.$
\begin{definition}[\cite{Yang-manifold-optimality-2014}]
A function $F$ is said to be regular at $X\in\M$ along $\T_X\M$ if
        \begin{enumerate}
                \item for all $\eta\in \T_X\M$, $F'(X;\eta):=\lim_{t\downarrow 0} \frac{F(X+t\eta)-F(X)}{t}$ exists, and
                \item for all $\eta\in \T_X\M$, $F'(X;\eta) = F^\circ (X;\eta)$.
        \end{enumerate}
\end{definition}
Using Definition \ref{def:smooth_differential},  $F'(X;\eta)=\mathrm{D}F(X)[\eta]$ if $F$ is smooth. The convex function in Euclidean space is subdifferentially regular, and it is also regular on $\M.$
  By \cite[Theorem 5.1]{Yang-manifold-optimality-2014}, for a regular function $F$, we  have $\partial_C F(X)=\Proj_{\T_X\M}(\partial F(X)),$ where $\partial F(X)$ is the standard Clarke subdifferential in Euclidean space. Since $g$ is convex and $f$ is smooth, the function $F=f+g$ in \eqref{prob:nonsmooth_stiefel} is regular according to Lemma 5.1 in \cite{Yang-manifold-optimality-2014}. Therefore, we have $\partial_C F(X)=\Proj_{\T_X\M}(\nabla f(X)+\partial g(X))=\grad f(X)+ \Proj_{\T_X\M}(\partial g(X))$. By Theorem 4.1 in \cite{Yang-manifold-optimality-2014}, the first-order necessary condition of \eqref{prob:nonsmooth_stiefel} is given by $0\in \grad f(X)+ \Proj_{\T_X\M}(\partial g(X))$.   By the normal space structure \eqref{eq:normal_Stiefel},    this is equivalent to the KKT condition in the embedded Euclidean space.
  We now measure first-order optimality of \eqref{prob:nonsmooth_stiefel} at \emph{arbitrary}
$X\in\mathbb{R}^{n\times r}$ via a KKT residual.

\begin{definition}[KKT residual and $\varepsilon$-KKT point for \texorpdfstring{\eqref{prob:nonsmooth_stiefel}}{(P0)}]
\label{def:KKT-residual-P0}
For any $X\in\mathbb{R}^{n\times r}$, the KKT residual of problem
\eqref{prob:nonsmooth_stiefel} at $X$ is defined as
\begin{equation}\label{eq:R0-def}
  \mathcal{R}_0(X)
  :=
  \inf_{\substack{G\in\partial g(X) \\[1pt] \Lambda\in\mathbb{S}^r}}
  \Big(
    \underbrace{\normfro{\nabla f(X) +   G + X\Lambda}}_{\text{first-order (stationarity) residual}}
    +
    \underbrace{\big\|X^\top X - I_r\big\|_{\mathrm{F}}}_{\text{feasibility residual}}
  \Big).
\end{equation}
Given $\varepsilon> 0$, we say that $X$ is an \emph{$\varepsilon$-KKT point}
(or an \emph{$\varepsilon$-first-order point}) of \eqref{prob:nonsmooth_stiefel} if
\begin{equation}\label{eq:eps-KKT}
  \mathcal{R}_0(X) \;\le\; \varepsilon.
\end{equation}
In particular, $X$ is a first-order stationary point of \eqref{prob:nonsmooth_stiefel} if and only if
$\mathcal{R}_0(X)=0$.
\end{definition}


\subsection{Partial smoothness}\label{sec:partial_smoothness}
For a nonsmooth (possibly nonconvex) function $g:\E\rightarrow\R$, partial smoothness captures the common situation in which $g$ is smooth along an active manifold and nonsmooth across it. For example, the $\ell_1$ norm is smooth relative to a fixed sparsity pattern. We present the formal definition as follows.
\begin{definition}\label{def:partial_smoothness}
    A function $g:\E\rightarrow\R$ is $\mathcal{C}^k$ partly smooth at a point $\bar{x}$ relative to a set $\M_g$ containing $\bar{x}$ if $\M_g$ is a $\mathcal{C}^k$ manifold around  $\bar{x}$ and:
    \begin{enumerate}[label={(\arabic*)}]
        \item (smoothness) the restriction of $g$ to $\M_g$ is a $\mathcal{C}^k$ function near $\bar{x}$;
        \item (regularity) The function $g$ is (Clarke) regular at all points $x \in\M_g$ near $\bar{x}$, with $\partial g(x)\footnote{We use $\partial_{\ell} g$ to denote the limiting subdifferential \cite{Rockafellar2009}, which reduces to the convex subdifferential when $g$ is convex. Therefore, we use the same notation for simplicity.
}\neq \emptyset$;
        \item (sharpness) the affine space  $\aff(\partial g(\bar{x}))$ is a translate of $\rmN_{\bar{x}}\M_g$;
        \item (sub-continuity) the set-valued mapping $\partial g$ restricted to $\M_g$ is continuous at $\bar{x}$.
    \end{enumerate}
\end{definition}
The continuity in Definition~\ref{def:partial_smoothness}(4)
is understood in the usual set-valued sense of outer and inner
semicontinuity; see Appendix~\ref{app:prelim-subcontinuity}.

It follows from \cite[Proposition~2.10]{lewis2002active} that if $g$ is partly smooth at
$\bar x$ relative to $\M_g$, then for all $x\in \M_g$ sufficiently close to $\bar x$,
\begin{equation}\label{eq:normal_space_eq_to_parallel_g}
	\para\bigl(\partial g(x)\bigr)=\rmN_x \M_g .
\end{equation}

For the $\ell_1$ norm function $g(x)=\|x\|_1$ and a point $\bar x\in\R^d$, $g$ is partly smooth at $\bar x$ relative to the coordinate manifold
\[
\M_I:= \{ x\in\R^d\mid  x_i=0~\text{ for } ~i\in I\},
\]
where $I$ is the active set defined by  $I=\{i=1,\ldots,d\mid \bar x_i=0\}.$ Since \( g(x) = \|x\|_1 \) is convex, its   subdifferential   is regular with
\[
\partial g(x) = J_1 \times J_2 \times \cdots \times  J_d,
\]
{\color{blue}where, with \(\mathrm{sign}\) denoting the usual single-valued scalar sign
function satisfying \(\mathrm{sign}(0)=0\),
\[
J_i =
\begin{cases}
[-1,1], & x_i=0,\\
\{\mathrm{sign}(x_i)\}, & x_i\neq0.
\end{cases}
\]
Hence, we can verify all conditions in Definition \ref{def:partial_smoothness}.
Similarly, the row-sparsity regularizer
\(g(X)=\|X\|_{2,1}:=\sum_{i=1}^n\|X_{i,:}\|_2\) is partly smooth at
\(\bar X\in\R^{n\times r}\) relative to the row-support manifold
\[
\M_I^{\rm row}:=\{X\in\R^{n\times r}:X_{i,:}=0\ \text{for }i\in I\},
\qquad I:=\{i:\bar X_{i,:}=0\}.
\]
Indeed, \(G\in\partial g(X)\) if and only if, for each \(i\),
\[
G_{i,:}=
\begin{cases}
X_{i,:}/\|X_{i,:}\|_2, & X_{i,:}\neq0,\\
u_i,\ \|u_i\|_2\leq1, & X_{i,:}=0.
\end{cases}
\]
This formula verifies the smoothness, regularity, sharpness, and
subdifferential-continuity conditions in Definition~\ref{def:partial_smoothness}.}
When we know the structure of the manifold $\M_g$ for the local minimizer of a problem, we can design Riemannian Newton-type algorithms to achieve fast convergence locally and obtain a more accurate solution. Due to the Stiefel constraint, problem \eqref{prob:nonsmooth_stiefel} is more complex than the Euclidean setting. We will exploit the structure of $\M_g\cap \St(n,r)$ and develop a new Riemannian optimization algorithm on it.

%% file: shared/section_identification_ManPG.tex
\section{Identification of ManPG}\label{sec:id_manpg}

In this section, we study the manifold identification property of ManPG for general convex functions $g$. In Euclidean settings, manifold identification is well understood through the framework of partial smoothness and transversality \cite{lewis2002active}. In the manifold-constrained setting, however, the situation is more delicate. As we will show in Section~\ref{sec:intersection_property}, even though the intersection $\mathcal{M}_g\cap \St(n,r)$ may remain a smooth submanifold under mild conditions, it can fail to satisfy the transversality condition in the sparse Stiefel setting for the $\ell_1$ regularizer. As a result, the associated proximal mapping may fail to identify the active manifold stably.

To overcome this difficulty, we introduce a perturbed family of Stiefel constraints and analyze the corresponding perturbed problems~\eqref{prob:P-Delta}. Our goal is to obtain the regular intersection geometry needed for identification. We show that, for generic perturbations, the active manifold and the perturbed Stiefel manifold intersect transversely, which restores the manifold identification property of ManPG. Related perturbation-based regularization ideas have also appeared in other contexts, for example in low-rank SDP problems \cite{tang2024feasible}.

\subsection{Perturbed Stiefel problems and geometry}

\paragraph{Perturbed Stiefel constraints.}
We begin with the perturbed Stiefel family defined in \eqref{def:St-Delta}. The case $\Delta=0$ recovers the standard Stiefel manifold $\St(n,r)$. Throughout, we restrict attention to perturbations with small Frobenius norm, say $\|\Delta\|_{\mathrm F}<1$, so that $I_r+\Delta$ remains positive definite and close to $I_r$. Under this condition, $\St_\Delta(n,r)$ is an embedded submanifold of $\R^{n\times r}$. The perturbation serves two purposes. First, it regularizes the local intersection geometry. Second, the original Stiefel manifold may have empty intersection with the relevant active manifold. A further perturbed Stiefel manifold can avoid this difficulty. This point will be used later in Section~\ref{sec:alg_on_intersection}.

\paragraph{Manifold structure.}
The manifold structure can be derived in the same way as in \cite{edelman1998geometry,Absil2009}. Define the constraint mapping
\(\Phi_\Delta : \mathbb{R}^{n\times r} \to \mathbb{S}^r\) by
\[\Phi_\Delta(X) := X^\top X - (I_r + \Delta)\].
Then $\St_\Delta(n,r)$ can be written as the level set
\(\St_\Delta(n,r) = \Phi_\Delta^{-1}(0)\).
The derivative of $\Phi_\Delta$ at $X\in\mathbb{R}^{n\times r}$ in the
direction $H \in \mathbb{R}^{n\times r}$ is given by
\(\mathrm{D}\Phi_\Delta(X)[H] = X^\top H + H^\top X\).
This linear map does not depend on $\Delta$, and coincides with the
derivative of the standard Stiefel constraint.

Since $I_r+\Delta \succ 0$, the rank of $\mathrm{D}\Phi_\Delta(X)$ is equal to
$\frac{r(r+1)}{2}$ for every $X \in \St_\Delta(n,r)$; in particular,
$\mathrm{D}\Phi_\Delta(X)$ is surjective as a map from $\mathbb{R}^{n\times r}$ to
$\mathbb{S}^r$. Consequently, $\St_\Delta(n,r)$ is an embedded submanifold
of $\mathbb{R}^{n\times r}$ of dimension
\(\dim \St_\Delta(n,r) = nr - \frac{r(r+1)}{2}\).

\paragraph{Tangent and normal spaces.}
The tangent and normal spaces of $\St_\Delta(n,r)$ at $X$ admit the same
expressions as for the standard Stiefel manifold. Specifically,
\begin{equation}\label{eq:T-N-St-Delta}
	\begin{aligned}
		\T_X \St_\Delta(n,r)
		&= \bigl\{ H \in \mathbb{R}^{n\times r} : X^\top H + H^\top X = 0 \bigr\},\\[2mm]
		\rmN_X \St_\Delta(n,r)
		&= \bigl\{ X \Lambda : \Lambda \in \mathbb{S}^r \bigr\}.
	\end{aligned}
\end{equation}
Unless otherwise stated, we equip $\St_\Delta(n,r)$ with the Riemannian
metric induced by the ambient Frobenius inner product.
With this choice, the Riemannian gradient at $X \in \St_\Delta(n,r)$ is
given by the orthogonal projection of the Euclidean gradient onto the
tangent space $\T_X \St_\Delta(n,r)$.  Details are deferred to Appendix~\ref{sec:append:normal}.

The KKT conditions for \eqref{prob:P-Delta} at a feasible point
$X\in\mathbb{R}^{n\times r}$ are:
there exist $G_\Delta\in\partial g(X)$ and
$\Lambda_\Delta\in\mathbb{S}^r$ such that
\begin{equation}\label{eq:KKT-P-Delta}
	\nabla f(X) + G_\Delta + X\Lambda_\Delta = 0,
	\qquad
	X^\top X = I_r + \Delta.
\end{equation}

We now show that any KKT point of the perturbed problem \eqref{prob:P-Delta}
is an $\varepsilon$-KKT point of the original problem \eqref{prob:nonsmooth_stiefel}, with
$\varepsilon$ proportional to the size of the perturbation~$\Delta$.

\begin{proposition}
	\label{prop:P-Delta-eps-KKT}
	Let $\Delta\in\mathbb{S}^r$ be such that $I_r+\Delta\succ 0$ and let
	$X_\Delta\in\mathbb{R}^{n\times r}$ be a KKT point of the perturbed
	problem~\eqref{prob:P-Delta}. Then
	\begin{equation}\label{eq:P-Delta-eps-KKT}
		\mathcal{R}_0(X_\Delta) \;\le\; \|\Delta\|_{\mathrm{F}}.
	\end{equation}
\end{proposition}

\begin{proof}
	Since $X_\Delta$ is a KKT point of \eqref{prob:P-Delta}, there exist
	$G_\Delta\in\partial g(X_\Delta)$ and $\Lambda_\Delta\in\mathbb{S}^r$
	such that
	\[
	\nabla f(X_\Delta) +  G_\Delta + X_\Delta\Lambda_\Delta = 0,
	\qquad
	X_\Delta^\top X_\Delta = I_r + \Delta.
	\]
	By the definition of $\mathcal{R}_0$ in~\eqref{eq:R0-def}, we can bound
	$\mathcal{R}_0(X_\Delta)$ from above by evaluating the right-hand side at
	this specific choice $(G,\Lambda)=(G_\Delta,\Lambda_\Delta)$:
	\[
	\mathcal{R}_0(X_\Delta)
	\;\le\;
	\big\|\nabla f(X_\Delta) +  G_\Delta + X_\Delta\Lambda_\Delta\big\|_{\mathrm{F}}
	+
	\big\|X_\Delta^\top X_\Delta - I_r\big\|_{\mathrm{F}}=\|\Delta\|_{\mathrm{F}}.
	\]
\end{proof}

\subsection{\texorpdfstring{Active manifold}{Active manifold}}

For the ManPG algorithm \cite{chen2020proximal}, we rewrite the tangent subproblem~\eqref{eq:manpg_subprob_intro} as a proximal-gradient mapping on the affine tangent space:
\begin{manualtaggedequation}
\begin{equation}\label{eq:tangent_problem}
	\mathcal{T}_{t}(X):=\argmin_{Y\in X+\T_X\St_{\Delta}(n,r)} \rho_t(X,Y),\tag{Tprox}
\end{equation}
\end{manualtaggedequation}
where
\begin{equation}\label{eq:rho_t}
	\rho_t(X,Y):= \frac{1}{2t}\|Y-(X-t\nabla f(X))\|_{\mathrm{F}}^2 +g(Y).
\end{equation}
To relate the   subproblem~\eqref{eq:tangent_problem} to manifold identification in composite optimization, we introduce a local coordinate representation of the affine slice $X+\T_X\St_\Delta(n,r)$.

Let $x:=\vet(X)$ and $d:=\dim\T_X\St_\Delta(n,r)$, and let
$E_x\in\R^{nr\times d}$ be a matrix whose columns form an orthonormal basis of
$\vet(\T_X\St_\Delta(n,r))$.
Then every $Y\in X+\T_X\St_\Delta(n,r)$ can be written as
\(\vet(Y)=\vet(X)+E_x z\) for some $z\in\R^d$.
Identifying $x+E_xz$ with the matrix whose vectorization is $x+E_xz$ in the last term,
the ManPG subproblem~\eqref{eq:tangent_problem} is equivalently
\begin{equation}\label{eq:tprox-z}
	z^\sharp(X)\in\argmin_{z\in\R^d}\;
	\frac{1}{2t}\bigl\|E_x z+t\,\vet(\nabla f(X))\bigr\|_2^2
	+\;g\bigl(x+E_x z\bigr).
\end{equation}
This coordinate viewpoint links our tangent-slice subproblem to the composite framework
$g(c_x(\cdot))$ studied in \cite{bareilles2023harnessing} and related identification theory.
The key difference is that our original problem is constrained on a  Stiefel manifold,
and the composite structure above arises in the ManPG subproblem.
Nevertheless, the role of transversality condition in \cite{bareilles2023harnessing} remains highly relevant for
understanding identification and stability of the ManPG subproblem; we will show that the active manifold is determined by the solution $\mathcal T_t(X)$.

For example, fix a reference point $\bar X\in\St_{\Delta}(n,r)$ and let
$\bar Y:=\mathcal{T}_t(\bar X)$.
Assume that $g(Y)=\lambda\|Y\|_1$ and that $\bar Y$ has $s$ nonzero entries.
Define the support set by
\[\mathcal S:=\mathrm{supp}(\bar Y)\subset [n]\times[r] \quad \text{ and}\quad 
 \mathcal S^c:=\big([n]\times[r]\big)\setminus \mathcal S.\]
For any matrix $U$, define the coordinate projector
\[
(\Proj_{\mathcal S^c}U)_{ij}=
\begin{cases}
U_{ij},&(i,j)\in \mathcal S^c,\\[2pt]
0,&(i,j)\in \mathcal S.
\end{cases}
\]
We define the active manifold (with fixed zero pattern)
\begin{equation}\label{eq:manifold_2}
	\mathcal{M}_g(\bar Y)
	:=\Big\{\,Y\in\mathbb{R}^{n\times r}\ :\ \Proj_{\mathcal S^c}Y=0\,\Big\},
\end{equation}
which is a linear subspace of $\mathbb{R}^{n\times r}$. It is an embedded smooth submanifold of $\mathbb{R}^{n\times r}$ of dimension $|\mathcal S|=s$, with tangent space
\(\T_Y\mathcal{M}_g(\bar Y)=\{Z \in\mathbb{R}^{n\times r}:\ Z_{\mathcal S^c}=0\}\).
Moreover, since $\bar Y_{ij}\neq 0$ for all $(i,j)\in\mathcal S$, there exists a neighborhood of $\bar Y$ in $\mathcal{M}_g(\bar Y)$ on which the sign pattern is fixed. Hence the restriction of $g(Y)=\lambda\|Y\|_1$ to $\mathcal{M}_g(\bar Y)$ is smooth near $\bar Y$, with constant gradient
\(\nabla g(Y)=\lambda\,\mathrm{sign}(\bar Y)\)
and zero Hessian along the manifold directions.

A main difficulty in sparse models is that transversality may fail because of the interaction
between the Stiefel geometry and the subdifferential structure of the $\ell_1$ regularizer.
This leads us to introduce generic perturbations and study the perturbed problem~\eqref{prob:P-Delta}.
To state the resulting genericity results uniformly over all active manifolds that may arise,
we impose a finite-family assumption on active manifolds
(Assumption~\ref{assump:finite-active-manifolds}). This allows the corresponding exceptional parameter sets to be handled by a finite union.

The following assumption can be viewed as a finite, locally stable
active-manifold version of standard stratification frameworks
\cite{whitney1965tangents,bolte2007clarke}: near a point on a stratum, the
same stratum continues to describe the active manifold. To avoid introducing
additional stratification terminology, we state this property directly. This assumption is satisfied by standard partly smooth regularizers, including
the \(\ell_1\) norm with fixed support manifolds, the \(\ell_{2,1}\) norm with
fixed row-support manifolds.

\begin{assumption}[Finite family of active manifolds]
		\label{assump:finite-active-manifolds}
	Assume that there exist finitely many smooth maps
	$h_i:\R^{n\times r}\to\R^{p_i}$, $i=1,\ldots,N$, such that
	\[
	\M_i=\{X:h_i(X)=0\},
	\qquad
	\rank Dh_i(X)=p_i,\quad X\in\M_i,
	\]
	and, for every $\bar X\in E$, $g$ is $C^2$-partly smooth at
	$\bar X$ relative to some $\M_i$.
	Moreover, the selected active manifold is locally stable: whenever
	$g$ is $C^2$-partly smooth at $\bar X$ relative to $\M_i$, there
	exists a neighborhood $U$ of $\bar X$ such that
	\[
	Y\in U\cap\M_i
	\quad\Longrightarrow\quad
	\M_g(Y)=\M_i .
	\]
\end{assumption}
For the sparse regularizers used in this paper, the relevant active manifolds
are linear or row-support manifolds, and hence this smoothness requirement is
automatically satisfied.

\subsection{Structural assumptions and lemmas}

In this section, we introduce the structural assumptions needed for the identification analysis and collect the main auxiliary lemmas. We begin with the optimality condition for the tangent subproblem~\eqref{eq:tangent_problem}.
\begin{lemma}\label{lem:tangt_prox_optcond}\cite[Lem.~5.3]{chen2020proximal}
	Given $\bar{X}\in\St_{\Delta}(n,r)$ and $\bar{Y} = \mathcal{T}_{t}(\bar{X})$, we have
	\[
	0\in \frac{1}{t}(\bar{Y}-\bar{X})+\grad f(\bar{X})+\Proj_{\T_{\bar{X}}\St_{\Delta}(n,r)}\,\partial g(\bar{Y}).
	\]
	Moreover, $\bar X$ is a first-order critical point of problem~\eqref{prob:P-Delta} if and only if
	\(
	\bar Y=\bar X.
	\)
\end{lemma}
\begin{proof}
The sufficiency part follows from \cite[Lem.~5.3]{chen2020proximal}.

For necessity, suppose that $\bar X$ is a first-order critical point of \eqref{prob:P-Delta}. Then
\[
0\in \grad f(\bar X)+\Proj_{\T_{\bar X}\St_\Delta(n,r)}\,\partial g(\bar X),
\]
so $Y=\bar X$ satisfies the optimality condition of the tangent proximal subproblem \eqref{eq:tangent_problem}. Since the objective in \eqref{eq:tangent_problem} is strongly convex on the affine space $\bar X+\T_{\bar X}\St_\Delta(n,r)$, the minimizer is unique. Hence
	\(
	\bar Y=\bar X.
		\)
\end{proof}
We next recall the nondegeneracy (ND) condition  at a first-order critical point.
In the Euclidean setting, the ND condition used in manifold identification takes the form
	\[
	0\in \nabla f(x^*)+\mathrm{ri}\,\partial g(x^*),
	\]
	which is standard in the literature on partial smoothness and identifiable manifolds; see, for example, \cite{burke1988identification,wright1993identifiable,drusvyatskiy2014optimality,lewis2002active,liang2017activity,bareilles2023newton}. On the Stiefel manifold, since $g$ is regular and $\partial g$ is convex, it is exactly the   projection of the ambient subdifferential onto tangent space:
\[
0\in \grad f(X^*)+\mathrm{ri}\big(\Proj_{\T_{X^*}\St_{\Delta}(n,r)}\partial g(X^*)\big).
\]
\begin{definition}\label{def:r-structured critical}
	A point $X^*\in\St_{\Delta}(n,r)$ is called a nondegenerate critical point of~\eqref{prob:P-Delta} if
	\begin{manualtaggedequation}
	\begin{equation}\label{eq:ND}\tag{ND}
		0\in \grad f(X^*) + \mathrm{ri}\!\left(\Proj_{\T_{X^*}\St_{\Delta}(n,r)}\partial g(X^*)\right),
	\end{equation}
	\end{manualtaggedequation}
	where $\mathrm{ri}$ denotes relative interior.
\end{definition}

Assume the nondegeneracy condition at $X^*$. We now explain the idea behind the manifold identification result for the tangent subproblem~\eqref{eq:tangent_problem}. Rather than working directly with~\eqref{eq:tangent_problem}, we first restrict the subproblem to the active manifold determined by $\bar Y$. Specifically, we define
\begin{equation}\label{eq:tangent_sub_supp}
	Y_t^\sharp(X):=\arg\min_{Y\in \T_{\cap}(X)}\rho_t(X,Y),
\end{equation}
where
\begin{equation}\label{def:slice_tangent_inter}
	\T_{\cap}(X):= \mathcal{M}_g(\bar{Y})\cap\big(X+\T_X\St_{\Delta}(n,r)\big).
\end{equation}
Under the   transversality condition used below, we will show that $Y_t^\sharp(X)$ depends $C^1$-smoothly on $X$ near any strong critical point $X^*$. We then prove that $Y_t^\sharp(X)$ also solves the original subproblem~\eqref{eq:tangent_problem}. Consequently, the tangent proximal mapping $\mathcal{T}_t(X)$ identifies the same support pattern as $Y^*=X^*$ in a neighborhood of $X^*$. In this way, our analysis extends the Euclidean identification framework of \cite{lewis2002active,bareilles2023newton,bareilles2023harnessing} to the Stiefel manifold constrained problem~\eqref{prob:P-Delta} through the transversality condition. Generic perturbations provide one sufficient way to obtain this condition, while the unperturbed case is also covered whenever the same condition holds.

\subsection{Local geometry under generic perturbations}

Throughout this subsection, we fix a reference point $\bar X\in\mathbb{R}^{n\times r}$ and write $\bar Y=\mathcal{T}_t(\bar X)$. We also define
\begin{equation}\label{def:pert_Delta}
    \mathscr P:=\{\Delta\in\mathbb S^r:\ \diag(\Delta)=0,\ I_r+\Delta\succ 0\}.
\end{equation}

We assume that $\mathcal{M}_g(\bar Y)$ is an embedded smooth manifold and that $\St_\Delta(n,r)$ is well defined for some perturbation $\Delta\in\mathscr P$.
If $\bar X$ is a first-order critical point of \eqref{prob:P-Delta}, then $\bar Y=\bar X$, and hence $\mathcal{M}_g(\bar Y)\cap \St_\Delta(n,r)\neq\emptyset$. In the sequel, we state the results in a form that applies whenever this intersection contains a point near $\bar X$. The case of empty intersection will be addressed later.

\paragraph{Generic transversality.}
We next establish a generic transversality result for perturbations
\(
\Delta\in\mathscr P.
\)
Since such perturbations affect only the off-diagonal orthogonality constraints, the diagonal constraints remain fixed and give rise to the oblique manifold
\[
\Ob(n,r):=\{X\in\R^{n\times r}:\diag(X^\top X)=\mathbf 1\}.
\]
The following result shows that, under a transversality condition between the active manifolds and $\Ob(n,r)$, transversality with the perturbed Stiefel manifold holds for generic $\Delta$. This genericity statement follows from the Morse-Sard theorem \cite{sard1942measure} through \cite[Theorem~4]{tang2024feasible}.

\begin{proposition}
	\label{lem:generic-transv}
	Suppose Assumption~\ref{assump:finite-active-manifolds} holds.   Moreover, we assume that for every
	$i\in\{1,\dots,N\}$,
	\begin{equation}\label{eq:oblique-transv-condition}
		\rmN_X \mathcal M_i \cap \rmN_X \Ob(n,r)=\{0\},
		\qquad \forall\,X\in \mathcal M_i\cap \Ob(n,r).
	\end{equation}
	Then there exists a subset
	$\mathscr P_{\rm gen}\subseteq\mathscr P$ such that
	$\mathscr P\setminus\mathscr P_{\rm gen}$ has Lebesgue measure zero in the linear subspace
	\(
	\mathbb S_0^r:=\{\Delta\in\mathbb S^r:\ \diag(\Delta)=0\},
	\)
	and for every $\Delta\in\mathscr P_{\rm gen}$ and every $i\in\{1,\dots,N\}$,
	\begin{equation}\label{eq:transv-Mg-StDelta}
		\rmN_X \mathcal M_i \cap \rmN_X \St_\Delta(n,r)=\{0\},
		\qquad \forall\,X\in \mathcal M_i\cap \St_\Delta(n,r).
	\end{equation}
	In particular, for every \(i\) and every
	\(\Delta\in\mathscr P_{\rm gen}\), the intersection
	\(\mathcal M_i\cap\St_\Delta(n,r)\) is either empty, or else necessarily
	\(\dim(\mathcal M_i)\ge r(r+1)/2\) and it is a smooth embedded submanifold of dimension
	\(
	\dim(\mathcal M_i)-\frac{r(r+1)}{2}.
	\)
\end{proposition}

\begin{proof}[Proof sketch]
Separate the Stiefel constraints into the fixed diagonal
part and the off-diagonal parameter \(d=\overline\vet_{\rm off}(\Delta)\).
The oblique transversality condition implies LICQ for the system defining
\(\mathcal M_i\cap\Ob(n,r)\). Applying the parametric transversality result
	of \cite[Theorem~4]{tang2024feasible} 
	to the off-diagonal constraint gives,
for each active manifold \(\mathcal M_i\), a measure-zero exceptional set of
parameters. Taking the finite union over \(i=1,\ldots,N\) yields
\(\mathscr P_{\rm gen}\). For \(\Delta\in\mathscr P_{\rm gen}\), the normal
spaces intersect trivially, and the final dimension statement follows from
the transverse-intersection theorem. The full proof is given in
Appendix~\ref{app:proof-generic-transv}.
\end{proof}

	\begin{remark}
    \begin{enumerate}[label={(\arabic*)}]
        \item Condition \eqref{eq:oblique-transv-condition} is mild. In the sparse setting considered later in Section~\ref{sec:intersection_property}, it is automatically satisfied; see Proposition~\ref{prop:oblique-transversal}. When this condition is not available, one should work with general symmetric perturbations rather than restricting to zero-diagonal perturbations.
             \item When $\Delta = 0$, the intersection
		$\mathcal{M}_g(\bar X) \cap \St(n,r)$ is also a smooth embedded
		submanifold in a neighborhood of $\bar X$ with mild assumptions; however, transversality may
		fail in $\R^{n\times r}$. We will return to this point and show  the clean intersection condition in Section~\ref{sec:intersection_property}.
        \item A necessary condition for the LICQ   is that $p_i+q\le nr$. Equivalently, $\dim(\M_i)\geq \frac{r(r+1)}{2}$. For the sparse problem,  if $\dim(\M_g)<\frac{r(r+1)}{2}$, this does not necessarily imply the intersection is empty, but  the transversality condition may fail; see Proposition~\ref{lem:M_1 and M_2 are not transverse}.
    \end{enumerate}
	\end{remark}

Since the ManPG subproblem is posed on the affine tangent slice
$X+\T_X\St_\Delta(n,r)$, the relevant local geometry for identification is given by its intersection with the active manifold. When $X^*$ is a critical point of \eqref{prob:P-Delta}, Lemma~\ref{lem:tangt_prox_optcond} gives $\mathcal T_t(X^*)=X^*$, so that
$X^*\in \St_\Delta(n,r)\cap \M_g(X^*)$. The following incidence result is stated under the  transversality condition at $X^*$; Proposition~\ref{lem:generic-transv} gives this condition for almost every $\Delta\in\mathscr P$ under its hypotheses.
		\begin{lemma}
			\label{lem:incidence-Tcap}
				Assume Assumption~\ref{assump:finite-active-manifolds}, and let
		$X^*\in\St_\Delta(n,r)$ be a critical point of \eqref{prob:P-Delta}. Suppose  that
		\(\rmN_{X^*}\M_g(X^*)\cap\rmN_{X^*}\St_\Delta(n,r)=\{0\}\). Set \(Y^*=X^*\) and define \(\T_\cap(X)\) in \eqref{def:slice_tangent_inter} with \(\bar Y=Y^*\).
			There exist a neighborhood $\mathcal{U}_0 \subset \St_\Delta(n,r)$
			of $X^*$ and a neighborhood $\mathcal{V}_0 \subset \mathbb{R}^{n\times r}$
			of $Y^*$ such that the restricted incidence set
		\[
		\mathcal{I}_0
		:= \bigl\{ (X,Y) \in \mathcal{U}_0 \times \mathcal{V}_0 :
		Y \in \T_{\cap}(X) \bigr\}
		\]
		is a smooth embedded  submanifold of
		$\St_\Delta(n,r) \times \mathbb{R}^{n\times r}$. In particular, $\T_\cap(X)\cap \mathcal{V}_0$ is also a smooth embedded submanifold of $\R^{n\times r}$ for every $X\in\mathcal{U}_0$.
	\end{lemma}

	\begin{proof}
			The proof uses local defining functions and a normal-coordinate
			representation of the affine tangent slices; see
			Appendix~\ref{app:proof-incidence-Tcap}.
	\end{proof}

			\subsection{\texorpdfstring{Strong local minimizers and continuity of parallel subspaces}{Strong local minimizers and continuity of parallel subspaces}}\label{subsec:strong-criticality}
			This subsection provides the technical lemmas needed for the identification argument.
			For the sensitivity analysis of the ManPG subproblem \eqref{eq:tangent_problem}, we need the following strong critical point condition.

	\begin{definition}[Strong critical point {\cite[Def.~5.6]{lewis2002active}}]\label{def:strong-critical}
	Let $\E$ be a Euclidean space and let
	$h:\E\to\bar\R$ be partly smooth at $x_0$
	relative to a manifold $\mathcal M\subset \E$.
	We say that $x_0$ is a \emph{strong local minimizer} of $h|_{\mathcal M}$ if
	there exist a neighborhood $U$ of $x_0$ and a constant $\alpha>0$ such that
	\begin{equation}\label{eq:QG-on-M}
		h(x)\ \ge\ h(x_0)+\alpha\|x-x_0\|^2,\qquad \forall\,x\in \mathcal M\cap U.
	\end{equation}
	We call $x_0$ a \emph{strong critical point} of $h$ relative to $\mathcal M$ if
	\begin{enumerate}[label={(\arabic*)}]
		\item $x_0$ is a strong local minimizer of $h|_{\mathcal M}$, and
		\item $0\in \ri\,\partial h(x_0)$.
	\end{enumerate}
\end{definition}
 
We now relate this notion to the tangent proximal subproblem at a nondegenerate critical point $X^*$ of \eqref{prob:P-Delta}. By Lemma~\ref{lem:tangt_prox_optcond}, one has $\mathcal T_t(X^*)=X^*.$
Define the function on the affine tangent slice $X^*+\T_{X^*}\St_\Delta(n,r)$ by
\[
\varphi_{X^*}(Y)
:= \ip{\grad f(X^*)}{Y-X^*}+g(Y)+\frac{1}{2t}\|Y-X^*\|_{\mathrm F}^2,
\qquad
Y\in X^*+\T_{X^*}\St_\Delta(n,r).
\]
Then $\varphi_{X^*}$ is $(1/t)$-strongly convex on $X^*+\T_{X^*}\St_\Delta(n,r)$. Hence its minimizer
\(
Y^*=\mathcal T_t(X^*)=X^*
\)
is unique and satisfies the quadratic growth inequality
\[
\varphi_{X^*}(Y)\ge \varphi_{X^*}(X^*)+\frac{1}{2t}\|Y-X^*\|_{\mathrm F}^2
\qquad
\forall\,Y\in X^*+\T_{X^*}\St_\Delta(n,r).
\]
In particular, $X^*$ is a strong local minimizer of $\varphi_{X^*}$ restricted to the manifold
$\T_\cap(X^*).$
Moreover, the nondegeneracy condition \eqref{eq:ND} is exactly
\[
0\in \ri\,\partial \varphi_{X^*}(X^*).
\]
Therefore, $X^*$ is a strong critical point of $\varphi_{X^*}$ relative to $\T_\cap(X^*)$. The following standard second-order characterization will be used below.

\begin{proposition}\label{lem:QG-implies-PDHess}
	Let $\M$ be a $C^2$ embedded Riemannian manifold and let
	$h:\M\to\R$ be $C^2$ in a neighborhood of $\bar y\in\M$.
	Then $\bar y$ is a strong local minimizer of $h$ if and only if
	\[
	\grad h(\bar y)=0
	\]
	and there exists $\beta>0$ such that
	\[
	\langle \eta,\Hess h(\bar y)[\eta]\rangle
	\geq \beta\|\eta\|^2,
	\qquad
	\forall\,\eta\in \T_{\bar y}\M .
	\]
\end{proposition}

	\begin{proof}
			The proof follows by applying the quadratic-growth inequality
			along geodesics through $\bar y$; see
			Appendix~\ref{app:proof-QG-implies-PDHess}.
	\end{proof}

		We now formalize the local behavior of the solution mapping  $Y_t^\sharp(X)$ associated with the structured
	tangent  subproblem~\eqref{eq:tangent_sub_supp}. The result can be viewed as a
	Riemannian analogue of the parametric strong-minimizer theory developed for Euclidean
	problems in \cite{lewis2002active,bareilles2023newton}, tailored here to the ManPG subproblem.
	The transversality condition \eqref{eq:transv-Mg-StDelta} is equivalent to the LICQ condition
	for the two equality constraints $\M_g(\bar X)$ and $\St_{\Delta}$. Consequently, an
	alternative route is to invoke Robinson's strong regularity for parametric KKT systems
	\cite[Thm.~2.1]{robinson1980strongly} (see also the discussion in \cite{lewis2002active}).

In this work, we instead apply a manifold implicit-function argument to keep the proof
self-contained. We also allow the ManPG stepsize $t$ to vary over a
compact interval, so that the resulting local smoothness and identification statements
apply uniformly to varying-step tangent proximal mappings.

\begin{lemma}\label{lem:Ysharp-smooth}
Assume the hypotheses of Lemma~\ref{lem:incidence-Tcap} and the nondegeneracy condition \eqref{eq:ND} at $X^*$.
Let \(\mathcal I=[t_{\min},t_{\max}]\subset(0,\infty)\) be any compact interval. Then there exist neighborhoods
$\mathcal U_1\ni X^*$ in $\St_\Delta(n,r)$ and
$\mathcal V_1\ni X^*$ in $\mathbb R^{n\times r}$ such that, for every
$(X,t)\in\mathcal U_1\times\mathcal I$, the problem
\[
    \min_{Y\in\T_{\cap}(X)\cap\mathcal V_1}\rho_t(X,Y)
\]
admits a unique strong local minimizer $Y_t^\sharp(X)\in\mathcal V_1$.
Moreover, the mapping
\[
    (X,t)\mapsto Y_t^\sharp(X)
\]
is of class $C^1$ on $\mathcal U_1\times\mathcal I$.
\end{lemma}

\begin{proof}
By Lemma~\ref{lem:incidence-Tcap}, there exist neighborhoods
$\mathcal U_0\subset\St_\Delta(n,r)$ of $X^*$ and
$\mathcal V_0\subset\mathbb R^{n\times r}$ of $X^*$ such that
\[
    \mathcal P_0
    :=
    \{(X,Y)\in\mathcal U_0\times\mathcal V_0:
      Y\in\T_{\cap}(X)\}
\]
is a smooth embedded submanifold. Consider the parameterized family
\[
    \min_{Y\in\T_{\cap}(X)\cap\mathcal V_0}\rho_t(X,Y),
    \qquad
    (X,t)\in\mathcal U_0\times\mathcal I .
\]
Its first-order condition can be written on
$\mathcal P_0\times\mathcal I$ as
\[
    G(X,t,Y)
    :=
    \grad_{\T_{\cap}(X)\cap\mathcal V_0}\rho_t(X,\cdot)(Y)
    =
    0 .
\]
Since $t$ is bounded away from zero, $\rho_t$ and $G$ depend smoothly on
$(X,t,Y)$.

At $X=X^*$, the point $Y=X^*$ solves the tangent proximal subproblem for every
$t\in\mathcal I$. For each fixed $t$, the function
$Y\mapsto \rho_t(X^*,Y)$ is strongly convex on the affine tangent slice
$X^*+\T_{X^*}\St_\Delta(n,r)$, and hence its restriction to
$\T_\cap(X^*)$ has quadratic growth at $X^*$. Since $g$ is smooth along the
active manifold, the restricted objective is smooth on $\T_\cap(X^*)$.
Therefore Proposition~\ref{lem:QG-implies-PDHess} gives the positive
definiteness of the restricted Riemannian Hessian at $X^*$ for each
$t\in\mathcal I$. The restricted Hessian depends continuously on $t$, and the
compactness of $\mathcal I$ yields a uniform positive lower bound.

Applying the parameterized implicit function theorem and using compactness of
$\mathcal I$, there exist neighborhoods
$\mathcal U_1\subset\mathcal U_0$ and
$\mathcal V_1\subset\mathcal V_0$, independent of $t\in\mathcal I$, and a
unique $C^1$ mapping
\[
    (X,t)\mapsto Y_t^\sharp(X)
\]
satisfying
\[
    G\bigl(X,t,Y_t^\sharp(X)\bigr)=0,
    \qquad
    (X,t)\in\mathcal U_1\times\mathcal I .
\]
By uniform positive definiteness and continuity, after shrinking
$\mathcal U_1$ and $\mathcal V_1$ if necessary,
$Y_t^\sharp(X)$ is the unique strong local minimizer of
$\rho_t(X,\cdot)$ on $\T_\cap(X)\cap\mathcal V_1$ for every
$(X,t)\in\mathcal U_1\times\mathcal I$.
\end{proof}

	To complete the identification argument, we will need a separation criterion for points lying outside the relative interior of a convex set. The following is the standard relative-interior separation criterion;
	see~\cite[Theorem~2.39 and Exercise~2.45(e)]{Rockafellar2009}.
	\begin{lemma}\label{lem:ri-separation}
		Let $\E$ be a finite-dimensional Euclidean space and let $C\subset \E$ be a nonempty
		closed convex set with $b\in \aff C$. Then $b\notin \ri\,C$ if and only if there exists
		a nonzero vector $h\in \para( C)$ such that
		\[
		\inf_{z\in C}\ \langle h,z\rangle\ \ge\ \inp{h}{b} .
		\]
	\end{lemma}

	In addition, we will need a notion of convergence for subspaces of fixed dimension. The following definition recalls the standard Grassmannian convergence in terms of orthogonal projectors.
	\begin{definition}[Grassmannian convergence of subspaces]\label{def:grassmann-conv}
		Let $\E$ be a finite-dimensional Euclidean space and let $V_k,V\subset \E$
		be linear subspaces with $\dim V_k=\dim V$ for all $k$. We say that $V_k$ converges to $V$
		\emph{in the Grassmannian sense} if the orthogonal projectors onto these subspaces converge
		in operator norm, i.e.,
		\[
		\|P_{V_k}-P_V\|\ \to\ 0,
		\]
		where $P_{V_k}$ and $P_V$ denote the orthogonal projections onto $V_k$ and $V$, respectively.
	\end{definition}

	\subsection{Identification theorem}

We are now ready to show that the local minimizer $Y_t^\sharp(X)$ of the
restricted problem is in fact the minimizer of the ManPG subproblem
\eqref{eq:tangent_problem}, uniformly for the stepsize
$t\in\mathcal I=[t_{\min},t_{\max}]$. This yields the local identification
property of the ManPG tangent proximal mapping with varying stepsizes.

\begin{theorem}\label{thm:id-strong-crit}
Let $g$ be a convex function satisfying
Assumption~\ref{assump:finite-active-manifolds}. Let
$\Delta\in\mathscr P$, and let $X^*\in\St_\Delta(n,r)$ be a nondegenerate
critical point of \eqref{prob:P-Delta}. Assume that 
\(\rmN_{X^*}\M_g(X^*)\cap\rmN_{X^*}\St_\Delta(n,r)=\{0\}\). Let
$\mathcal I=[t_{\min},t_{\max}]\subset(0,\infty)$ be the
compact stepsize interval in Lemma~\ref{lem:Ysharp-smooth}. Then there exists
a neighborhood $\bar{\mathcal U}\subset\St_\Delta(n,r)$ of $X^*$ such that,
for all $X\in\bar{\mathcal U}$ and $t\in\mathcal I$,
\[
    Y_t^\sharp(X)=\mathcal T_t(X),
    \qquad
    \M_g(\mathcal T_t(X))=\M_g(X^*).
\]
\end{theorem}

\begin{proof}
Let $\mathcal U_1$ and $\mathcal V_1$ be the neighborhoods given by
Lemma~\ref{lem:Ysharp-smooth}. Shrinking them if necessary, we may assume that
\[
    Y_t^\sharp(X)\in\M_g(X^*),
    \qquad
    \forall\,(X,t)\in\mathcal U_1\times\mathcal I .
\]
Define
\[
    b_t(X)
    :=
    -\frac{1}{t}\bigl(Y_t^\sharp(X)-X\bigr)-\grad f(X),
    \qquad
    \mathcal S_t(X)
    :=
    \Proj_{\T_X\St_\Delta}
    \partial g\bigl(Y_t^\sharp(X)\bigr).
\]
Since \eqref{eq:tangent_problem} is strongly convex, it suffices to prove that
\[
    b_t(X)\in\ri\,\mathcal S_t(X),
    \qquad
    \forall\,(X,t)\in\bar{\mathcal U}\times\mathcal I ,
\]
for some neighborhood $\bar{\mathcal U}$ of $X^*$.

Suppose, to the contrary, that no such neighborhood exists. Then there are
sequences $X_j\to X^*$ and $t_j\in\mathcal I$ such that
\[
    b_{t_j}(X_j)\notin\ri\,\mathcal S_{t_j}(X_j),
    \qquad \forall\,j .
\]
By compactness of $\mathcal I$, after taking a subsequence, one has
$t_j\to\bar t\in\mathcal I$. Set
\[
    Y_j:=Y_{t_j}^\sharp(X_j),
    \qquad
    b_j:=b_{t_j}(X_j),
    \qquad
    \mathcal S_j:=\mathcal S_{t_j}(X_j).
\]
The parameterized smoothness in Lemma~\ref{lem:Ysharp-smooth} gives
$Y_j\to X^*$. Moreover, since $Y_{\bar t}^\sharp(X^*)=X^*$,
\[
    b_j\to
    b^*:=-\grad f(X^*),
    \qquad
    \mathcal S_j\to
    \mathcal S^*
    :=
    \Proj_{\T_{X^*}\St_\Delta}\partial g(X^*).
\]

We now verify the affine-hull inclusion directly. For all large $j$, set
\[
    T_j:=\T_{X_j}\St_\Delta(n,r),
    \qquad
    L_j:=\T_{Y_j}\M_g(X^*).
\]
Since $Y_j$ is the local minimizer of the restricted problem on
$\M_g(X^*)\cap(X_j+T_j)$, the first-order condition gives, for any
$G_j\in\partial g(Y_j)$,
\[
    \left\langle
    \frac{1}{t_j}(Y_j-X_j)+\grad f(X_j)+G_j,\xi
    \right\rangle=0,
    \qquad
    \forall\,\xi\in L_j\cap T_j .
\]
Using the definition of $b_j$ and $b_j\in T_j$, this is equivalent to
\[
    b_j-\Proj_{T_j}G_j\in (L_j\cap T_j)^\perp\cap T_j .
\]
By the transversality condition, for all large $j$,
\[
    (L_j\cap T_j)^\perp\cap T_j=\Proj_{T_j}(L_j^\perp).
\]
Partial smoothness gives
$L_j^\perp=\rmN_{Y_j}\M_g(X^*)=\para\partial g(Y_j)$. Therefore
\[
    b_j\in \Proj_{T_j}G_j+\Proj_{T_j}\para\partial g(Y_j)
    =
    \aff\bigl(\Proj_{T_j}\partial g(Y_j)\bigr)
    =
    \aff\mathcal S_j
\]
for all sufficiently large $j$.
Since $b_j\notin\ri\,\mathcal S_j$, Lemma~\ref{lem:ri-separation} yields unit
vectors $h_j\in\para(\mathcal S_j)$ satisfying
\[
    \inf_{z\in\mathcal S_j}\langle h_j,z\rangle
    \ge
    \langle h_j,b_j\rangle .
\]
The partial smoothness of $g$, the parameterized continuity of
$Y_t^\sharp(X)$, and the transversality condition imply
\[
    \mathcal S_j\to\mathcal S^*,
    \qquad
    \para(\mathcal S_j)\to\para(\mathcal S^*),
\]
where the latter convergence is in the Grassmannian sense of Definition~\ref{def:grassmann-conv}.
Passing to a further subsequence if necessary, let
$h_j\to h^*\in\para(\mathcal S^*)$. Taking the limit in the separation
inequality gives
\[
    \inf_{z\in\mathcal S^*}\langle h^*,z\rangle
    \ge
    \langle h^*,b^*\rangle .
\]
Hence $b^*\notin\ri\,\mathcal S^*$, which contradicts the nondegeneracy
condition~\eqref{eq:ND} at $X^*$.

Therefore, after shrinking $\bar{\mathcal U}$ if necessary,
\[
    b_t(X)\in\ri\,\mathcal S_t(X),
    \qquad
    \forall\,(X,t)\in\bar{\mathcal U}\times\mathcal I .
\]
The strong convexity of \eqref{eq:tangent_problem} then implies
\[
    \mathcal T_t(X)=Y_t^\sharp(X),
    \qquad
    \forall\,(X,t)\in\bar{\mathcal U}\times\mathcal I .
\]
Since $Y_t^\sharp(X)\in\M_g(X^*)$,   by the local stability condition in
Assumption~\ref{assump:finite-active-manifolds}, after shrinking \(\bar U\) if
necessary, we have 
\[
    \M_g(\mathcal T_t(X))
    =
    \M_g\bigl(Y_t^\sharp(X)\bigr)
    =
    \M_g(X^*),
\]
which  completes the proof.

\end{proof}

By Proposition~\ref{lem:generic-transv}, the transversality assumption in Theorem~\ref{thm:id-strong-crit} holds for almost every \(\Delta\in\mathscr P\) under Assumption~\ref{assump:finite-active-manifolds} and \eqref{eq:oblique-transv-condition}; it also covers the unperturbed case \(\Delta=0\) whenever the same transversality condition holds.

		\begin{remark}\label{rmk:identification}
    (i). If transversality holds automatically, the perturbation step is unnecessary.
This happens, for example, in the sphere case (see Proposition~\ref{prop:oblique-transversal})  or  \(g(X)=\|X\|_{2,1}\). Indeed, let
\(I=\{i\in[n]:\|\bar X_{i,:}\|_2\neq 0\}\). Around \(\bar X\), the active manifold of
\(g\) is
\[
    \mathcal M_g
    =
    \{Y\in\mathbb R^{n\times r}:Y_{I^c,:}=0,\ 
    \|Y_{i,:}\|_2\neq 0\ \text{for } i\in I\}.
\]
Thus,
\(
    \T_{\bar X}\mathcal M_g
    =
    \{Z\in\mathbb R^{n\times r}:Z_{I^c,:}=0\}.
\)
Since \(\bar X_{I^c,:}=0\), any matrix supported on \(I^c\) belongs to
\(\T_{\bar X}\mathrm{St}(n,r)\). Hence,
\(
    \T_{\bar X}\mathrm{St}(n,r)
    +
    \T_{\bar X}\mathcal M_g
    =
    \mathbb R^{n\times r}.
\) This will be discussed further for the  SPCA problem in section~\ref{sec:numerical_exp}.
\\
    (ii).
	The neighborhood $\bar{\mathcal U}$ is chosen as the intersection of several
	previously constructed neighborhoods so that all local properties required in the
	proof hold simultaneously. In particular, we use the following facts.

		\begin{enumerate}[label={(\arabic*)}]
			\item \textbf{Local manifold structure of the incidence set.}
			Under the transversality assumption in Theorem~\ref{thm:id-strong-crit},
			Lemma~\ref{lem:incidence-Tcap} provides neighborhoods $\mathcal U_0\ni X^*$ and
			$\mathcal V_0\ni Y^*$ such that the restricted incidence set $\mathcal I_0$ is a
			smooth embedded submanifold. In particular, for every $X\in\mathcal U_0$, the set 
		$\T_\cap(X)\cap\mathcal V_0$ is a smooth embedded submanifold.

		\item \textbf{Stability of the active-manifold geometry of $g$.}
		Since $g$ is convex and partly smooth at $X^*$ relative to $\M_g(X^*)$,
		there exists a neighborhood $\mathcal V_g\ni X^*$ such that
		\eqref{eq:normal_space_eq_to_parallel_g} holds for all
		$Y\in \mathcal V_g\cap \M_g(X^*)$.

		\item \textbf{Smooth local minimizer.}
		Lemma~\ref{lem:Ysharp-smooth} yields a neighborhood $\mathcal U_1\ni X^*$ such that
		the localized tangent subproblem admits a unique strong local minimizer
			 $Y_t^\sharp(X)\in \mathcal V_1$, and the mapping
			 $(X,t)\mapsto Y_t^\sharp(X)$ is $C^1$.
	\end{enumerate}
\end{remark}

Theorem~\ref{thm:id-strong-crit} verifies, for the perturbed problem~\eqref{prob:P-Delta}, the identification property required in \cite[Assumption~3.7]{si2023riemannian} under the transversality condition. We also note that the condition on $g$ can be extended beyond the convex setting, for instance to prox-regular functions; see \cite{bareilles2023newton}. An immediate consequence is that, once the iterates $X_k$ enter a neighborhood of a nondegenerate critical point $\bar X$, the tangent proximal point $\mathcal T_t(X_k)$ lies on the active manifold $\M_g(\bar X)$. This is the key difference from the Euclidean proximal-gradient method, where the proximal point is itself the next iterate \cite{liang2017activity}. For ManPG, however, the actual iterate is obtained by retracting the tangent step back to the Stiefel manifold, and standard retractions on the Stiefel manifold generally do not preserve the active-manifold structure. In particular, for sparse regularizers, the retracted iterate need not preserve the support pattern identified by $\mathcal T_t(X_k)$. This observation is precisely what motivates the algorithmic developments in Section~\ref{sec:alg_on_intersection}.

Lewis and Tian \cite{LewisTian2024} developed a general metric-space framework that yields identification results for manifold proximal point algorithms, including the Riemannian proximal point algorithm \cite{FerreiraOliveira2002}. Such methods require solving a proximal subproblem over the manifold, which may be as difficult as the original problem. In contrast, we establish identification for ManPG, whose tangent-space proximal subproblem is computationally more tractable.

We next state the finite identification property under the assumption that the sequence $\{X_k\}$ converges. This assumption is reasonable in light of the recent sequential convergence result of Li et al.~\cite{li2024proximal} for ManPG, obtained under a Kurdyka--\L ojasiewicz (K\L) assumption on certain Lyapunov function.

	\begin{corollary} \label{thm:manpg_tangent_identify}
		Under the same conditions of Theorem~\ref{thm:id-strong-crit}, suppose  \(t_k\in I\) for all \(k\) and the sequence $\{X_k\}_{k\ge 0}$ generated by ManPG satisfies $X_k \to   X^*$. Then there
		exists $K<\infty$ such that, for all $k\ge K$, one has
		\[
		\M_g(\mathcal{T}_{t_k}(X_k))= \mathcal{M}_g(X^*).
		\]
	\end{corollary}

%% file: shared/section_geometry_intersection.tex
\section{Geometry of the intersection of the Stiefel manifold and sparse manifolds}\label{sec:intersection_property}

In this section, we restrict attention to sparse regularizers, with particular emphasis on
\(g(X)=\lambda\|X\|_1\).
For such models, the active manifold $\M_g(Y)$ is determined by a sparse pattern and admits an explicit description; see \eqref{eq:manifold_2}. Our goal here is to study the local geometry of its intersection with $\St_\Delta(n,r)$, which will be the basis for the Newton correction developed in  section~\ref{sec:alg_on_intersection}.

The generic transversality result of the previous section provides useful geometric intuition: for almost every perturbation $\Delta\in\mathscr P$, the intersection
\[\overline{\M}_\Delta(Y):=\M_g(Y)\cap \St_\Delta(n,r)\]
is an embedded submanifold of dimension $\dim(\M_g(Y))-\frac{r(r+1)}{2}$.  For the algorithmic construction, however, this result serves only as geometric intuition.

Let $X_k\in \St_\Delta(n,r)$ be an iterate of ManPG, and write
\(Y_k=\mathcal T_t(X_k)\) and \(V_k=Y_k-X_k\).
A natural idea, inspired by \cite{bareilles2023newton}, is to exploit the identified active manifold $\M_g(Y_k)$ and perform a second-order step on a smooth intersection manifold. The difficulty is that the most obvious candidate,
\(\M_g(Y_k)\cap \St_\Delta(n,r)\),
may be empty. A natural alternative is to replace $\St_{\Delta}(n,r)$ by its affine tangent slice and
consider $\M_g(Y_k)\cap\bigl(X_k+\T_{X_k}\St_{\Delta}(n,r)\bigr)$. This approach is not
satisfactory for our purpose. Indeed, a point in the affine tangent slice has the form
$X_k+H$ with $H\in\T_{X_k}\St_\Delta(n,r)$, and
\[(X_k+H)^\top(X_k+H)=(I_r+\Delta)+H^\top H.\]
Thus the deviation from the original Stiefel constraint is given by $H^\top H$. Although this
is a second-order term, it is not tied to the actual ManPG step and is therefore not
uniformly controlled.  

Instead, we exploit the fact that the tangent proximal point $Y_k=X_k+V_k$ itself lies on a
nearby Stiefel manifold. Define
\[
\Delta_k:=\Delta+V_k^\top V_k,
\qquad
\St_{\Delta_k}(n,r):=\{Y\in\R^{n\times r}:Y^\top Y=I_r+\Delta_k\}.
\]
Since $V_k\in\T_{X_k}\St_\Delta(n,r)$, we have
\(Y_k^\top Y_k=(X_k+V_k)^\top(X_k+V_k)=I_r+\Delta+V_k^\top V_k=I_r+\Delta_k\),
and hence $Y_k\in\St_{\Delta_k}(n,r)$. In contrast to the affine tangent slice, the mismatch
between $\St_{\Delta_k}(n,r)$ and the original constraint is exactly the second-order quantity
$V_k^\top V_k$, which is generated by the ManPG step and tends to zero along the iteration.
This makes $\St_{\Delta_k}(n,r)$ a much more suitable local model for the constrained
geometry.  Nevertheless, $\Delta_k$ is generated by the iteration and need not
belong to the generic set $\mathscr P_{\rm gen}$, even when the original perturbation
$\Delta$ does. Thus the generic transversality result from the previous section cannot
be used directly, and a clean-intersection analysis is needed.

This leads to the intersection candidate
\begin{equation}\label{eq:M_k}
	\overline{\M}_k:=\M_g(Y_k)\cap \St_{\Delta_k}(n,r).
\end{equation}
By construction, $\overline{\M}_k$ is nonempty. The goal of this section is to characterize the local manifold structure of $\overline{\M}_k$ and the geometric properties needed later for optimization on this intersection. For convenience, we also use the parameterized notation
\(\overline{\M}_\Delta(Y):=\M_g(Y)\cap \St_\Delta(n,r)\),
where $(Y,\Delta)$ are viewed as fixed parameters.

When $r=1$, the transversality condition between $\M_g(Y)$ and $\St_\Delta(n,r)$ always holds. For $r\ge 2$, however, transversality may fail, and we instead verify a constant-dimension condition under mild assumptions on the support pattern. 
  Under this condition, $\overline{\M}_\Delta(Y)$ is a smooth embedded submanifold in a neighborhood of the reference point; see Proposition~\ref{prop:verify-intersection}.  The possible failure is already visible in the unperturbed case $\Delta=0$: when the supports in $Y$ are disjoint, one typically has
\(\rmN_X\St_\Delta(n,r)\cap \rmN_X\M_g(Y)\neq\{0\}\),
so transversality fails; see Proposition~\ref{lem:M_1 and M_2 are not transverse}.

In the low-dimensional cases, the constant-dimension condition can be verified directly from the normal-space characterization. In particular, when $r=2$, this yields the local smoothness of $\overline{\M}_\Delta(Y)$ deterministically. When $r=3$, the unperturbed case $\Delta=0$ is still deterministic, while nonzero zero-diagonal perturbations require the one-column-witness condition included in the \(2/3\)-row-cover assumption below. The same support-level condition is then used for all \(r\ge3\), together with a probabilistic estimate in the regime \(n\gg r\).

\subsection{Transversality: when to hold and how to fix it}
We begin with the oblique transversality condition appearing in Proposition~\ref{lem:generic-transv}.  

 \begin{proposition}\label{prop:oblique-transversal}
 Fix a sparse pattern
\[
\mathcal S\subset [n]\times[r],
\qquad
\M_{\mathcal S}:=\{Y\in\R^{n\times r}: Y_{\mathcal S^c}=0\}.
\] Let $X\in \Ob(n,r)\cap \M_{\mathcal S}$. Then it follows that
\[
\rmN_X\Ob(n,r)\cap \rmN_X\M_{\mathcal S}=\{0\}.
\]
\end{proposition}

\begin{proof}
The normal space of the oblique manifold is given by 
\[
\rmN_X\Ob(n,r)
=
\{X\Lambda:\ \Lambda\in\mathbb D^r\},
\]
where $\mathbb D^r$ denotes the set of diagonal $r\times r$ matrices. On the other hand, we have 
\[
\rmN_X\M_{\mathcal S}=\{Z\in\R^{n\times r}: Z_{\mathcal S}=0\}.
\]

Now let $Z\in \rmN_X\Ob(n,r)\cap \rmN_X\M_{\mathcal S}$. Then
$Z=X\Lambda$ for some $\Lambda\in\mathbb D^r.$
Since $X\in \M_{\mathcal S}$, we have $X_{\mathcal S^c}=0$, and hence
\(
Z_{\mathcal S^c}=(X\Lambda)_{\mathcal S^c}=0.
\)
But $Z\in \rmN_X\M_{\mathcal S}$ also implies
\(
Z_{\mathcal S}=0.
\)
Therefore $Z=0$, proving
$\rmN_X\Ob(n,r)\cap \rmN_X\M_{\mathcal S}=\{0\}.$
This completes the proof.
\end{proof}
As a special case, when $r=1$, the oblique manifold $\Ob(n,1)$ coincides with the sphere
$\mathbb S^{n-1}$. Hence the transversality
condition between the sparse active manifold and the Stiefel constraint holds automatically in
this case, and no perturbation is needed. Moreover, the normalization map provides a natural
retraction on the intersection manifold and preserves the support pattern.

\medskip
When $r>1$, transversality is no longer automatic. For $Y\in\R^{n\times r}$ and $j\in[r]$, define the support of the $j$-th column by
\[\mathcal I_j^Y:=\{\,i\in[n]:\ Y_{ij}\neq 0\,\}.\]
The next result shows that, once two columns of $Y$ have disjoint supports, the transversality condition fails.

\begin{proposition}\label{lem:M_1 and M_2 are not transverse}
Let $r>1$, let $\Delta\in\mathbb S^r$ satisfy
$I_r+\Delta\succ0$, and let
$Y\in\St_\Delta(n,r)$. Let $\M_g(Y)$ be the submanifold defined in \eqref{eq:manifold_2}. Then
	\[
	\mathrm{span}\{YS_{\ell k}:\ \mathcal I_\ell^Y\cap \mathcal I_k^Y=\emptyset\}
	\subset \rmN_Y\St_\Delta(n,r)\cap \rmN_Y\M_g(Y).
	\]
	In particular, if there exist $\ell\neq k$ such that $\mathcal I_\ell^Y\cap \mathcal I_k^Y=\emptyset$, then $\St_\Delta(n,r)$ and $\M_g(Y)$ are not transverse at $Y$.
\end{proposition}

\begin{proof}
	Note that $\rmN_Y\St_\Delta(n,r)$ is spanned by $\{YS_{ij}:1\le i\le j\le r\}$,
	where $S_{ij}=e_ie_j^\top+e_je_i^\top$. In particular,
	 {$YS_{\ell k}=Ye_\ell e_k^\top+Ye_k e_\ell^\top$. Thus $YS_{\ell k}$ has only two possibly nonzero columns: its $k$-th column is $Y_\ell$, and its $\ell$-th column is $Y_k$.}
	
	On the other hand, $\rmN_Y \M_g(Y)=\T_Y \M_g(Y)^\perp$, and for the active manifold
	$\M_g(Y)$ defined by \eqref{eq:manifold_2} this normal space consists of matrices whose
	entries vanish on the active coordinates:
	\[
	\rmN_Y \M_g(Y)
	=\{\,Z\in\R^{n\times r}:\ Z_{ij}=0 \ \text{whenever } i\in \mathcal I^Y_j,\ j\in[r]\,\}.
	\]
	If $\mathcal I^Y_\ell\cap \mathcal I^Y_k=\emptyset$, then the $k$-th column of
	$YS_{\ell k}$ is supported only on $\mathcal I^Y_\ell$ and the $\ell$-th column is
	supported only on $\mathcal I^Y_k$. Since these supports are disjoint, it follows that
	$YS_{\ell k}$ vanishes on each $\mathcal I^Y_j$ in the corresponding column and hence
	$YS_{\ell k}\in \rmN_Y \M_g(Y)$. Therefore,
	$YS_{\ell k}\in \rmN_Y\St_\Delta(n,r)\cap \rmN_Y \M_g(Y)$. This completes the proof.
\end{proof}

Empirically, the supports of stationary points in the compressed modes problem can be nearly disjoint across different columns; see \cite[Fig.~5]{chen2020proximal}. This phenomenon is consistent with Proposition~\ref{lem:M_1 and M_2 are not transverse} and suggests that, near such stationary points, $\St(n,r)$ and the associated active manifold $\M_g$ may fail to intersect transversely. Furthermore, transversality may fail even when every pair of supports overlaps, as shown in Example~\ref{example:r_4_wrong} in Appendix. By contrast, for the perturbed manifold $\St_\Delta(n,r)$ with nonzero off-diagonal entries in $\Delta$, completely disjoint supports are generically ruled out, since the constraint $X^\top X=I_r+\Delta$ forces nonzero correlations between different columns. This helps explain why the perturbation is effective in restoring the regular intersection geometry. For nongeneric perturbations $\Delta$, however, transversality need not hold, and the clean-intersection analysis developed below becomes necessary.

\subsection{Intersection graph}

To describe the support interaction across columns, we introduce a graph representation of the active pattern. For $Y\in\R^{n\times r}$ and $\ell\in[n]$, define the active set of the $\ell$-th row by
\[J_\ell^Y:=\{\,k\in[r]:\ Y_{\ell k}\neq 0\,\}.\]
We associate with $Y$ an undirected support graph
\begin{manualtaggedequation}
\begin{equation}\label{eq:graph}\tag{G}
	\mathcal G(Y)=(\mathcal V,\mathcal E),\qquad
	\mathcal V=[r],\quad
	\mathcal E:=\Bigl\{\{i,j\}\subseteq\mathcal V:\ \exists\,\ell\in[n]\ \text{s.t.}\ \{i,j\}\subseteq J_\ell^Y\Bigr\},
\end{equation}
\end{manualtaggedequation}
where self-loops $\{i,i\}$ are allowed.

With this notation, $\{i,j\}\notin\mathcal E$ means that the $i$-th and $j$-th columns of $Y$ have disjoint supports. Proposition~\ref{lem:M_1 and M_2 are not transverse} then gives
\[
\mathrm{span}\{YS_{ij}:\{i,j\}\notin\mathcal E\}\subset \rmN_Y\St_\Delta(n,r)\cap \rmN_Y \M_g(Y).
\]
As we will see, the connectivity structure of $\mathcal G(Y)$ governs the local
geometry of $\overline{\M}_\Delta(Y)$. We quantify this connectivity by the edge count
\begin{equation}\label{def:intersect_number}
	\gamma(Y):=|\mathcal E|.
\end{equation}
In particular, define the structured subspace of symmetric matrices
\[
\mathbb S_{r,\gamma}(Y):=\Bigl\{\Lambda\in{\color{blue}\mathbb S^r}:\ \Lambda_{ij}=0\ \text{whenever}\ \{i,j\}\notin \mathcal E\Bigr\}.
\]
Equivalently, $\Lambda_{ij}$ is allowed to be nonzero only if the vertices $i$ and $j$
are adjacent in $\mathcal G(Y)$ (including the case $i=j$).
By construction,
\begin{equation}\label{eq:dim_sparse_set_symmetric}
	\dim\bigl(\mathbb S_{r,\gamma}(Y)\bigr)=\gamma(Y).
\end{equation}

Let $\mathcal U_s(Y)\subset\M_g(Y)$ be a neighborhood of $Y$ on which
the support pattern agrees with that of $Y$. Define
\[
    \mathcal H_Y:\M_g(Y)\to\mathbb S^r,
    \qquad
    \mathcal H_Y(X)
    :=
    X^\top X-(I_r+\Delta).
\]
Then, we have
\(    \overline{\M}_\Delta(Y)=\mathcal H_Y^{-1}(0),
\)
and, for $X\in\M_g(Y)$ and $\eta\in\T_X\M_g(Y)$,
\[
    \mathrm D\mathcal H_Y(X)[\eta]
    =
    \eta^\top X+X^\top\eta.
\]
Moreover, it follows that
\[
    \operatorname{range}\bigl(\mathrm D\mathcal H_Y(X)\bigr)
    \subseteq \mathbb S_{r,\gamma}(Y),
\]
because the entries corresponding to missing edges vanish.

{\color{blue}
For later use, we introduce the normal--tangent interaction matrix
associated with this differential. For
$X\in\overline{\M}_\Delta(Y)\cap\mathcal U_s(Y)$,
let $x:=\vvec(X)$, set $q:=r(r+1)/2$ and
$s:=\dim\M_g(Y)$, and let
$U_r\in\mathbb R^{r^2\times q}$ be the duplication matrix satisfying
$U_r\svec(\Lambda)=\vvec(\Lambda)$ for every $\Lambda\in\mathbb S^r$.
Define
$B_x:=(I_r\otimes X)U_r\in\mathbb R^{nr\times q}$, whose columns form
a basis of $\vvec(\rmN_X\St_\Delta(n,r))$, and let
$E_x\in\mathbb R^{nr\times s}$ be an orthonormal basis matrix of
$\vvec(\T_X\M_g(Y))$. We set
\begin{equation}\label{eq:matrix_A}
    A_x:=B_x^\top E_x.
\end{equation}
Since $\M_g(Y)$ is a fixed coordinate subspace, $E_x$ may be chosen
independently of $x$ on $\mathcal U_s(Y)$. The calculation in
Appendix~\ref{app:section4-examples} gives
\begin{equation}\label{eq:rank_DH_rankA}
    \rank\bigl(\mathrm D\mathcal H_Y(X)\bigr)=\rank(A_x).
\end{equation}
}

Example~\ref{ex:support-graph-mixed} illustrates these definitions.

\begin{example}\label{ex:support-graph-mixed}
	\noindent
	\begin{minipage}[c]{0.54\textwidth}
		\centering
		\[
		Y=\begin{bmatrix}
			\frac{1}{\sqrt2} & \frac{1}{2} & 0 & 0\\
			\frac{1}{\sqrt2} & -\frac{1}{2} & 0 & 0\\
			0 & \frac{1}{2} & \frac{1}{\sqrt2} & 0\\
			0 & \frac{1}{2} & -\frac{1}{\sqrt2} & 0\\
			0 & 0 & 0 & 0\\
			0 & 0 & 0 & 1
		\end{bmatrix}\in\St(6,4).
		\]
	\end{minipage}\hfill
	\begin{minipage}[c]{0.42\textwidth}
		\centering
		\begin{tikzpicture}[scale=1]
			\tikzset{
				vertex/.style={draw,circle,minimum size=5mm,inner sep=0.6pt,line width=0.6pt},
				edge/.style={line width=0.6pt}
			}
			\node[vertex] (v1) at (0,1.6) {\scriptsize 1};
			\node[vertex] (v2) at (1.8,1.6) {\scriptsize 2};
			\node[vertex] (v4) at (0,0)   {\scriptsize 4};
			\node[vertex] (v3) at (1.8,0) {\scriptsize 3};

			\draw[edge] (v1) -- (v2);
			\draw[edge] (v2) -- (v3);

			\path (v1) edge[loop above,looseness=7] (v1);
			\path (v2) edge[loop above,looseness=7] (v2);
			\path (v3) edge[loop below,looseness=7] (v3);
			\path (v4) edge[loop below,looseness=7] (v4);
		\end{tikzpicture}
	\end{minipage}

	\medskip
	
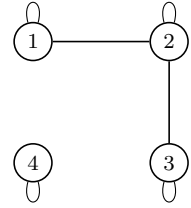
\captionof{figure}{Support graph $\mathcal G(Y)$ for a sparse orthogonal $Y\in\St(6,4)$ with mixed overlaps:
		$\mathcal V=\{1,2,3,4\}$ and $\mathcal E=\{(1,1),(2,2),(3,3),(4,4),(1,2),(2,3)\}$, hence $\gamma(Y)=6$.
		Columns $1$--$2$ and $2$--$3$ overlap, while $1$--$3$ do not; column $4$ is disjoint from the others.}
	\label{fig:support-graph-mixed}

	\medskip
	The induced sparse subspace $\mathbb{S}_{4,\gamma}(Y)$ consists of matrices of the form
	\[
	\mathbb{S}_{4,\gamma}(Y)
	= \left\{
	\begin{bmatrix}
		a & b & 0 & 0\\
		b & c & d & 0\\
		0 & d & e & 0\\
		0 & 0 & 0 & f
	\end{bmatrix}
	:\;
	a,b,c,d,e,f \in \mathbb{R}
	\right\}.
	\]
\end{example}

Motivated by Proposition~\ref{lem:M_1 and M_2 are not transverse}, we impose the following  condition on this intersection in a neighborhood where the support pattern remains fixed. It is a sufficient condition for the clean-intersection geometry.

\begin{assumption}\label{assump:intersection}
For every $X\in\mathcal U_s(Y)\cap\St_\Delta(n,r)$, it follows that 
\[
    \rmN_X\St_\Delta(n,r)\cap\rmN_X\M_g(Y)
    =
    \operatorname{span}\{
        XS_{ij}: i<j,\ \{i,j\}\notin\mathcal E(Y)
    \}.
\]
\end{assumption}
\begin{theorem}\label{thm:Mg-intersection-submanifold}
	Suppose Assumption~\ref{assump:intersection} holds and  that  $\overline{\M}_\Delta(Y)\cap \mathcal U_s(Y)$ is nonempty. Then, the set
	$\overline{\M}_\Delta(Y)\cap \mathcal U_s(Y)$ is an embedded submanifold of
	$\M_g(Y)$ with codimension $\gamma(Y)$. Moreover, for any
	$X\in \overline{\M}_\Delta(Y)\cap \mathcal U_s(Y)$,
	\[
	\T_X\overline{\M}_\Delta(Y)=\T_X\St_\Delta(n,r)\cap \T_X\M_g(Y).
	\]
\end{theorem}
\begin{proof}
Fix
$X\in\overline{\M}_\Delta(Y)\cap\mathcal U_s(Y)$ and set
\(
    L_X:=\mathrm D\mathcal H_Y(X):
    \T_X\M_g(Y)\to\mathbb S^r.
\)
With respect to the Frobenius inner product, its adjoint satisfies
\[
    L_X^*(\Lambda)
    =
    2\Proj_{\T_X\M_g(Y)}(X\Lambda),
    \qquad \Lambda\in\mathbb S^r.
\]
Consequently, we have
\[
    \ker L_X^*
    =
    \{\Lambda\in\mathbb S^r:
      X\Lambda\in\rmN_X\M_g(Y)\}.
\]
Since
$\rmN_X\St_\Delta(n,r)=\{X\Lambda:\Lambda\in\mathbb S^r\}$
and $X$ has full column rank, Assumption~\ref{assump:intersection}
gives
\[
    \dim\ker L_X^*
    =
    \frac{r(r+1)}2-\gamma(Y).
\]
It follows that
\(
    \rank L_X=\gamma(Y).
\)

For every $Z\in\mathcal U_s(Y)$,
$\operatorname{range}(L_Z)\subseteq\mathbb S_{r,\gamma}(Y)$, so
$\rank L_Z\leq\gamma(Y)$. Since $\rank L_X=\gamma(Y)$, a nonzero
$\gamma(Y)\times\gamma(Y)$ minor remains nonzero near $X$.
Hence $\rank L_Z=\gamma(Y)$ in a neighborhood of $X$.
The constant-rank level-set theorem therefore shows that
$\overline{\M}_\Delta(Y)\cap\mathcal U_s(Y)$ is locally an embedded
submanifold of $\M_g(Y)$ of codimension $\gamma(Y)$.

Finally, we obtain
\[
\begin{aligned}
\T_X\overline{\M}_\Delta(Y)
&=\ker L_X\\
&=\{\eta\in\T_X\M_g(Y):
      \eta^\top X+X^\top\eta=0\}\\
&=\T_X\St_\Delta(n,r)\cap\T_X\M_g(Y).
\end{aligned}
\]
\end{proof}

\begin{remark}\label{rmk:ident_fail_Stiefel}
	\begin{enumerate}[label={(\arabic*)}]
		\item In contrast with Proposition~\ref{lem:generic-transv}, this result does not rely on transversality in $\R^{n\times r}$. Instead, it is based on a clean-intersection condition; see \cite[Section~3]{drusvyatskiy2015transversality}. In particular, the intersection remains a manifold and satisfies
		\[
		\T_X(\M_1\cap \M_2)=\T_X\M_1\cap \T_X\M_2.
		\]
		Although the dimension formula again takes the form $\dim(\M_g)-\gamma(\cdot)$, the relevant quantity is evaluated on different active manifolds associated with different feasible intersection points. In the unperturbed case $\Delta=0$, the orthogonality constraints allow column supports to be disjoint, so $\St(n,r)\cap\M_g$ may have many missing edges and typically corresponds to sparser support patterns. By contrast, for the perturbed model $Y^\top Y=I_r+\Delta$, a nonzero off-diagonal entry $\Delta_{ij}$ enforces $\langle Y_i,Y_j\rangle=\Delta_{ij}\neq 0$, and hence $\supp(Y_i)\cap\supp(Y_j)\neq\emptyset$.

		\item {\color{blue}Theorem~\ref{thm:Mg-intersection-submanifold} relies on the local constancy of $\rank(\mathrm D\mathcal H_Y(X))$, or equivalently $\rank(A_x)$ by \eqref{eq:rank_DH_rankA}, on the support-invariant neighborhood $\mathcal U_s(Y)\subset \M_g(Y)$.} This is different from Theorem~\ref{thm:id-strong-crit}, which is stated on a neighborhood $\bar{\mathcal U}\subset \St_\Delta(n,r)$ and does not require the support pattern to remain fixed. In the ManPG algorithm, such support invariance is not automatic. Consequently, the identification result established for generic $\Delta$ does not directly extend to the unperturbed Stiefel manifold.
	\end{enumerate}
\end{remark}



\subsection{Verification of Assumption~\ref{assump:intersection}}
In this subsection, we verify Assumption~\ref{assump:intersection} deterministically in the following cases: (i) $\Delta=0$ with $r\le 3$; (ii) general $\Delta$ with $r=2$; and (iii) the \(2/3\)-row-cover condition for all \(r\ge3\), including the one-column witnesses required when \(\Delta\neq0\). We then establish a high-probability guarantee for the remaining cases under a sparsity model, with particular emphasis on the regime $n\gg r$.

Recall the support graph $\mathcal G(Y)=(\mathcal V,\mathcal E)$ defined in
\eqref{eq:graph}.

\paragraph{\textbf{1) The case  $r=2$}.}

\begin{lemma}\label{lem:diag-vanish}
	Let $Y\in\St_\Delta(n,r)$ and suppose that $\St_\Delta(n,r)\cap \M_g(Y)\neq \emptyset$. Let $X\in\St_\Delta(n,r)\cap\M_g(Y)$.
	If $X$ has the support as $Y$,
	\(
	Z\in \rmN_X\St_\Delta(n,r)\cap \rmN_X\M_g(Y),
	\)
	and $Z=X\Lambda$ for some $\Lambda\in{\color{blue}\mathbb S^r}$, then $\diag((I_r+\Delta)\Lambda)=0$. In particular, if $\Delta=0$, then
	\[
	\Lambda_{ii}=0,\qquad i=1,\dots,r.
	\]
\end{lemma}

\begin{proof}
	Since $Z\in \rmN_X\M_g(Y)$, we have $Z_{\ell i}=0$ whenever $X_{\ell i}\neq 0$.
	Therefore, for each $i$,
	\[
	\langle X_i,Z_i\rangle=0.
	\]
	On the other hand, using $Z=X\Lambda$ and $X^\top X=I_r+\Delta$ yields 
	\[
	0=	\langle X_i,Z_i\rangle
	=e_i^\top X^\top Z e_i
	=e_i^\top (X^\top X)\Lambda e_i
	=\Big((I_r+\Delta)\Lambda\Big)_{ii}
	\]
	for $i\in[r]$.
\end{proof}

\begin{proposition}\label{prop:intersection-for_r=2}
	Let $r=2$ and 
	assume   $\St_\Delta(n,2)\cap \M_g(Y)\neq\emptyset$.
	Let $\mathcal U_s(Y)$ be the neighborhood so that any $X\in\mathcal U_s(Y)$ has the same support as $Y$.
	Then, it holds for any  $X\in\St_\Delta(n,2)\cap \M_g(Y)\cap\mathcal U_s(Y)$ that
	\[
	\rmN_X\St_\Delta(n,2)\cap \rmN_X\M_g(Y)= \mathrm{span}\{\,XS_{12}:\ \{1,2\}\notin\mathcal E(Y)\,\},
	\]
	where $\mathcal E(Y)$ is the edge set introduced by the column supports of $Y$.
\end{proposition}

The proof is given in
Appendix~\ref{app:proof-intersection-r2}.

\paragraph{\textbf{2) The case  $r=3$}.}

For $r=3$, we first show that the constant dimension holds for $\Delta=0$.
	\begin{proposition}[Case $r=3, \Delta=0$]\label{prop:intersection-r3-clean}
		Let $r=3$, let $Y$ be a reference point, and let
			$X\in\St(n,3)\cap \M_g(Y)\cap \mathcal U_s(Y)$. Then Assumption~\ref{assump:intersection} holds, i.e.,
				\[
				\rmN_X\St(n,3)\cap \rmN_X\M_g(Y)
				=\mathrm{span}\{\,XS_{ij}:\ i<j,\ \{i,j\}\notin\mathcal E(Y)\,\}.
				\]
	\end{proposition}

The proof of Proposition~\ref{prop:intersection-r3-clean} is given in
	Appendix~\ref{app:proof-intersection-r3-clean}. Thus, for \(r=3\) and
	\(\Delta=0\), Assumption~\ref{assump:intersection} holds by the preceding
	proposition. For   nonzero   perturbations, however,
	the case \(r=3\) is not automatic: Assumption~\ref{assump:intersection}
	may fail without one-column witnesses; see
	Example~\ref{ex:nonzero-intersection-r3-denseDelta} in
	Appendix~\ref{app:section4-examples}.

\paragraph{\textbf{3) A $2/3$-row-cover condition for $r\ge 3$}.}

For $r\ge4$, Assumption~\ref{assump:intersection} is not
automatic even when $\Delta=0$; see Example~\ref{example:r_4_wrong} in
Appendix~\ref{app:section4-examples}.

For \(r\ge3\), we impose the following \emph{\(2/3\)-row cover} condition for the support pattern, under which
Assumption~\ref{assump:intersection} can be reduced to two- and three-column row arguments.
For convenience, for each row $\ell$ we denote the set of active columns by
\[
J_\ell^Y:=\{\,k\in[r]:\ Y_{\ell k}\neq 0\,\}\subseteq [r].
\]

\begin{assumption}[$2/3$-row cover]\label{assump:row-cover}
	Let $\mathcal G(Y)=(\mathcal V,\mathcal E)$ be the support graph defined in \eqref{eq:graph},
	and denote the off-diagonal edge set by
	\[
	\mathcal E_{\rm off}:=\{\{i,j\}\in\mathcal E:\ i\neq j\}.
	\]
	There exists a finite set of row indices $\mathcal L\subseteq [n]$ such that
	\[
	|J_\ell^Y|\in\{2,3\}\quad \forall\,\ell\in\mathcal L,
	\qquad\text{and}\qquad
	\mathcal E_{\rm off}
	=\bigcup_{\ell\in\mathcal L}\Big\{\{i,j\}\subseteq J_\ell^Y:\ i\neq j\Big\}.
	\]
	Equivalently, every $\{i,j\}\in\mathcal E_{\rm off}$ is witnessed by at least one row
	$\ell\in\mathcal L$ whose active set $J_\ell^Y$ has cardinality $2$ or $3$ and contains $\{i,j\}$.

	Furthermore, if $\Delta\neq 0$, we   assume that for each column $k\in[r]$ there exists a row index
	$\ell_k\in[n]$ (a \emph{column witness}) such that
	\begin{equation}\label{assump:one_colm_witness}
		Y_{\ell_k k}\neq 0,
		\qquad
		Y_{\ell_k j}=0, \ \ \forall\, j\in[r]\setminus\{k\},
	\end{equation}
	i.e., $J_{\ell_k}^Y=\{k\}$.
\end{assumption}

Then, Assumption~\ref{assump:intersection} holds for $r\ge3$ under Assumption~\ref{assump:row-cover}.  The proof is given in
Appendix~\ref{app:proof-verify-intersection}.
\begin{proposition}[Assumption~\ref{assump:intersection} under a $2/3$-row cover]\label{prop:verify-intersection}
	Fix \(r\ge3\) and \(\Delta\in \mathbb{S}^r, I_r+\Delta\succ 0\).
	Given a reference point $Y$, assume that $\St_\Delta(n,r)\cap\M_g(Y)$ is nonempty.
	Suppose Assumption~\ref{assump:row-cover} holds for $Y$.
	Then, for every
	\(	X\in \St_\Delta(n,r)\cap\M_g(Y)\cap \mathcal U_s(Y),
	\)
	it holds that
	\[
	\rmN_X\St_\Delta(n,r)\cap \rmN_X \M_g(Y)
	\;=\;
	\mathrm{span}\bigl\{\, XS_{ij} \ \big|\ i<j,\ \{i,j\}\notin \mathcal E_{\rm off}(Y)\,\bigr\},
	\]
	where $\mathcal E_{\rm off}(Y)=\{\{i,j\}\in\mathcal E(Y):i\neq j\}$.
\end{proposition}

  The row-cover condition is a support-level sufficient condition for this constant rank condition. We now use a Bernoulli support model only to estimate how often it occurs.
\begin{assumption}[Bernoulli support model]\label{assump:bernoulli-support}
	Fix $r\ge 2$ and $p\in(0,1)$. The support pattern of $X\in\R^{n\times r}$ is generated as follows:
	for each row $\ell\in[n]$ and column $j\in[r]$, let the indicator variables
	\[
	\delta_{\ell j}:=\mathbf 1\{X_{\ell j}\neq 0\}\sim{\rm Bernoulli}(p)
	\]
	be i.i.d.\ across $(\ell,j)$. 
	%
\end{assumption}

The Bernoulli model in
Assumption~\ref{assump:bernoulli-support} specifies only the support pattern
and is not a generative model for feasible points on \(\St_\Delta(n,r)\). It is
used only to estimate how often the support-level row-cover sufficient condition in
Assumption~\ref{assump:row-cover} occurs. Such support information alone does
not guarantee feasibility of the perturbed Stiefel constraint, and the
intersection \(\St_\Delta(n,r)\cap\M_g(Y)\) may be empty. Accordingly,
feasibility is imposed separately in Proposition~\ref{prop:verify-intersection},
which assumes \(\St_\Delta(n,r)\cap\M_g(Y)\neq\emptyset\) and then studies the
intersection geometry under this condition; verifiable support-level conditions
that also guarantee feasibility are left for future work.

Thus Proposition~\ref{prop:row-cover-bernoulli} should be read
as a probability bound for the support-level sufficient condition in
Assumption~\ref{assump:row-cover}, conditional on using such supports in a
feasible intersection model, rather than as a feasibility guarantee.
\begin{proposition}\label{prop:row-cover-bernoulli}
	 Under Assumption~\ref{assump:bernoulli-support}, the row-cover condition in Assumption~\ref{assump:row-cover} holds with probability at least
	\[
	1-\Bigl(\binom{r}{2}+r\,\mathbf 1_{\{\Delta\neq 0\}}\Bigr)e^{-n q_w},
	\]
	where
	\begin{equation}\label{eq:qw}
		q_w:=
		\begin{cases}
			p^2(1-p)^{r-2}+(r-2)p^3(1-p)^{r-3}, & \Delta=0,\\[4pt]
			\min\Bigl\{\,p^2(1-p)^{r-2}+(r-2)p^3(1-p)^{r-3},\ \ p(1-p)^{r-1}\Bigr\}, & \Delta\neq 0.
		\end{cases}
	\end{equation}
	In particular, for any $\delta\in(0,1)$, if
	\[
	n\ \ge\ \frac{\log\!\Bigl(\binom{r}{2}+r\,\mathbf 1_{\{\Delta\neq 0\}}\Bigr)+\log(1/\delta)}{q_w},
	\]
	then Assumption~\ref{assump:row-cover} holds with probability at least $1-\delta$.
\end{proposition}

The proof is a union-bound argument over uncovered column
pairs and missing one-column witnesses and is deferred to
Appendix~\ref{app:section4-examples}.

\begin{remark} 
	Our focus is on problems whose solutions are \emph{sparse} (e.g., with localized/compactly
	supported columns), so the Bernoulli parameter $p$ is not intended to be large. In particular,
	the relevant regime is when the expected row sparsity $\mathbb{E}|J_\ell|=rp$ is $O(1)$, rather than
	of order $r$.   
	More concretely, in the sparse scaling $p=c/r$, one has the
	approximation
	\[
	q_w \approx \frac{c^2(1+c)e^{-c}}{r^2}.
	\]
Hence the sample size condition in Proposition~\ref{prop:row-cover-bernoulli} becomes
\[
n \;\gtrsim\; 
\frac{r^2}{c^2(1+c)e^{-c}}
\Bigl(\log\binom{r}{2}+\log(1/\delta)\Bigr)
=\Theta\!\bigl(r^2\log(r/\delta)\bigr).
\]

	In contrast, when $p$ is close to $1$ the model describes a dense Stiefel-type regime where
	the ``sparse intersection'' viewpoint is no longer essential: the active manifold constraint
	becomes vacuous. Finally, in applications such as {Compressed
		Modes}, numerical solutions are often observed to exhibit highly localized supports, which is
	consistent with row supports of size on the order of a few (e.g., $2$--$3$) and thus aligns
	well with the 2/3-cover condition. For entrywise-sparse PCA formulations with an $\ell_1$ regularizer,
	the same support-level condition can also be mild when $n\gg r$ (e.g., $n\gtrsim 2000$) and a small target rank (e.g., $r<10$), so the
	high-probability requirement in Proposition~\ref{prop:row-cover-bernoulli} could be satisfied
	for the sparse regimes of interest.
	In the numerical section, we therefore report direct
	support/rank diagnostics for the supports identified by MIX, rather than
	relying solely on the Bernoulli model.
	
\end{remark}

%% file: shared/section_optimization_on_intersection.tex
\section{Optimization on the intersection manifold}\label{sec:opt_on_intersection}

In this section, we specialize to the case
\(g(X)=\lambda\|X\|_1\), and let $\M_g(\bar Y)$ denote the active manifold defined in \eqref{eq:manifold_2}. We develop the differential-geometric ingredients needed for optimization on the intersection manifold
\[\overline{\M}_{\Delta}(\bar Y) =\M_g(\bar Y)\cap \St_\Delta(n,r),\]
including orthogonal projections onto tangent spaces, the Riemannian gradient and Hessian, and practical retractions. In particular, we use the alternating-projection result developed in \cite{chen2026alternating}\cite{zheng2026rnslra} to construct computable second-order retractions on the intersection manifold.
Yang et al. \cite{yang2026optimizationintersectionmanifolds}
recently also develop a first-order geometric framework for optimization over the intersection of
manifolds.
 
\subsection{Projection onto the affine tangent intersection}\label{sec:projection_to_inter_tangent}

In this subsection, we work in vectorized coordinates and write $\St_\Delta$ for $\St_\Delta(n,r)$ for simplicity. Let $x:=\vvec(X)$ and $\bar y:=\vvec(\bar Y)$. For $X\in \St_\Delta\cap \M_g(\bar Y)$, we study the projection of a vector  $c\in\R^{nr}$ onto the affine tangent intersection
\[
\T_\cap^b(x):=(b+\vvec(\T_X\St_\Delta))\cap \vvec(\T_X\M_g(\bar Y)).
\]
{\color{blue}Recall from Section~\ref{sec:intersection_property} that $B_x$ and
$E_x$ are basis matrices of $\vvec(\rmN_X\St_\Delta)$ and
$\vvec(\T_X\M_g(\bar Y))$, respectively, and that
$A_x=B_x^\top E_x$ by \eqref{eq:matrix_A}. Since the tangent space of
$\M_g(\bar Y)$ is the fixed coordinate subspace determined by the support
pattern, $E_x$ is independent of $x$, and we simply write it as $E$.}
In the fixed support coordinates, this projection is equivalent to
\begin{equation}\label{prob:projection_eq}
	\begin{aligned}
		&\min_{a\in\R^s}\frac12\|a-E^\top c\|^2,\\
		&\st \quad A_x a=B_x^\top b.
	\end{aligned}
\end{equation}
The derivation follows from the equality-constrained least-squares formulation and is given in Appendix~\ref{app:proof-projection-intersection}.

For generic perturbations $\Delta$, the transversality condition \eqref{eq:transv-Mg-StDelta} implies that, at every point
$X\in \overline{\M}_\Delta(\bar Y)$, the matrix $A_x$ has full row rank $r(r+1)/2$. For nongeneric \(\Delta\), transversality may fail.  Proposition~\ref{prop:verify-intersection} shows that Assumption~\ref{assump:row-cover} is a support-based sufficient condition for the constant-rank property. Under the Bernoulli support model, Proposition~\ref{prop:row-cover-bernoulli} gives a high-probability bound for this support condition.  

We now exploit the structure of $A_x$ under a fixed support pattern. By Proposition~\ref{lem:M_1 and M_2 are not transverse}, if two column supports are disjoint, say
$\mathcal I_\ell^X\cap \mathcal I_k^X=\emptyset$, then
$XS_{\ell k}\in \rmN_X\St_\Delta\cap \rmN_X\M_g(\bar Y)$.
Since $\vvec(XS_{\ell k})$ corresponds to a row of $B_x^\top$, the same row of
$A_x=B_x^\top E$ is zero. Thus every missing edge in the support graph produces a zero row of $A_x$.

Let $A_{\gamma(\bar Y)}\in\R^{\gamma(\bar Y)\times s}$ denote the submatrix of $A_x$ obtained by removing all zero rows. Correspondingly, let
\begin{equation}\label{eq:U_gamma}
	U_{\gamma(\bar Y)}^\top\in\R^{\gamma(\bar Y)\times r^2}
\end{equation}
be the submatrix of $U_r^\top$ consisting of the same rows as $A_{\gamma(\bar Y)}$, and define
\begin{equation}\label{eq:B_gamma}
	B_{\gamma(\bar Y)}^\top:=U_{\gamma(\bar Y)}^\top(I_r\otimes X^\top).
\end{equation}
Then, we have
\(
A_{\gamma(\bar Y)}=B_{\gamma(\bar Y)}^\top E.
\)

Under Assumption~\ref{assump:intersection}, it follows from Lemma~\ref{lem:simple_dim_lem} and the dimension formula for
$\rmN_X\St_\Delta\cap \rmN_X\M_g(\bar Y)$ that
\[\rank(A_x)=\gamma(\bar Y).\]
Hence $A_{\gamma(\bar Y)}$ has full row rank, and the projection problem reduces to
\[
\min_{a\in\R^s}\frac12\|a-E^\top c\|^2
\qquad \text{s.t.}\qquad
A_{\gamma(\bar Y)}a=B_{\gamma(\bar Y)}^\top b.
\]
Therefore,
\begin{equation}\label{eq:projection_vector_g_old}
	\Proj_{\T_\cap^b(x)}(c)
	=
	E\Bigl(E^\top c-
	A_{\gamma(\bar Y)}^\dagger\bigl(A_{\gamma(\bar Y)}E^\top c-B_{\gamma(\bar Y)}^\top b\bigr)\Bigr),
\end{equation}
where \(A_{\gamma(\bar Y)}^\dagger
=
A_{\gamma(\bar Y)}^\top\bigl(A_{\gamma(\bar Y)}A_{\gamma(\bar Y)}^\top\bigr)^{-1}\)
is the Moore--Penrose inverse of $A_{\gamma(\bar Y)}$.

In particular, for $b=0$, let
\[
M_x:=EE^\top,\qquad
G_x:=B_{\gamma(\bar Y)}^\top M_x,\qquad
\Gamma_x:=G_x^\top(G_xG_x^\top)^{-1}G_x.
\]
Then the orthogonal projection onto the tangent space of the intersection manifold is
\begin{equation}\label{eq:projection_vector_g}
	\Proj_{\T_X\overline{\M}_\Delta(\bar Y)}c
	=
	M_xc-G_x^\top(G_xG_x^\top)^{-1}G_xc
	=
	(M_x-\Gamma_x)c.
\end{equation}
\emph{Complexity.} Under a fixed support pattern, the projector $M_x$ is block diagonal:
\[
M_x=\diag(D_1,\dots,D_r),
\]
where each $D_j\in\R^{n\times n}$ is the diagonal matrix associated with the support of the
$j$th column. Hence
\[
K_x:=(I_r\otimes X^\top)M_x(I_r\otimes X)
=\diag(X^\top D_1X,\dots,X^\top D_rX)
\]
is also block diagonal, and
\[
G_xG_x^\top
=
U_{\gamma(\bar Y)}^\top K_x U_{\gamma(\bar Y)}.
\]
If $s_j$ denotes the support size of the $j$th column and $s:=\sum_{j=1}^r s_j$, then forming
$K_x$ costs $O(sr^2)$, while extracting $G_xG_x^\top$ from $K_x$ costs $O(\gamma(\bar Y)^2)$.
Its Cholesky factorization costs $O(\gamma(\bar Y)^3)$. Therefore, for a fixed $x$, the
one-time preprocessing cost is $O(sr^2+\gamma(\bar Y)^3)$. After this preprocessing, each subsequent projection in \eqref{eq:projection_vector_g} costs
$O(sr+\gamma(\bar Y)^2)$.

\subsection{Riemannian gradient and Riemannian Hessian}

Equipped with the Riemannian metric induced by the ambient Euclidean space, we next characterize the Riemannian gradient and Riemannian Hessian of a smooth function \(f\) on the intersection manifold \(\overline{\M}_{\Delta}(\bar Y)\). For any \(X\in \overline{\M}_{\Delta}(\bar Y)\), the Riemannian gradient is given by
\begin{equation}\label{eq:riemannian grad}
	\grad f(X)=\Proj_{\T_X\overline{\M}_{\Delta}(\bar Y)}\nabla f(X),
\end{equation}
where \(\nabla f(X)\) denotes the Euclidean gradient, and  {the tangent projection is computed by \eqref{eq:projection_vector_g}, the \(b=0\) case of \eqref{prob:projection_eq}.}

Using the Levi--Civita connection introduced in \eqref{def:Levi-Civita}, the Riemannian Hessian at \(X\) along a tangent vector \(\eta\in \T_X\overline{\M}_{\Delta}(\bar Y)\) is
\[
\Hess f(X)[\eta]
=\Proj_{\T_X\overline{\M}_{\Delta}(\bar Y)}
\bigl(\mathrm D(\grad f)(X)[\eta]\bigr).
\]
Here \(\mathrm D(\grad f)(X)[\eta]\) is the ambient directional derivative of the Riemannian gradient field, and the outer projection corresponds to the Levi--Civita connection on \(\overline{\M}_{\Delta}(\bar Y)\).

To compute the Riemannian Hessian, we need the differential of the tangent projector. This is naturally related to the second fundamental form and the Weingarten map; see \cite[Def.5.48]{boumal2023introduction}.

Applying \cite[Eq.~(5.34)]{boumal2023introduction} to the embedded manifold $\overline{\M}_\Delta(\bar Y)$, we obtain
\begin{equation}\label{hess_1}
\Hess f(X)[\eta]
=
\Proj_{\T_X\overline{\M}_{\Delta}(\bar Y)}
\Bigl(\nabla^2 f(X)[\eta]+\Proj_\eta(\nabla f(X))\Bigr),
\end{equation}
where \(\Proj_\eta:=\mathrm D\bigl(X\mapsto \Proj_{\T_X\overline{\M}_\Delta(\bar Y)}\bigr)(X)[\eta]\).
Let $g_x:=\vvec(\nabla f(X))$ denote the vectorized Euclidean gradient.
Since \(G_x=B_{\gamma(\bar Y)}^\top M_x\) and
\(B_{\gamma(\bar Y)}^\top=U_{\gamma(\bar Y)}^\top(I_r\otimes X^\top)\),
we define \[\dot G_x[\eta]
:=\mathrm D G_x(x)[\eta]
=
U_{\gamma(\bar Y)}^\top(I_r\otimes \eta_m^\top)M_x,\]
where $\eta_m\in\R^{n\times r}$ is the matrix representation of the vector $\eta\in\R^{nr}$.
Differentiating the projector \eqref{eq:projection_vector_g} and projecting away the normal component gives the explicit Hessian formula; the derivation is given in Appendix~\ref{app:proof-intersection-hessian}.
\begin{equation}\label{eq:Rieman_Hessian}
\Hess f(X)[\eta]
=
\Proj_{\T_X\overline{\M}_{\Delta}(\bar Y)}
\Bigl(
\nabla^2 f(X)[\eta]
-
\operatorname{mat}\!\Big(
\dot G_x[\eta]^\top(G_xG_x^\top)^{-1}G_x g_x
\Big)
\Bigr),
\end{equation}
where $\operatorname{mat}$ denotes the inverse operation of $\vvec$.

\emph{Complexity.} For a fixed $X$, the factorization of $G_xG_x^\top$ is a one-time cost. After this preprocessing, one application of the Riemannian Hessian in \eqref{eq:Rieman_Hessian} has the same order of complexity as two projections onto $\T_X\overline{\M}_{\Delta}(\bar Y)$, together with one Euclidean Hessian--vector product. Consequently, similar to the discussion of \eqref{eq:projection_vector_g}, one iteration of Algorithm~\ref{alg:tcg} costs $O(sr+\gamma(\bar Y)^2)$, in addition to one Euclidean Hessian--vector product.

\subsection{Retraction on the intersection manifold}

After computing a search direction in the tangent space, one needs a mechanism to map it back to the manifold. The exponential map is the canonical choice, but it is often too expensive to evaluate in practice. Retractions provide computable approximations of the exponential map and are therefore standard in manifold optimization. For embedded submanifolds of $\E$, many retractions are available, in particular projection-type retractions; see \cite{absil2012projection}. For the intersection manifold considered here, however, constructing such retractions directly is not straightforward. Fortunately, the recent work \cite{chen2026alternating} shows that alternating-projection-type algorithms induce second-order retractions on intersection manifolds. We briefly recall this construction below.

\paragraph{\bf{Alternating-projection-type retractions}}
Instead of computing the exact projection onto the intersection manifold,
alternating-projection-type methods generate approximations to this projection.
As shown in \cite{chen2026alternating}, the corresponding limiting maps can
induce retractions on manifold intersections.  More concretely, for our two cleanly intersecting manifolds $\M_g, \St_{\Delta}$, one alternating-projection-type iteration has the form \[\varphi := \phi_1\circ \phi_2,\]
where $\phi_1:\mathbb R^{n\times r}\to \M_g(\bar Y)$
and $\phi_2: {\mathbb R^{n\times r}\to \mathbb R^{n\times r}}$ are either an exact or an inexact projection mapping. Since $\M_g(\bar Y)$ is easy to project, we use the projection operator.
For a point $Y\in \overline{\M}_\Delta(\bar Y)$ and a tangent vector $\eta\in \T_Y\overline{\M}_\Delta(\bar Y)$,
we associate the one-step map $\varphi$ with the limiting map
\begin{manualtaggedequation}
\begin{equation}\label{eq:limiting_alter_proj}\tag{Alt-R}
		\psi(Y,\eta)
		:= \lim_{k\to\infty}\varphi^{k}(Y+\eta).
\end{equation}
\end{manualtaggedequation}
\cite{chen2026alternating} shows that $\psi$
is a second-order retraction for a class of alternating-projection methods.
Examples include the classical Alternating Projections method
(APM)~\cite{lewis2008alternating,andersson2013alternating,drusvyatskiy2015transversality,noll2016local}, a broader family of
	inexact Alternating Projections methods
(IAPM)~\cite{kruger2016regularity,drusvyatskiy2019local}, and
quadratically convergent variants such as NewtonSLRA~\cite{schost2016quadratically}
and RN-SLRA~\cite{zheng2026rnslra}.

We describe the alternating-projection-type retractions below.
\begin{enumerate}
    \item\textbf{Alternating projections method (APM)} 
Denote \(D=I_r+\Delta\).
The formulation of the projection onto $\St_\Delta$ is given by
\begin{equation} \label{eq:proj_closed_StDelta}
\Proj_{\St_\Delta}(Y)
= Y D^{1/2} \left( D^{1/2} Y^\top Y D^{1/2} \right)^{-1/2} D^{1/2}.
\end{equation}
The derivation is given in
Appendix~\ref{sec:append_projection_to_St}. 
Therefore, with
\begin{equation}\label{eq:APM}
    \varphi_{\mathrm{APM}}(Y)= S_{\bar Y}\odot (\Proj_{\St_\Delta}(Y)),
\end{equation}
the limiting map $\psi$ in \eqref{eq:limiting_alter_proj} defines a second-order retraction; see~\cite[Theorem 3]{chen2026alternating}. However,
the inverse square root of $D^{1/2} Y^\top Y D^{1/2}$ will be expensive when $r$ is large. We seek an approximation to it.\\
\smallskip
\emph{Complexity.} Each evaluation of $\varphi_{\mathrm{APM}}$ costs $O(nr^2+r^3)$ flops.
\item \textbf{Inexact Alternating Projections method (IAPM)} For the generalized Stiefel manifold
\(
\St_{\Delta}(n,r)
\)
we use the following first-order approximation to $\Proj_{\St_\Delta}$:
\(\phi_2(Y)=Y\bigl(I-S(Y)\bigr)\),
where $S(Y)\in\mathbb S^r$ is the unique symmetric solution of the Sylvester equation
\(DS(Y)+S(Y)D=Y^\top Y-D\).

Given $D$, we can precompute its eigendecomposition
\(
D=U\Lambda U^\top,  \Lambda=\mathrm{Diag}(\lambda_1,\ldots,\lambda_r),
\)
and define \(K_{ij}:=\frac{1}{\lambda_i+\lambda_j}\).
Then, for each $Y$, letting $Z:=YU$, the Sylvester solution can be applied without explicitly forming $S(Y)$:
\begin{equation}\label{eq:Sylvester}
   \widehat S = K\odot\bigl(Z^\top Z-\Lambda\bigr),\qquad
\phi_2(Y)=Y- Y(U\widehat S U^\top). 
\end{equation}
\emph{Complexity.} After a one-time $O(r^3)$ preprocessing for $D$, each evaluation of $\phi_2$ costs $O(nr^2+r^3)$ flops, which is more efficient than the APM update in \eqref{eq:APM}.
The derivation is given in Appendix~\ref{app:sylvester_corrected_phi2}.
 \begin{remark}
     When $D=I_r$, the Sylvester equation reduces to
\(
2S(Y)=Y^\top Y-I,
\)
hence
\[
\phi_2(Y)
=
Y\Bigl(I-\frac12(Y^\top Y-I)\Bigr)
=
\frac12\,Y(3I-Y^\top Y),
\]
which is exactly the standard Newton-Schulz iteration \cite{kovarik1970some,bjorck1971iterative}.
 \end{remark}

Therefore, the IAPM method we use is given as follows:
\begin{equation}\label{eq:IAPM}
    \varphi_{\mathrm{IAPM}}(Y) = S_{\bar Y} \odot\Big(Y(I- U\widehat S U^\top)\Big).
\end{equation}
To guarantee the convergence of the iteration \eqref{eq:IAPM}, the inexactness of $\phi_2$ should satisfy
\begin{equation}\label{eq:inexact_projection_St}
     \|\phi_2(Y)-\Proj_{\St_\Delta}(Y)\|
    = o\big(\dist(Y,\St_\Delta)\big),
    \quad \text{as } Y\to X,\ \ X\in \mathcal U_{\bar Y}\cap \St_\Delta,
\end{equation}
where $\mathcal U_{\bar Y} $ is a neighborhood of $\bar Y$ in $\E^n.$
The Sylvester-corrected map \(\phi_2\) satisfies the second-order residual expansion required by the alternating-projection framework; the precise statement and proof are given in Lemma~\ref{lem:expansion-phi} in Appendix~\ref{append_subsec:proof_lemma}. Consequently, under the clean-intersection and regularity assumptions of \cite[Assumption~3 and Theorem~3]{chen2026alternating}, the limiting map induced by \(\varphi_{\mathrm{IAPM}}\) defines a second-order retraction.

\item \textbf{Regularized NewtonSLRA.}
Although both APM and IAPM converge linearly to a point on
$\overline{\M}$, they may be slow when the angle between the two manifolds is
small. We use NewtonSLRA~\cite{schost2016quadratically} and its regularized
variant RN-SLRA~\cite{zheng2026rnslra} to accelerate the alternating-projection scheme
by projecting onto a linearized intersection at the current Stiefel point.

Given $Y_j\in\mathbb R^{n\times r}$, we first compute
\(X_j=\Proj_{\St_{\Delta}}(Y_j)\in\St_{\Delta}(n,r)\)
and write $x_j=\mathrm{vec}(X_j)\in\mathbb R^{nr}$. We then compute a
regularized projection of $x_j$ onto the affine linearized intersection
\[
    \T_{\cap}^b(x_j)
    :=
    \bigl(x_j+\vvec(\T_{X_j}\St_\Delta)\bigr)
    \cap \vvec(\M_g(\bar Y)).
\]
This is the affine version of the construction of $\T_\cap^b(x)$ above, with $b=x_j$ and the fixed coordinate subspace determined by $\M_g(\bar Y)$.
Unlike \eqref{prob:projection_eq}, this affine intersection may be
empty when transversality fails. To obtain a well-defined update, we use a
Levenberg--Marquardt-type regularization of the linearized projection
problem and solve
\begin{equation}\label{eq:newtonslra_penalized}
\min_{a\in\mathbb R^s}
\frac12\|a-E^\top x_j\|^2
+
\frac{1}{2\lambda_{\rm slra}}
\|A_{x_j}a-B_{x_j}^\top x_j\|^2,
\qquad \lambda_{\rm slra}>0,
\end{equation}
where $E\in\mathbb R^{nr\times s}$ is an orthonormal basis of
$\vvec(\M_g(\bar Y))$, $B_{x_j}\in\mathbb R^{nr\times q}$,
$q=r(r+1)/2$, is a basis of $\rmN_{X_j}\St_\Delta$, and
$A_{x_j}=B_{x_j}^\top E$. The solution is given by
\[
    a_j^{\lambda_{\rm slra}}
    =
    E^\top x_j-A_{x_j}^\top u_j,
    \qquad
    (A_{x_j}A_{x_j}^\top+\lambda_{\rm slra}I_q)u_j
    =
    A_{x_j}E^\top x_j-B_{x_j}^\top x_j .
\]
The RN-SLRA update is then
\[
    y_{j+1}
    :=
    E a_j^{\lambda_{\rm slra}}
    =
    E\bigl(E^\top x_j-A_{x_j}^\top u_j\bigr)
    \in \vvec(\M_g(\bar Y)).
\]
When $\T_{\cap}^b(x_j)$ is nonempty, the regularized projection converges to
the exact Euclidean projection onto $\T_{\cap}^b(x_j)$ as
$\lambda_{\rm slra}\downarrow0$.

The original NewtonSLRA method is locally quadratically convergent under
transversality. In the clean-intersection setting, transversality may fail, and
the linearized intersection can become ill-conditioned or empty. The
regularized step~\eqref{eq:newtonslra_penalized} avoids this difficulty and is
supported by the convergence theory of RN-SLRA~\cite{zheng2026rnslra}: it is
locally linearly convergent under weak intersection conditions and recovers
higher-order convergence under transversality. For a fixed regularization
parameter, the corresponding limiting map also admits a second-order retraction
on the clean intersection manifold.

\smallskip
\emph{Complexity.}
One RN-SLRA step consists of two parts. First, computing
$X_j=\Proj_{\St_\Delta}(Y_j)$ costs $O(nr^2+r^3)$. Second, solving the
$q\times q$ linear system
\[
(A_{x_j}A_{x_j}^{\top}+\lambda_{\rm slra}I_q)u_j
=
A_{x_j}E^\top x_j-B_{x_j}^{\top}x_j
\]
costs $O(sr^2+r^6)$, where $O(sr^2)$ comes from forming
$A_{x_j}A_{x_j}^\top$ using the block structure derived above, and
$O(r^6)$ is the cost of solving the resulting $q\times q$ system. Therefore,
one RN-SLRA step has overall complexity $O(nr^2+sr^2+r^6)$.

\end{enumerate}
In practice, we use the TAPR method proposed in
\cite{chen2026alternating}, which is a hybrid method combining APM, IAPM, and
NewtonSLRA-type corrections. Given $Y\in\overline{\M}_k$ and
$\eta\in\T_Y\overline{\M}_k$, let $\{Z_j\}$ denote the inner sequence generated
by TAPR with initial point $Z_0=Y+\eta$. In practice, we use APM while $\|Z_{j+1}-Z_j\|\ge 10^{-2}$, switch to IAPM
when $\|Z_{j+1}-Z_j\|\in[10^{-5},10^{-2}]$, and use RN-SLRA when $\|Z_{j+1}-Z_j\|<10^{-5}$. In our experiments,
no more than 30 APM/IAPM steps and  {2 RN-SLRA steps are} typically
sufficient. 

{\color{blue}
	For the later convergence analysis, $\Retr^{\rm TAPR}$ denotes the TAPR retraction in \cite[Theorem~5]{chen2026alternating}, where the numbers of APM and IAPM iterations are fixed independently of the input. The adaptive switching strategy used in the experiments is regarded as a practical approximation.
}

The following results follow from \cite{chen2026alternating,zheng2026rnslra}.
\begin{proposition}\label{fact:ret_Tapr}
Under the local intersection assumptions in
Theorem~\ref{thm:Mg-intersection-submanifold}, including the unperturbed case
$\Delta=0$, the APM, IAPM \cite{chen2026alternating}, and fixed-regularization RN-SLRA \cite{zheng2026rnslra} phases used in
TAPR induce second-order retractions on
$\overline{\M}_\Delta(\bar Y)$. Consequently, the TAPR limiting map with fixed
$\lambda_{\rm slra}>0$ defines a second-order retraction in the local
clean-intersection regime.

If, in addition, the transversality condition~\eqref{eq:transv-Mg-StDelta}
holds, then the standard NewtonSLRA phase is also second order \cite{chen2026alternating}. Thus the same
conclusion holds when the TAPR map uses the standard NewtonSLRA phase. This
applies, in particular, to generic perturbations
$\Delta\in\mathscr P_{\rm gen}$.
\end{proposition}

%% file: shared/section_algorithm_intersection.tex
\section{A hybrid algorithm: MIX}\label{sec:alg_on_intersection}

In this section, we introduce a new algorithm tailored to problem \((P_\Delta)\).
Its global descent and KKT-residual guarantees do not require the generic
perturbation condition and, in particular, yield vanishing KKT residual for the
original problem when \(\Delta=0\). The condition \(\Delta\in P_{\rm gen}\) is
used only as a sufficient way to verify the transversality condition needed
for finite identification; it can be replaced by direct transversality,
including when \(\Delta=0\).

The design of our algorithm is inspired by \cite{bareilles2023newton}, but the
Stiefel-type constraint $\St_\Delta$ makes the constrained setting substantially different
from the Euclidean one. To exploit the sparsity pattern identified by the ManPG step while
maintaining near-orthogonality, we work with both the fixed perturbed constraint $\St_\Delta$
and an iteration-dependent Stiefel-type constraint determined by the Gram matrix $Y_k^\top Y_k$
of the ManPG predictor $Y_k$. Although the proposed framework can be extended to more general equality-constrained manifolds, the intersection structure between the identification manifold of the nonsmooth term and the constraint manifold is generally difficult to characterize; therefore, we defer a brief discussion of such extensions to the appendix.

According to Theorem~\ref{thm:id-strong-crit}, near a nondegenerate critical point
$X^*$, the support of $\mathcal T_t(X)$ remains invariant. This suggests the following local smooth optimization problem:
\[
\min_{X\in \M_\Delta(X^\star)} \bar F(X),
\qquad
\bar F:=F|_{\M_\Delta(X^\star)} .
\]
Such a problem can be efficiently handled by second-order Riemannian methods under suitable
local assumptions. However, the support of the local solution is not known in advance. In
\cite{bareilles2023newton}, the support is estimated by a proximal-gradient-type step and
iteratively refined. In the present constrained setting, however, the set
\(
    \St_\Delta(n,r)\cap \M_g(\mathcal T_t(X_k))
\)
may be empty during the early iterations.

To avoid this difficulty, with 
\(V_k=\mathcal T_t(X_k)-X_k\) and
\(Y_k=\mathcal T_t(X_k)=X_k+V_k\), we have
\(Y_k^\top Y_k=I_r+\Delta+V_k^\top V_k\).
Define
\(\Delta_k=\Delta+V_k^\top V_k\) and
\(D_k=I_r+\Delta_k=Y_k^\top Y_k\).
Then $Y_k$ lies exactly on the iteration-dependent Stiefel-type manifold
\[
    \St_{\Delta_k}(n,r)
    :=
    \left\{
    X\in\mathbb R^{n\times r}\ \middle|\ X^\top X=D_k
    \right\}.
\]

We now present the proposed hybrid method in Algorithm~\ref{alg:manpg-mix}, termed the
\underline{M}anifold \underline{i}dentification pro\underline{x}imal Newton method (MIX).
The MIX algorithm combines a globally safeguarded ManPG-type step with a local Riemannian
Newton-CG step constructed on an identified smooth geometric model.

At iteration $k$, we first solve the tangent subproblem~\eqref{eq:tangent_problem} to obtain
a direction $V_k\in\T_{X_k}\St_\Delta$ and set $Y_k=X_k+V_k$. The quantity $\normfro{V_k}/t_k$
is used as the stationarity measure. If $\normfro{V_k}/t_k<\epsilon_{\rm stat}$, the algorithm terminates and
returns $Y_k$.
Indeed,
\(
    Y_k^\top Y_k-(I_r+\Delta)=V_k^\top V_k,
\)
and hence the feasibility violation of $Y_k$ is of order $O(\normfro{V_k}^2)$.

When $\normfro{V_k}<\delta$, where $\delta\in(0,1)$ is the threshold for activating the
identification step, we exploit the local structure revealed by the ManPG predictor $Y_k$.
Specifically, we identify the smooth intersection manifold
\[\overline{\M}_k=\St_{\Delta_k}\cap \M_g(Y_k),\]
where $\St_{\Delta_k}$ is the iteration-dependent Stiefel-type constraint defined above.
Although ${\M}_g(Y_k)$ need not coincide with the limiting active manifold $\M_g(X^*)$,
it provides a smooth low-dimensional model around $Y_k$ on which second-order information
can be used to accelerate the ManPG step.  The well-definedness of this intersection manifold
follows from Theorem~\ref{thm:Mg-intersection-submanifold} whenever
Assumption~\ref{assump:intersection} holds. Section~\ref{sec:intersection_property}
gives deterministic and probabilistic support conditions under which this assumption
is satisfied, including the unperturbed case $\Delta=0$ in the stated regimes.

\begin{algorithm}[!t]
\caption{A \underline{M}anifold \underline{i}dentification pro\underline{x}imal Newton method (MIX)}
\label{alg:manpg-mix}
\begin{algorithmic}[1]
\Require target tolerance $\epsilon>0$, perturbation $\Delta$ with
$I_r+\Delta\succ0$ and $\normfro{\Delta}<\epsilon/2$, stationarity tolerance
$\epsilon_{\rm stat} = O(\epsilon-\normfro{\Delta})$,  initial point $X_0\in\St_\Delta(n,r)$,
initial stepsize $t_0\in(0,t_{\max}]$, upper stepsize safeguard
$t_{\max}>0$, initial regularization scale $\lambda_0>0$,
exponent $\sigma\in[0,1]$, threshold $\delta\in(0,1)$
satisfying \(0<\delta\le\frac12\sqrt{\lambda_{\min}(I_r+\Delta)}\), parameters 
$\beta_t^\uparrow>1$, $\beta_t^\downarrow\in(0,1)$,
$\beta_\lambda^\uparrow>1$, $\beta_\lambda^\downarrow\in(0,1)$.
\State Set $k=0$.
\While{true}
    \State Solve the tangent subproblem~\eqref{eq:tangent_problem} with stepsize
    $t_k$ to obtain $V_k\in\T_{X_k}\St_\Delta$, and set $Y_k=X_k+V_k$.
    \If{$\normfro{V_k}/t_k \le \epsilon_{\rm stat}$}
        \State \Return $Y_k$.
    \EndIf

\State Set $\mathrm{tag}\gets\texttt{manpg}$,
$\mathrm{cg\_status}\gets\texttt{not-invoked}$, and
$\ell_{\rm N}\gets\infty$.
\If{$\normfro{V_k}<\delta$} 
    \State Identify $\overline{\M}_k$ as in~\eqref{eq:M_k}.
    \State
    $(X_k^{\rm N},\mathrm{tag},\mathrm{cg\_status},\ell_{\rm N})
    \gets
    \textsc{NewtonCG}(X_k,Y_k,V_k,t_k,\overline{\M}_k,\lambda_k,\sigma)$.
    \Comment{Algorithm~\ref{alg:newton-cg}}
\EndIf

    \If{$\mathrm{tag}=\texttt{accepted}$}
        \State $X_{k+1}=X_k^{\rm N}$.
    \Else
        \State Use the ManPG line-search procedure in~\cite{chen2020proximal}
        to obtain $\hat\alpha_k$.
        \State $X_{k+1}=\Proj_{\St_\Delta}(X_k+\hat\alpha_k V_k)$.
    \EndIf

    \State Update $(t_{k+1},\lambda_{k+1})$ according to~\eqref{eq:mix-update-rule}.
    \State $k\gets k+1$.
\EndWhile
\end{algorithmic}
\end{algorithm}

\begin{algorithm}[!t]
\caption{\textsc{NewtonCG}}
\label{alg:newton-cg}
\begin{algorithmic}[1]
\Require $X_k,Y_k,V_k,t_k,\overline{\M}_k,\lambda_k,\sigma$.
\Ensure A candidate $X_k^{\rm N}$, a tag, a tCG status $\mathrm{cg\_status}$,
and the number $\ell_{\rm N}$ of Newton backtracking steps.
\State Set line-search parameters $\bar\tau\in(0,1/2)$, $c_{\rm N}\in(0,1/2)$,
$\varsigma\in(0,1)$, $\alpha_{\min}>0$ chosen sufficiently small.
\State Set $X_k^{\rm N}\gets X_k$, $\mathrm{tag}\gets\texttt{newton-fail}$,
and $\ell_{\rm N}\gets\infty$.
\State Compute
$(w_k,\mathrm{cg\_status})
=\mathrm{tCG}(\overline{\M}_k,Y_k,\bar F_k,\lambda_k,\sigma)$.
\Comment{Algorithm~\ref{alg:tcg}}
\If{$\mathrm{cg\_status}\notin\{\texttt{linear},\texttt{superlinear}\}$}
    \State \Return $(X_k^{\rm N},\mathrm{tag},\mathrm{cg\_status},\ell_{\rm N})$.
\EndIf
\State Set $\alpha=1$, $\ell=0$, and $\mathrm{model\_ok}\gets\mathrm{false}$.
\While{$\alpha\ge\alpha_{\min}$}
    \State Generate $Z_k=\Retr_{Y_k}^{\rm TAPR}(\alpha w_k)$ by TAPR \cite{chen2026alternating}.
    \If{no valid trial point is generated} \Comment{TAPR is locally defined.}
        \State $\alpha\gets\varsigma\alpha$ and $\ell\gets\ell+1$.
        \State \textbf{continue}.
    \EndIf

  \If{\eqref{eq:newton-armijo} fails at $Z_k$}
  \State $\alpha\gets\varsigma\alpha$ and $\ell\gets\ell+1$.
  \State \textbf{continue}.
  \EndIf

  \State $\mathrm{model\_ok}\gets\mathrm{true}$ and
  $X_k^{\rm trial}=\Proj_{\St_\Delta}(Z_k)$.
  \If{$F(X_k^{\rm trial})\le F(X_k)-\bar\tau\normfro{V_k}^2/t_k$}
  \State $X_k^{\rm N}\gets X_k^{\rm trial}$,
  $\mathrm{tag}\gets\texttt{accepted}$, and
  $\ell_{\rm N}\gets\ell$.
  \State \Return $(X_k^{\rm N},\mathrm{tag},\mathrm{cg\_status},\ell_{\rm N})$.
  \Else
  \State $\mathrm{tag}\gets\texttt{projection-fail}$ and $\ell_{\rm N}\gets\ell$.
  \State \Return $(X_k^{\rm N},\mathrm{tag},\mathrm{cg\_status},\ell_{\rm N})$.
  \EndIf

\EndWhile

\State $\mathrm{tag}\gets\texttt{newton-fail}$. \Comment{Local smooth model is not good.}
\State $\ell_{\rm N}\gets\ell$.
\State \Return $(X_k^{\rm N},\mathrm{tag},\mathrm{cg\_status},\ell_{\rm N})$.
\end{algorithmic}
\end{algorithm}

On $\overline{\M}_k$, define the smooth restricted  objective
\[
    \bar F_k(Y):=F|_{\overline{\M}_k}(Y).
\]
The corresponding Riemannian gradient and Hessian of $\bar F_k$ are given by
\eqref{eq:riemannian grad} and \eqref{eq:Rieman_Hessian}, respectively. After the
identified manifold $\overline{\M}_k$ is constructed, we use a Riemannian
Newton-CG step (Algorithm~\ref{alg:newton-cg}) as a safeguarded acceleration of
the ManPG iteration. The Newton correction is computed from a regularized
Newton equation, since the restricted smooth problem on $\overline{\M}_k$ may be
nonconvex and the Riemannian Hessian may be indefinite or ill-conditioned. The
regularization stabilizes the Newton-CG subproblem and damps the computed
correction, but it is not used as the mechanism for the global
descent guarantee.

Indeed, the identified manifold $\overline{\M}_k$ is an iteration-dependent local
model induced by the ManPG predictor $Y_k$, and the TAPR retraction in
Proposition~\ref{fact:ret_Tapr} is only guaranteed within a neighborhood. The radius of
such a neighborhood may be very small and is difficult to estimate in practice;
for instance, it can be sensitive to nearly vanishing nonzero entries in the
identified support and sign pattern. Moreover, even when the TAPR step is successfully
computed on $\overline{\M}_k$, projecting the trial point back to $\St_\Delta$
may increase the objective value. Therefore, the Newton-CG trial point is
accepted only if it satisfies a   sufficient decrease condition; 
otherwise, the algorithm falls back to the ManPG step. In this way, the global descent and KKT-residual guarantees are supplied by the ManPG safeguard, while the Newton-CG step can further improve practical efficiency and provides local second-order acceleration. Under the identification property and {\color{blue}the SOSC}, this step yields the superlinear local convergence stated later.

Specifically, we compute an inexact regularized Riemannian Newton direction
by solving
\begin{equation}\label{eq:Newton_equation}
    \Hess \bar F_k(Y_k)[w_k]+\varrho_k w_k
    =
    -\grad \bar F_k(Y_k),
    \qquad
    \varrho_k:=\lambda_k\normfro{\grad \bar F_k(Y_k)}^\sigma,
\end{equation}
where $\lambda_k>0$ and $\sigma\in[0,1]$ are regularization parameters. The system
\eqref{eq:Newton_equation} is solved approximately by the tCG algorithm; see
Algorithm~\ref{alg:tcg} in Appendix, following the truncated  conjugate-gradient
framework~\cite{dembo1983truncated}. If negative curvature is detected, or if tCG fails to satisfy the stopping criterion, the
Newton-CG step is skipped, the regularization scale $\lambda_k$ is increased, and
the algorithm falls back to a ManPG step. Unlike capped Newton-CG methods \cite{royer2020newtoncg} that exploit
negative-curvature directions for second-order complexity guarantees, we use tCG
only to compute a safeguarded Newton correction. Incorporating negative-curvature steps on the moving identified manifolds is an interesting direction for future research.

When a Newton direction $w_k$ is obtained, we generate
\[
    Z_k(\alpha)=\Retr_{Y_k}^{\rm TAPR}(\alpha w_k),
    \qquad
    X_k^{\rm trial}(\alpha)=\Proj_{\St_\Delta}(Z_k(\alpha)),
\]
where $\Retr^{\rm TAPR}$ is the TAPR retraction~\cite{chen2026alternating}. The TAPR computation is treated as an internal part of the Newton-CG trial
generation rather than as a separate outcome. If TAPR does not produce a usable
trial point for a tested stepsize, the algorithm continues the backtracking
procedure. 
Then, the trial point is tested in two stages.  First, we check the Newton Armijo
condition on the smooth model $\overline{\M}_k$:
\begin{equation}\label{eq:newton-armijo}
    \bar F_k(Z_k(\alpha))
    \le
    \bar F_k(Y_k)
    +
    c_{\rm N}\alpha
    \langle \grad \bar F_k(Y_k),w_k\rangle,
\end{equation}
where $c_{\rm N}\in(0,1/2)$. This condition ensures that the Newton-CG correction
is consistent with the local smooth model. If it fails, the Newton-CG trial point
is rejected and the regularization parameter is adjusted.

If \eqref{eq:newton-armijo} holds, we then test the projected point
$X_k^{\rm trial}(\alpha)$ by the sufficient decrease condition
\begin{equation}\label{eq:projected-decrease}
    F(X_k^{\rm trial}(\alpha))
    \le
    F(X_k)-\bar\tau\frac{\normfro{V_k}^2}{t_k} .
\end{equation}
This condition determines whether the projected Newton-CG trial point is accepted
and preserves the ManPG-type descent safeguard. 
 
\paragraph{Parameter update.}
The parameters $t_k$ and $\lambda_k$ are updated according to the outcome of
Algorithm~\ref{alg:newton-cg}. If an accepted Newton-CG trial point is computed
with $\mathrm{cg\_status}=\texttt{superlinear}$, we keep $t_k$ unchanged in
order to preserve the local Newton regime. Otherwise, $t_k$ is adjusted
according to whether the accepted step is reliable or whether the projected
Newton-CG trial fails. The rationale for this rule is explained in
Remark~\ref{rmk:adjust_t} after the descent lemma.
 
For compactness, denote
\[
    t_k^+ := \min\{\beta_t^\uparrow t_k,t_{\max}\},
    \qquad
    t_k^- := \beta_t^\downarrow t_k,
    \qquad
    \lambda_k^+ := \beta_\lambda^\uparrow\lambda_k,
    \qquad
    \lambda_k^- := \beta_\lambda^\downarrow\lambda_k .
\]
The parameter update rule is given by
\begin{equation}\label{eq:mix-update-rule}
\renewcommand{\arraystretch}{1.18}
\begin{array}{@{}c@{\qquad}c@{\qquad}c@{\qquad}c@{}}
\hline
\mathrm{tag}
& \mathrm{cg\_status}\ /\ \mathrm{condition}
& t_{k+1}
& \lambda_{k+1}
\\
\hline
\texttt{accepted}
& \texttt{superlinear},\ \ell_{\rm N}=0
& t_k
& \lambda_k^-
\\
\texttt{accepted}
& \texttt{superlinear},\ \ell_{\rm N}>0
& t_k
& \lambda_k^+
\\[1mm]
\texttt{accepted}
& \texttt{linear},\ \ell_{\rm N}=0
& t_k^+
& \lambda_k^-
\\
\texttt{accepted}
& \texttt{linear},\ \ell_{\rm N}>0
& t_k
& \lambda_k^+
\\[1mm]
\texttt{projection-fail}
& -
& t_k^-
& \lambda_k
\\
\texttt{newton-fail}
& -
& t_k
& \lambda_k^+
\\[1mm]
\texttt{manpg}
& \hat\alpha_k=1
& t_k^+
& \lambda_k
\\
\texttt{manpg}
& \hat\alpha_k<1
& t_k^-
& \lambda_k
\\
\hline
\end{array}
\end{equation}

The tag \texttt{accepted} means that both the Newton Armijo condition
\eqref{eq:newton-armijo} and the projected sufficient decrease condition
\eqref{eq:projected-decrease} hold. When
$\mathrm{cg\_status}=\texttt{superlinear}$, $t_k$ is kept unchanged to
preserve the local convergence regime. When
$\mathrm{cg\_status}=\texttt{linear}$, an accepted full step permits an
increase of $t_k$, while an accepted backtracked step keeps $t_k$ unchanged.
For the regularization scale, a full accepted Newton-CG step decreases
$\lambda_k$, whereas any accepted backtracked step increases $\lambda_k$ as a
conservative damping strategy for subsequent Newton corrections.

The tag \texttt{projection-fail} means that a trial point has passed the Newton Armijo test on the moving intersection model \(\overline{\M}_k\), but its projection back to the fixed constraint \(\St_\Delta\) fails the projected sufficient-decrease test \eqref{eq:projected-decrease}. The algorithm then leaves the Newton branch immediately. This failure is attributed to the mismatch between
\(\St_{\Delta_k}\) and \(\St_\Delta\), rather than to the length of the Newton
step, and the update rule therefore reduces \(t_k\). The tag \texttt{newton-fail} means that no reliable
Newton-CG trial point is obtained, and hence only $\lambda_k$ is increased.
Finally, if $\mathrm{tag}=\texttt{manpg}$, the Newton-CG branch is not invoked,
and $t_k$ is updated according to the ManPG line-search outcome.

Table~\ref{tab:mix-complexity} reports the computational complexity of the main
steps in MIX. If the intrinsic dimension of the intersection manifold is small,
then each tCG iteration and the TAPR retraction can be performed at a substantially
lower cost. In particular, after the active manifold is identified, the Newton-CG
correction is computed on the low-dimensional tangent space of
$\overline{\M}_k$, rather than on the ambient Stiefel manifold. Therefore, the
additional cost of the second-order correction can be moderate when the identified
support is sparse and the intersection structure is low-dimensional.
\begin{table}[t]
\centering
\caption{Total computational cost of the main components in one iteration of
Algorithm~\ref{alg:manpg-mix}, with the active manifold determined by $Y_k$.
Here $\gamma:=\gamma(Y_k)$, $s:=\dim \M_g(Y_k)$, $T_f$ denotes the cost of one
function evaluation, $T_{\nabla f}$ the cost of one Euclidean gradient evaluation,
and $T_{\mathrm{Hv}}$ the cost of one Euclidean Hessian--vector product. Moreover,
$N_{\mathrm{ssn}}$, $N_{\mathrm{tCG}}$, $N_{\mathrm{AP}}$, and $N_{\mathrm{SLRA}}$
denote the numbers of semismooth Newton iterations, tCG iterations, alternating
projection iterations, and NewtonSLRA iterations, respectively. For the ManPG
subproblem, $\hat s$ denotes the current active-set size appearing in the
semismooth Newton iterations.}
\label{tab:mix-complexity}
\begin{tabular}{l c}
\toprule
Component & Total cost \\
\midrule
Function value evaluation & $T_f$ \\

Euclidean gradient evaluation & $T_{\nabla f}$ \\

ManPG subproblem & $N_{\mathrm{ssn}}\cdot O(\hat s r^2+r^6)$ \\

Projection onto $\T_X\overline{\M}_k$
& $O(sr+\gamma^2)+O(sr^2+\gamma^3)$ \\

Newton step (tCG) on $\overline{\M}_k$
& $O(sr^2+\gamma^3)+N_{\mathrm{tCG}}\cdot\bigl(O(sr+\gamma^2)+T_{\mathrm{Hv}}\bigr)$ \\

APM/IAPM retraction & $N_{\mathrm{AP}}\cdot O(nr^2+r^3)$ \\

NewtonSLRA retraction & $N_{\mathrm{SLRA}}\cdot O(nr^2+sr^2+r^6)$ \\

Projection onto $\St_\Delta$ & $O(nr^2+r^3)$ \\
\bottomrule
\end{tabular}
\end{table}

\subsection{\texorpdfstring{Global convergence}{Global convergence}}
In the following subsections, we establish the global descent and
KKT-residual guarantees, as well as the local convergence, of Algorithm~\ref{alg:manpg-mix} under mild
assumptions.
 
\begin{assumption}\label{ass:working_set}
The iterates $\{X_k\}$ and ManPG predictors $\{Y_k\}$ generated by
Algorithm~\ref{alg:manpg-mix} remain in a bounded set $\mathcal B$. Moreover,
whenever the Newton branch is invoked, the corresponding TAPR trial points and
their projections remain in a bounded neighborhood of
\[
\mathcal N_\delta
:=
\left\{
X\in\mathbb R^{n\times r}:
\|X^\top X-I_r\|_{\mathrm F}
\le \|\Delta\|_{\mathrm F}+\delta^2
\right\}.
\]
On an open bounded neighborhood containing these sets, $\nabla f$ is
$L_f$-Lipschitz continuous and bounded by $M_f$, and $g$ is
$L_g$-Lipschitz continuous. Consequently, $F=f+g$ is Lipschitz continuous on
this neighborhood; denote one such Lipschitz constant by $L_{\mathrm F}$.
Moreover, every $G\in\partial g(X)$ on this neighborhood satisfies
$\|G\|_{\mathrm F}\le L_g$.
%
\end{assumption}

To simplify the presentation, several technical lemmas are put in
Appendix~\ref{append:technical_last}.
Our global descent analysis relies on two standard descent-type inequalities, recalled below
for completeness.

\begin{lemma}\label{lem:key_global_ineq}
Let $\{X_k\}$ be   generated by Algorithm~\ref{alg:manpg-mix}$.$
\begin{enumerate}
\item 
Given $X_k\in\St_\Delta$ and stepsize $t_k>0$, let $V_k$ be the minimizer of the ManPG tangent subproblem. Then, it follows that
\begin{equation}\label{eq:manpg_decrease}
\langle \nabla f(X_k),V_k\rangle+\frac{1}{2t_k}\|V_k\|_{\mathrm{F}}^2+g(X_k+V_k)\ \le\ g(X_k) -\frac{1}{2t_k}\normfro{V_k}^2.
\end{equation}
This inequality follows from \cite[Lemma~5.1]{chen2020proximal}.

\item 
For each $k$ such that $\overline{\M}_k$ is identified and $\bar F_k:=F|_{\overline{\M}_k}$ is smooth,
there exists a constant $\hat L_k>0$ such that for any $X\in\overline{\M}_k$ and any
$w\in \T_X\overline{\M}_k$ in a neighborhood where the retraction is defined,
\begin{equation}\label{eq:retraction_upper_bound}
\bar F_k(\Retr_X(w))
\ \le\
\bar F_k(X) + \langle \grad \bar F_k(X), w\rangle + \frac{\hat L_k}{2}\|w\|_{\mathrm{F}}^2.
\end{equation}
A statement of this type is standard for retractions on embedded manifolds; see, e.g.,
\cite{boumal2019global}.
\end{enumerate}
\end{lemma}

{
	\color{blue}The next lemma shows that, when $t_k$ is sufficiently small, every
	well-defined trial point satisfying the Newton Armijo
	condition~\eqref{eq:newton-armijo} also satisfies the projected
	sufficient-decrease condition~\eqref{eq:projected-decrease}.}
\begin{lemma}\label{lem:suff_decrease}
{\color{blue}Suppose
that Algorithm~\ref{alg:tcg} returns $w_k$ with
$\mathrm{cg\_status}\in\{\texttt{linear},\texttt{superlinear}\}$, and let
\[
    Z_k=\Retr_{Y_k}^{\rm TAPR}(\alpha w_k)
\]
be a well-defined trial point satisfying the Newton Armijo
condition~\eqref{eq:newton-armijo}. If
\[
    t_k\leq
    \bar t_{\rm N}:=\frac{1}{L_f+2c_\Delta L_{\mathrm F}}
    \qquad\text{and}\qquad
    0<\bar\tau\leq\frac12,
\]
then
\[
    F\bigl(\Proj_{\St_\Delta}(Z_k)\bigr)
    \leq
    F(X_k)-\frac{1}{2t_k}\normfro{V_k}^2
    \leq
    F(X_k)-\bar\tau\frac{\normfro{V_k}^2}{t_k}.
\]
}
\end{lemma}

\begin{proof}
{\color{blue}
Set $g_k:=\grad\bar F_k(Y_k)$,
$H_k:=\Hess\bar F_k(Y_k)+\varrho_k I$, and
$q_k:=H_k[w_k]+g_k$. For a successful tCG termination, $q_k$ is
orthogonal to the generated Krylov subspace, which contains $w_k$.
Since all generated CG directions satisfy the positive-curvature test
and are $H_k$-conjugate, one has
\[
    \langle q_k,w_k\rangle=0,
    \qquad
    \langle g_k,w_k\rangle
    =-\langle H_k[w_k],w_k\rangle
    \leq0.
\]

Since the Newton branch is invoked, $\normfro{V_k}<\delta$. Furthermore,
$Z_k\in\overline{\M}_k\subseteq\St_{\Delta_k}$ and
$\Delta_k=\Delta+V_k^\top V_k$. The choice of $\delta$ in
Algorithm~\ref{alg:manpg-mix}, together with
Lemma~\ref{lem:orthogonalization_increase_general_Delta} and the Lipschitz
continuity of $F$, gives
$F(\Proj_{\St_\Delta}(Z_k))
\leq\bar F_k(Z_k)+c_\Delta L_{\mathrm F}\normfro{V_k}^2$.
It follows from the Newton Armijo condition~\eqref{eq:newton-armijo} and
$\langle g_k,w_k\rangle\leq0$ that
$F(\Proj_{\St_\Delta}(Z_k))
\leq\bar F_k(Y_k)+c_\Delta L_{\mathrm F}\normfro{V_k}^2$.
Combining the ManPG inequality~\eqref{eq:manpg_decrease} with the
$L_f$-Lipschitz continuity of $\nabla f$, we obtain
$\bar F_k(Y_k)\leq F(X_k)
-(1/t_k-L_f/2)\normfro{V_k}^2$. Putting these estimates together gives
\[
\begin{aligned}
    F\bigl(\Proj_{\St_\Delta}(Z_k)\bigr)
    &\leq
    F(X_k)
    -
    \left(
        \frac{1}{t_k}
        -\frac{L_f}{2}
        -c_\Delta L_{\mathrm F}
    \right)
    \normfro{V_k}^2  \\
    &\leq
    F(X_k)-\frac{1}{2t_k}\normfro{V_k}^2.
\end{aligned}
\]
The last inequality follows from $t_k\leq\bar t_{\rm N}$. Since
$\bar\tau\leq1/2$, we obtain the projected sufficient-decrease
condition~\eqref{eq:projected-decrease}.}
\end{proof}

\begin{remark}\phantomsection\label{rmk:adjust_t}
\begin{enumerate}
    \item 
 The Newton Armijo condition is closely related to the ratio
$-\langle g_k,w_k\rangle/\normfro{w_k}^2$. Indeed, by
\eqref{eq:retraction_upper_bound}, it holds whenever
\(
\alpha
\le
\frac{2(1-c_{\rm N})(-\langle g_k,w_k\rangle)}
{\hat L_k\normfro{w_k}^2}.
\)
For an exact regularized Newton step,
\(
    (\Hess \bar F_k(Y_k)+\varrho_k I)w_k=-g_k,
\)
and hence
\[
-\langle g_k,w_k\rangle
=
\langle w_k,(\Hess \bar F_k(Y_k)+\varrho_k I)w_k\rangle .
\]
Thus, increasing $\lambda_k$ increases the shift $\varrho_k$ in the subsequent
Newton-CG solve and tends to enlarge the range of acceptable Newton steps.

\item  The ManPG stepsize $t_k$ in
$Y_k=\mathcal T_{t_k}(X_k)=X_k+V_k$ plays two roles. A larger $t_k$ usually
promotes stronger identification, whereas an overly large $t_k$ may reduce the
accuracy of the iteration-dependent model and make the projected decrease test
harder to satisfy. The update rule~\eqref{eq:mix-update-rule} reflects this
tradeoff. When a full Newton-CG step is accepted with only the \texttt{linear}
tCG status, $t_k$ is increased to encourage stronger identification. Once the
\texttt{superlinear} tCG regime is reached, $t_k$ is kept fixed so that the local
convergence analysis is carried out with a fixed ManPG scale. If the Newton
model is acceptable but the projected decrease test fails, $t_k$ is reduced;
otherwise, the regularization scale $\lambda_k$ is adjusted while $t_k$ is kept
unchanged.
\end{enumerate}
\end{remark}
 
\noindent\textbf{Boundedness of the ManPG stepsize $t_k$.}
{\color{blue}Let $\bar t_{\rm M}>0$ guarantee acceptance of the full ManPG step
when $t_k\le\bar t_{\rm M}$, set
$\bar t:=\min\{\bar t_{\rm M},\bar t_{\rm N}\}$, and assume
$0<\bar\tau\le1/2$. When $t_k\le\bar t$, a failed Newton trial leaves $t_k$
unchanged, while a trial satisfying Newton Armijo passes the projected test by
Lemma~\ref{lem:suff_decrease}; hence no branch reduces $t_k$. Therefore,  we have
\[
t_k\ge t_{\min}:=\min\{t_0,\beta_t^\downarrow\bar t\}>0,\qquad \forall k\ge0.
\]
Together with the safeguard in~\eqref{eq:mix-update-rule}, this gives
$0<t_{\min}\le t_k\le t_{\max}$ throughout the algorithm.}

\begin{definition}[ManPG residual]
For $X\in\St_\Delta$ and $t>0$, let
$\mathcal T_t(X)=X+V_t(X)$ be defined by the tangent subproblem
\eqref{eq:tangent_problem}. The ManPG residual is defined by
\[
    \mathcal G_t(X):=\frac{\normfro{V_t(X)}}{t}.
\]
Given $\eta\ge0$, we call $X$ an $\eta$-ManPG stationary point if
$\mathcal G_t(X)\le\eta$.
\end{definition}
Algorithm~\ref{alg:manpg-mix} uses $\normfro{V_k}/t_k$ as the stopping residual,
whereas Definition~\ref{def:KKT-residual-P0} measures stationarity by the KKT
residual of the original problem~\eqref{prob:nonsmooth_stiefel}. The following
lemma connects these two residuals. It shows that a small ManPG residual for the
perturbed problem yields a small KKT residual for the original problem, up to the perturbation error $\normfro{\Delta}$.

\begin{lemma}\label{lem:approx_KKT_residual}
Suppose that Assumption~\ref{ass:working_set} holds and that $t_k\le t_{\max}$.
Let $Y_k=X_k+V_k$ be the ManPG predictor generated by
Algorithm~\ref{alg:manpg-mix}. Then there exists a constant $C_{\rm R}>0$,
independent of $k$, such that
\[
    \mathcal R_0(Y_k)
    \le
    \normfro{\Delta}
    +
    C_{\rm R}\frac{\normfro{V_k}}{t_k}
    +
    \normfro{V_k}^2 .
\]
Consequently, after enlarging $C_{\rm R}$ if necessary, if
$\normfro{V_k}\le1$, then
\[
    \mathcal R_0(Y_k)
    \le
    \normfro{\Delta}
    +
    C_{\rm R}\frac{\normfro{V_k}}{t_k}.
\]
\end{lemma}
\begin{proof}
The proof evaluates the KKT residual at multipliers supplied by the tangent subproblem, bounds these multipliers on the working set, and uses \(Y_k^\top Y_k-I_r=\Delta+V_k^\top V_k\) to control feasibility; see Appendix~\ref{app:proof-approx-KKT-residual}.
\end{proof}
The following theorem gives the global descent and KKT-residual
guarantee for Algorithm~\ref{alg:manpg-mix}. It does not rely on the
identification theorem; for the perturbed problem the ManPG residual tends to
zero, while the predictor \(Y_k\) satisfies
\(\limsup_{k\to\infty}\mathcal R_0(Y_k)\le\normfro{\Delta}\) for the original
problem, giving vanishing KKT residual when \(\Delta=0\).
\begin{theorem}\label{thm:global_convergence_mix}
Let $\{X_k\}$ and $\{Y_k\}$ be generated by
Algorithm~\ref{alg:manpg-mix}. Suppose that Assumption~\ref{ass:working_set}
holds and that $F(X)\ge F_{\inf}$ on $\St_\Delta$. Then there exists a
constant $c_{\rm dec}>0$ such that, for all $K\ge1$,
\[
    \min_{0\le j\le K-1}\normfro{V_j}^2
    \le
    \frac{F(X_0)-F_{\inf}}{c_{\rm dec}K}.
\]
Consequently,
\[
    \limsup_{k\to\infty}\mathcal R_0(Y_k)
    \le
    \normfro{\Delta}.
\]
Moreover, for any $\varepsilon>\normfro{\Delta}$, choose
\[
    \epsilon_{\rm stat}
    =
    \min\left\{
        \frac{1}{t_{\max}},
        \frac{\varepsilon-\normfro{\Delta}}{C_{\rm R}}
    \right\},
\]
where $C_{\rm R}$ is the constant in
Lemma~\ref{lem:approx_KKT_residual}. Then
Algorithm~\ref{alg:manpg-mix} returns an $\varepsilon$-KKT point of
problem~\eqref{prob:nonsmooth_stiefel} within
$O(\epsilon_{\rm stat}^{-2})$ iterations. If, in addition,
$\varepsilon-\normfro{\Delta}\le C_{\rm R}/t_{\max}$, then this bound becomes
\(
O\!\left((\varepsilon-\normfro{\Delta})^{-2}\right).
\)
\end{theorem}
 
\begin{proof}
If the Newton-CG trial point is accepted, then
\eqref{eq:projected-decrease} gives
\(F(X_{k+1})\le F(X_k)-\bar\tau\normfro{V_k}^2/t_k
\le F(X_k)-(\bar\tau/t_{\max})\normfro{V_k}^2\).
Otherwise, Algorithm~\ref{alg:manpg-mix} takes the ManPG fallback step.
According to the ManPG line-search decrease estimate in
\cite[Lemma~5.2]{chen2020proximal}, together with $t_k\le t_{\max}$ established above, there
exists a constant $c_{\rm M}>0$ such that
\(F(X_{k+1})\le F(X_k)-c_{\rm M}\normfro{V_k}^2\).
Thus, with $c_{\rm dec}:=\min\{c_{\rm M},\bar\tau/t_{\max}\}$, one has
\[
    F(X_{k+1})
    \le
    F(X_k)-c_{\rm dec}\normfro{V_k}^2
\]
for all $k$. Summing this inequality yields
\[
    \sum_{k=0}^{K-1}\normfro{V_k}^2
    \le
    \frac{F(X_0)-F_{\inf}}{c_{\rm dec}},
    \qquad
    \min_{0\le j\le K-1}\normfro{V_j}^2
    \le
    \frac{F(X_0)-F_{\inf}}{c_{\rm dec}K}.
\]
In particular, $V_k\to0$. By the lower bound $t_k\ge t_{\min}>0$
established above, one has
$\normfro{V_k}/t_k\to0$.
Lemma~\ref{lem:approx_KKT_residual} therefore yields
\(
    \limsup_{k\to\infty}\mathcal R_0(Y_k)
    \le
    \normfro{\Delta}.
\)
For the iteration complexity, the preceding estimate yields an index
$j\in\{0,\ldots,K-1\}$ such that
\[
    \frac{\normfro{V_j}}{t_j}
    \le
    \frac{1}{t_{\min}}
    \left(
        \frac{F(X_0)-F_{\inf}}{c_{\rm dec}K}
    \right)^{1/2}.
\]
Hence, if
\(
    K
    \ge
    \frac{F(X_0)-F_{\inf}}
         {c_{\rm dec}t_{\min}^2\epsilon_{\rm stat}^2},
\)
then
\(
    \frac{\normfro{V_j}}{t_j}\le\epsilon_{\rm stat}.
\)
Since $\epsilon_{\rm stat}\le1/t_{\max}$ and $t_j\le t_{\max}$, one has
$\normfro{V_j}\le1$. Therefore, Lemma~\ref{lem:approx_KKT_residual} gives
\[
\begin{aligned}
    \mathcal R_0(Y_j)
    \le
    \normfro{\Delta}
    +
    C_{\rm R}\frac{\normfro{V_j}}{t_j} \le
    \normfro{\Delta}
    +
    C_{\rm R}\epsilon_{\rm stat}
    \le
    \varepsilon.
\end{aligned}
\]
Thus, $Y_j$ is an $\varepsilon$-KKT point of
\eqref{prob:nonsmooth_stiefel}. Therefore, the iteration bound is
$O(\epsilon_{\rm stat}^{-2})$. If
\(
    \varepsilon-\normfro{\Delta}
    \le
    \frac{C_{\rm R}}{t_{\max}},
\)
then
\(
    \epsilon_{\rm stat}
    =
    \frac{\varepsilon-\normfro{\Delta}}{C_{\rm R}},
\)
and hence the bound becomes
   $ O\!\left((\varepsilon-\normfro{\Delta})^{-2}\right).$
\end{proof}

\subsection{Local convergence}
{\color{blue} In this section, we show the local superlinear convergence of MIX. For the asymptotic analysis, we consider an infinite sequence
	$\{X_k\}$ generated by Algorithm~\ref{alg:manpg-mix}.  We first establish full-sequence convergence to an accumulation
point and invoke finite identification only afterward.} Write
\[
\M_\Delta^* := \M_g(X^*)\cap \mathrm{St}_\Delta(n,r).
\]

\begin{assumption}\label{assump:mix-seq-converge}
	{\color{blue}
	Let $X^*\in\St_\Delta(n,r)$ be a nondegenerate critical point of
	$(P_\Delta)$.   Moreover, assume that $X^*$ is an accumulation point
	of the sequence $\{X_k\}$ generated by
	Algorithm~\ref{alg:manpg-mix}.
	}
\end{assumption}

%

{\color{blue}
	For the local convergence analysis, we assume the standard second-order
	sufficient condition (SOSC) on $\mathcal M_\Delta^*$: there exists
	$\mu_{\rm sos}>0$ such that
	\begin{equation}\label{eq:mix-sosc}
		\bigl\langle
		\xi,
		\Hess(F|_{\mathcal M_\Delta^*})(X^*)[\xi]
		\bigr\rangle
		\ge
		\mu_{\rm sos}\|\xi\|_{\mathrm F}^2,
		\qquad
		\forall\,\xi\in \T_{X^*}\mathcal M_\Delta^* .
	\end{equation}
}

{\color{blue}The following lemma shows that the local solution depends
Lipschitz continuously on the perturbation parameter $\Delta'$} for the family of problems
\[
\min F(X)\quad \st\quad X\in \M_g(X^*),\qquad X^\top X = I_r+\Delta'.
\]
It mainly depends on  the transversality condition and SOSC. 
 A generic perturbation $\Delta\in\mathscr P_{\rm gen}$ is one sufficient way to obtain this transversality condition by Proposition~\ref{lem:generic-transv}. Lemma~\ref{lem:local-solution-branch} also covers $\Delta=0$ whenever the same transversality holds at $X^*$.
\begin{lemma}\label{lem:local-solution-branch}
Assume that $\M_g(X^*)$ and $\St_\Delta(n,r)$ intersect transversely at $X^*$.
Define
\[
    \overline{\M}_{\Delta'}^*
    :=
    \M_g(X^*)\cap\St_{\Delta'}(n,r),
    \qquad \Delta'\in\mathbb S^r .
\]
{\color{blue}Suppose that $F|_{\M_g(X^*)}$ is $C^2$ around $X^*$ and that
the SOSC~\eqref{eq:mix-sosc} holds.} Then there exist
neighborhoods $\mathcal U$ of $X^*$ and $\mathcal V$ of $\Delta$, a constant
$C_\Delta>0$, and a unique $C^1$ mapping
\[
    \mathcal X:\mathcal V\to\mathcal U
\]
such that
\(
    \mathcal X(\Delta)=X^*,
\)
and, for every $\Delta'\in\mathcal V$, $\mathcal X(\Delta')$ is the unique local
minimizer of $F$ on $\overline{\M}_{\Delta'}^*\cap\mathcal U$. Moreover,
\[
    \|\mathcal X(\Delta')-X^*\|_{\mathrm F}
    \le
    C_\Delta\|\Delta'-\Delta\|_{\mathrm F},
    \qquad \forall\,\Delta'\in\mathcal V.
\]
{\color{blue}After shrinking $\mathcal U$ and $\mathcal V$, there is
$\mathfrak m>0$ such that
$\Hess(F|_{\overline{\M}_{\Delta'}^*})(Y)\succeq\mathfrak m I$ for all
$\Delta'\in\mathcal V$ and
$Y\in\overline{\M}_{\Delta'}^*\cap\mathcal U$.}
\end{lemma}
\begin{proof}
The proof uses a submersion argument for the equality system defining \(\M_g(X^*)\cap\St_{\Delta'}(n,r)\), followed by an implicit-function argument for the restricted first-order condition; see Appendix~\ref{app:proof-local-solution-branch}.
{\color{blue}The Hessian bound follows directly from the SOSC and continuity
after shrinking the two neighborhoods.}
\end{proof}

\begin{lemma}\label{lem:alpha_one_eventually}
{\color{blue}
Suppose that Assumptions~\ref{assump:finite-active-manifolds}, \ref{ass:working_set} and
\ref{assump:mix-seq-converge} hold. Assume further that
$\M_g(X^*)$ and $\St_\Delta(n,r)$ intersect transversely at $X^*$,
and that the SOSC~\eqref{eq:mix-sosc} holds.
Let $\mathfrak m>0$ denote the uniform lower bound in
Lemma~\ref{lem:local-solution-branch}. Assume that $\varpi>0$ and that, for all sufficiently large $k$,
\(
\vartheta_k<\mathfrak m+\varrho_k .
\) Then
\(
    X_k\to X^*.
\)
Moreover, for all sufficiently large $k$,
\(    \M_g(Y_k)=\M_g(X^*),
\)
Algorithm~\ref{alg:tcg} returns
$\mathrm{cg\_status}=\texttt{superlinear}$, and
Algorithm~\ref{alg:newton-cg} accepts the full Newton step $\alpha=1$.
Consequently, $t_k$ is eventually constant and $\lambda_k\to0$.
}
\end{lemma}

\begin{proof}
{\color{blue}
See Appendix~\ref{app:proof-alpha-one-eventually}. {\color{blue}
	The proof follows the local convergence strategy of
	Bareilles, Iutzeler, and Malick~\cite{bareilles2023newton} for Euclidean
	composite optimization, adapted here to the moving intersection manifolds
	$\M_{\Delta_k}^*$ and the projection back to $\St_\Delta(n,r)$.
}
}
\end{proof}

	   We now establish the local convergence rate of
$\{X_k\}$.

\begin{theorem}[Local convergence rate of MIX]
\label{thm:local_superlinear_full_chain}
{\color{blue}
Under the assumptions of Lemma~\ref{lem:alpha_one_eventually}, the convergence
$X_k\to X^*$ is Q-superlinear.}
\end{theorem}
\begin{proof}
{\color{blue}
Lemma~\ref{lem:alpha_one_eventually} ensures that, for all sufficiently large $k$,
$\overline{\M}_k=\overline{\M}_{\Delta_k}^*$ and
$\Delta_k=\Delta+V_k^\top V_k$; moreover, the full Newton-CG step is
accepted, $t_k$ is fixed, and $\lambda_k\to0$. Set
$e_k:=\normfro{X_k-X^*}$, and let
$X_k^*:=\mathcal X(\Delta_k)$ be the local minimizer given by
Lemma~\ref{lem:local-solution-branch}.

It follows from Lemma~\ref{lem:Vk_dist_bound} that
$\normfro{V_k}=O(e_k)$. Combining this estimate with
Lemma~\ref{lem:local-solution-branch}, one has
$\normfro{X_k^*-X^*}=O(\normfro{\Delta_k-\Delta})
=O(\normfro{V_k}^2)=O(e_k^2)$. Since $Y_k=X_k+V_k$, we also have
$\normfro{Y_k-X_k^*}=O(e_k)$.

Applying Lemma~\ref{lem:local_regularized_newton_estimate}, we obtain
$\normfro{Z_k-X_k^*}
=o(e_k)+O(\varrho_k e_k+e_k^{1+\theta})$. Moreover, since
$Z_k\in\St_{\Delta_k}(n,r)$,
Lemma~\ref{lem:orthogonalization_increase_general_Delta} implies
$\normfro{X_{k+1}-Z_k}=O(\normfro{V_k}^2)=O(e_k^2)$.
Combining these estimates gives the central recursion
\begin{equation}\label{eq:full-local-recursion}
    e_{k+1}
    =
    o(e_k)+O\!\left(
        \varrho_k e_k+e_k^{1+\theta}
    \right).
\end{equation}

The estimate in Lemma~\ref{lem:grad_bound_by_manpg}, together with
$\normfro{V_k}=O(e_k)$, gives
$\normfro{\grad\bar F_k(Y_k)}=O(e_k)$. Thus
$\varrho_k=\lambda_k\normfro{\grad\bar F_k(Y_k)}^{\sigma}\to0$
for every $\sigma\in[0,1]$. Dividing
\eqref{eq:full-local-recursion} by $e_k$ gives
$e_{k+1}/e_k\to0$, which proves Q-superlinear convergence.
}
\end{proof}

\begin{remark}[Perturbations and degenerate local models]
	The perturbation used in the local analysis is meant to obtain identification,
	but it also changes the identified manifold and the restricted Riemannian
	Hessian. If the unperturbed restricted Hessian has zero eigenvalues, then a
	small perturbation may create negative curvature and the corresponding point
	need not remain a local minimizer. {\color{blue}The SOSC in
	Theorem~\ref{thm:local_superlinear_full_chain} rules out this instability for
	the perturbed identified model.}
	
	Degenerate unperturbed cases are not discussed in this paper. When transversality
	already holds, no perturbation is needed, and a superlinear theory may be
	possible under error-bound-type conditions for non-isolated minimizers; see
	\cite{rebjock2023fast,rebjock2023trust}. We also observe numerically that
	the unperturbed ManPG predictor often identifies the active manifold even in
	some nontransverse cases. We leave these issues for future work.
\end{remark}

%% file: shared/section_num_exp.tex
\section{Numerical experiments}\label{sec:numerical_exp}

In this section, we evaluate MIX on two standard sparse optimization problems
over the Stiefel manifold: sparse principal component analysis (SPCA) and the
compressed modes problem.

According to \cite{huang2024riemannian}, RPNCGH is more
efficient than algorithms such as ManPG~\cite{chen2020proximal} and
ManPQN~\cite{wang2023proximal}.
Therefore, we mainly compare MIX with RPNCGH. We test three
MIX variants, corresponding to perturbation scales \(\normfro{\Delta}=0\),
\(10^{-8}\), and \(10^{-7}\). For each perturbed run, we generate an i.i.d.\ Gaussian matrix
$G\in\R^{r\times r}$, set
\(
\widetilde\Delta
=
\frac{G+G^\top}{2}
-
\diag\!\left(\diag\!\left(\frac{G+G^\top}{2}\right)\right),
\)
and scale it as
\(
\Delta
=
\delta_\Delta
\frac{\widetilde\Delta}{\|\widetilde\Delta\|_{\mathrm F}}, \ 
\delta_\Delta\in\{10^{-8},10^{-7}\}.
\) The corresponding algorithms are denoted by
\(\mathrm{MIX}_0\), \(\mathrm{MIX}_{10^{-8}}\), and
\(\mathrm{MIX}_{10^{-7}}\), respectively. 
All algorithms were implemented in Julia under the same coding
standards\footnote{Our codes are available at \url{https://github.com/chenshixiang/manifold_mix/tree/main}}. RPNCGH was reimplemented in Julia for a fair comparison. We use the parameter settings provided in the authors'
public implementation of~\cite{huang2024riemannian}. All experiments
were run on an Apple M4 MacBook Pro.

 \paragraph{Parameter settings.}
In Algorithm~\ref{alg:manpg-mix}, we use
\(\epsilon_{\rm stat}=10^{-8}\).
The initial point is generated from an i.i.d. Gaussian matrix
and projected onto the corresponding constraint. We set \(t_0=1/L\), where $L$ is the Lipschitz constant of $\nabla f$, and set
\(t_{\max}=100t_0\), \(\delta=\sqrt{10^{-1}}\),
\(\beta_t^\uparrow=1.01\), and \(\beta_t^\downarrow=0.95\).
The regularized Newton equation~\eqref{eq:Newton_equation}
uses \(\lambda_0=1\) and \(\sigma=1/2\); in the update
rule~\eqref{eq:mix-update-rule}, the implementation uses the factors
\(\beta_\lambda^\downarrow=1/2\) and \(\beta_\lambda^\uparrow=2\), with an
additional factor \(4\) after a failed Newton Armijo test.
For Algorithm~\ref{alg:newton-cg}, backtracking starts from
\(\alpha=1\) and uses \(c_{\rm N}=10^{-3}\), \(\varsigma=1/2\), and at most
five reductions, so \(\alpha_{\min}=2^{-5}\).
The projected sufficient-decrease test
\eqref{eq:projected-decrease} is implemented with \(\bar\tau=10^{-3}\), i.e.,
\(F(X_k^{\rm trial})\le F(X_k)-10^{-3}\normfro{V_k}^2/t_k\).
The TAPR step in Algorithm~\ref{alg:newton-cg} is the
three-phase projection retraction described in
Section~\ref{sec:opt_on_intersection}: starting from \(Y_k+\alpha w_k\), it
uses APM while the inner difference is at least \(10^{-2}\), IAPM until it is
below \(10^{-5}\), and then RN-SLRA corrections; we cap the APM/IAPM phase at
\(30\) steps and the RN-SLRA phase at \(2\) steps.
The ManPG fallback line search uses the same decrease
constant \(10^{-3}\), halves the trial stepsize, and allows at most five
reductions.
{\color{blue}For Algorithm~\ref{alg:tcg}, we take \(\theta=1/2\) and
\(\varpi=10^{-1}\), use \(d_k=\dim(\overline{\M}_k)\) as the iteration cap,
and   set $\vartheta_k=\varrho_k$.}
The ManPG tangent subproblem is solved by semismooth Newton
\cite{chen2020proximal} with at most \(20\) iterations.

Since MIX returns the ManPG predictor \(Y_k=\mathcal T_t(X_k)\), we report
the same type of output for RPNCGH. This makes the comparison consistent at
the level of the sparse proximal predictor. In contrast, the numerical
experiments in~\cite{huang2024riemannian} report a thresholded version of the
manifold iterate \(X_k\); such post-processing promotes sparsity but may
degrade orthogonality. In our updated implementation, the reported RPNCGH
point is \(Y_k\), which is sparse and nearly orthogonal when the ManPG
subproblem is solved accurately, leading to a fair comparison with MIX.

At iteration \(k\), we use
\(\omega_k=\normfro{\mathcal T_t(X_k)-X_k}/t_k\) as the stationarity measure.
A run is assigned status \texttt{TOL} if
\(\omega_k\le10^{-8}\) and the semismooth Newton residual is below \(10^{-8}\);
status \texttt{ACC} if \(10^{-8}<\omega_k\le10^{-7}\) holds for 15
consecutive iterations. Otherwise, the run stops at the iteration cap
specified by the corresponding experiment script.

\subsection{Compressed modes problem}
For compressed modes, given a symmetric positive semidefinite matrix
\(H\in\mathbb R^{n\times n}\), typically a discretized Laplacian or Hamiltonian
operator \cite{Ozolins2013}, we solve
\begin{equation}\label{prob:cm-num}
	\min_{X\in\St(n,r)}
	\frac12\mathrm{Tr}(X^\top H X)+\mu\|X\|_1 .
\end{equation}
The quadratic term favors low-energy modes, whereas the \(\ell_1\) penalty
promotes spatial localization. The data are generated as in \cite{chen2020proximal}.

The initial point is generated from an i.i.d. Gaussian matrix and then projected onto the constraint manifold.
In Tables~\ref{tab:cms_varying_n}--\ref{tab:cms_varying_mu}, ``T/A'' reports the percentages of the 50 trials that terminate with status \texttt{TOL} and \texttt{ACC}, respectively. ``Iter(New)'' gives the average number of outer iterations, with the average number of Newton iterations  in parentheses. ``Time'', ``Obj'', and ``Sp'' denote the average CPU time, final objective value, and sparsity ratio (zero-entry fraction). ``O/F'' reports the orthogonality residual $\normfro{Y_k^\top Y_k-I_r}$ followed by the feasibility residual $\normfro{Y_k^\top Y_k-(I_r+\Delta)}$. For the perturbed variants, the orthogonality residual relative to \(\St(n,r)\) reflects the prescribed perturbation size, while the feasibility residual relative to \(\St_\Delta(n,r)\) is the relevant constraint residual. ``MProx'' gives the average ManPG subproblem iterations with its runtime percentage in parentheses; ``tCG'' gives the average inner tCG iterations for Newton directions with its runtime percentage; and ``TAPR'' gives the average TAPR iterations and runtime percentage, with ``--'' indicating that TAPR is not used.
Figure~\ref{fig:cms-convergence} shows a representative CM
convergence history. ManPG stagnates above the target
tolerance and does not reach high accuracy on this instance. Both MIX and RPNCGH exhibit local superlinear convergence, while MIX uses the fewest iterations.
\begin{figure}[!htbp]
\centering
\includegraphics[width=0.58\textwidth]{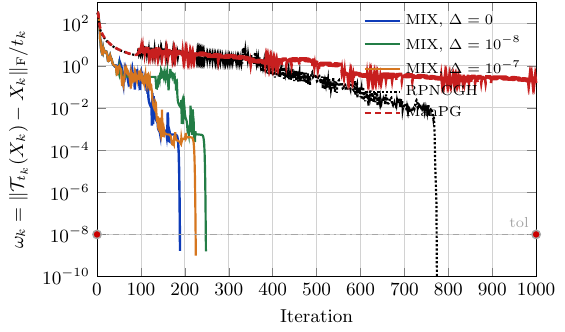}
\caption{CM convergence for \(n=512\), \(r=5\), and
\(\mu=0.1\).}
\label{fig:cms-convergence}
\end{figure}
\begin{table*}[t]
\centering
\scriptsize
\caption{CM performance comparison when varying the   dimension $n$, with fixed   $r=5, \mu=0.1$.}
\label{tab:cms_varying_n}
\resizebox{\textwidth}{!}{%
\begin{tabular}{@{}llccccccccc@{}}
\toprule
$n$ & Alg. & T/A & Iter(New) & Time & Obj & Sp & O/F & MProx & tCG & TAPR \\
\midrule
$128$ & RPNCGH & 100/0 & 181.1(65.5) & 0.028 & 2.35596 & 0.827 & 1.1e-12/1.1e-12 & 2.8(48.3\%) & 17.2(27.6\%) & -- \\
 & MIX$_0$ & 100/0 & 77.7(56.0) & 0.029 & 2.35594 & 0.827 & 1.2e-15/1.2e-15 & 5.4(35.4\%) & 15.7(25.7\%) & 9.5(15.5\%) \\
 & MIX$_{10^{-7}}$ & 100/0 & 77.3(55.5) & 0.031 & 2.35594 & 0.811 & 1.0e-07/1.8e-15 & 6.8(37.2\%) & 15.4(24.3\%) & 9.7(15.0\%) \\
 & MIX$_{10^{-8}}$ & 100/0 & 80.0(57.8) & 0.036 & 2.35594 & 0.811 & 1.0e-08/5.5e-13 & 7.8(41.6\%) & 15.5(22.3\%) & 9.3(13.5\%) \\
\addlinespace
$256$ & RPNCGH & 100/0 & 391.7(151.5) & 0.095 & 3.11171 & 0.850 & 7.3e-15/7.3e-15 & 2.0(32.0\%) & 21.3(43.8\%) & -- \\
 & MIX$_0$ & 100/0 & 115.0(89.6) & 0.083 & 3.11172 & 0.849 & 4.7e-15/4.7e-15 & 5.1(30.9\%) & 28.8(39.7\%) & 7.6(12.0\%) \\
 & MIX$_{10^{-7}}$ & 98/2 & 108.2(84.7) & 0.074 & 3.11172 & 0.842 & 1.0e-07/2.2e-15 & 5.4(26.9\%) & 28.8(43.6\%) & 7.9(12.8\%) \\
 & MIX$_{10^{-8}}$ & 100/0 & 113.7(88.5) & 0.087 & 3.11172 & 0.842 & 1.0e-08/5.8e-13 & 6.6(33.5\%) & 29.0(38.9\%) & 7.6(11.6\%) \\
\addlinespace
$512$ & RPNCGH & 100/0 & 630.5(221.3) & 0.314 & 4.10724 & 0.870 & 1.2e-14/1.2e-14 & 1.6(22.1\%) & 31.6(52.3\%) & -- \\
 & MIX$_0$ & 100/0 & 123.6(97.7) & 0.161 & 4.10724 & 0.869 & 9.8e-16/9.8e-16 & 4.6(24.1\%) & 40.0(55.4\%) & 6.7(10.0\%) \\
 & MIX$_{10^{-7}}$ & 100/0 & 122.5(97.8) & 0.163 & 4.10724 & 0.864 & 1.0e-07/2.0e-15 & 4.6(18.5\%) & 40.9(57.9\%) & 6.6(10.1\%) \\
 & MIX$_{10^{-8}}$ & 100/0 & 123.8(98.9) & 0.174 & 4.10724 & 0.865 & 1.0e-08/1.7e-15 & 5.5(24.5\%) & 39.7(53.2\%) & 6.6(9.6\%) \\
\addlinespace
$1024$ & RPNCGH & 100/0 & 734.1(379.5) & 1.428 & 5.42009 & 0.888 & 8.9e-13/8.9e-13 & 1.9(15.0\%) & 56.5(64.8\%) & -- \\
 & MIX$_0$ & 100/0 & 179.2(146.8) & 0.575 & 5.42009 & 0.888 & 1.5e-15/1.5e-15 & 3.8(15.0\%) & 65.2(70.0\%) & 5.3(7.1\%) \\
 & MIX$_{10^{-7}}$ & 100/0 & 179.8(149.2) & 0.595 & 5.42009 & 0.882 & 1.0e-07/2.0e-15 & 3.9(11.5\%) & 65.5(71.7\%) & 5.1(6.8\%) \\
 & MIX$_{10^{-8}}$ & 100/0 & 178.6(146.8) & 0.624 & 5.42009 & 0.885 & 1.0e-08/1.7e-15 & 4.5(14.2\%) & 66.6(69.6\%) & 5.2(6.8\%) \\
\addlinespace
$2048$ & RPNCGH & 100/0 & 789.2(521.5) & 4.417 & 7.1521 & 0.903 & 8.0e-15/8.0e-15 & 1.6(6.1\%) & 57.7(67.4\%) & -- \\
 & MIX$_0$ & 100/0 & 216.9(181.5) & 1.270 & 7.1521 & 0.902 & 1.7e-15/1.7e-15 & 4.4(20.4\%) & 85.6(67.3\%) & 5.3(5.5\%) \\
 & MIX$_{10^{-7}}$ & 100/0 & 209.9(175.9) & 1.280 & 7.1521 & 0.897 & 1.0e-07/2.1e-15 & 4.0(12.2\%) & 90.0(73.7\%) & 5.3(5.3\%) \\
 & MIX$_{10^{-8}}$ & 100/0 & 207.2(173.1) & 1.315 & 7.1521 & 0.899 & 1.0e-08/2.0e-15 & 4.6(14.5\%) & 91.1(71.5\%) & 5.4(5.3\%) \\
\bottomrule
\end{tabular}}
\end{table*}

\begin{table*}[t]
\centering
\scriptsize
\caption{CM performance comparison when varying the rank $r$, with fixed $n=256, \mu=0.1$.}
\label{tab:cms_varying_r}
\resizebox{\textwidth}{!}{%
\begin{tabular}{@{}llccccccccc@{}}
\toprule
$r$ & Alg. & T/A & Iter(New) & Time & Obj & Sp & O/F & MProx & tCG & TAPR \\
\midrule
$1$ & RPNCGH & 100/0 & 176.3(20.6) & 0.013 & 0.622338 & 0.859 & 4.3e-15/4.3e-15 & 1.1(17.2\%) & 9.5(3.8\%) & -- \\
 & MIX$_0$ & 100/0 & 36.7(28.8) & 0.005 & 0.622338 & 0.859 & 3.2e-16/3.2e-16 & 2.4(11.4\%) & 13.1(26.3\%) & -- \\
\addlinespace
$3$ & RPNCGH & 100/0 & 280.7(66.6) & 0.032 & 1.86702 & 0.856 & 2.4e-15/2.4e-15 & 1.3(24.0\%) & 21.4(31.7\%) & -- \\
 & MIX$_0$ & 100/0 & 61.2(47.6) & 0.018 & 1.86702 & 0.856 & 1.1e-15/1.1e-15 & 3.7(26.3\%) & 18.4(30.4\%) & 7.8(12.9\%) \\
 & MIX$_{10^{-7}}$ & 100/0 & 61.6(48.2) & 0.019 & 1.86702 & 0.847 & 1.0e-07/1.4e-15 & 4.1(24.1\%) & 18.8(28.5\%) & 7.0(13.4\%) \\
 & MIX$_{10^{-8}}$ & 100/0 & 60.6(47.1) & 0.021 & 1.86702 & 0.847 & 1.0e-08/1.6e-15 & 5.2(34.6\%) & 18.5(25.0\%) & 7.2(12.0\%) \\
\addlinespace
$5$ & RPNCGH & 100/0 & 391.7(151.5) & 0.091 & 3.11171 & 0.850 & 7.3e-15/7.3e-15 & 2.0(32.2\%) & 21.3(44.0\%) & -- \\
 & MIX$_0$ & 100/0 & 115.0(89.6) & 0.079 & 3.11172 & 0.849 & 4.7e-15/4.7e-15 & 5.1(31.8\%) & 28.8(41.0\%) & 7.6(12.2\%) \\
 & MIX$_{10^{-7}}$ & 98/2 & 108.2(84.7) & 0.071 & 3.11172 & 0.842 & 1.0e-07/2.2e-15 & 5.4(27.3\%) & 28.8(43.0\%) & 7.9(13.0\%) \\
 & MIX$_{10^{-8}}$ & 100/0 & 113.7(88.5) & 0.084 & 3.11172 & 0.842 & 1.0e-08/5.8e-13 & 6.6(33.6\%) & 29.0(39.3\%) & 7.6(11.4\%) \\
\addlinespace
$10$ & RPNCGH & 100/0 & 476.6(168.8) & 0.588 & 6.27228 & 0.758 & 2.8e-15/2.8e-15 & 6.7(53.9\%) & 25.0(38.9\%) & -- \\
 & MIX$_0$ & 100/0 & 139.4(102.1) & 0.401 & 6.27228 & 0.758 & 1.9e-15/1.9e-15 & 13.6(54.6\%) & 34.0(30.7\%) & 10.0(8.3\%) \\
 & MIX$_{10^{-7}}$ & 100/0 & 134.5(100.8) & 0.329 & 6.27228 & 0.760 & 1.0e-07/4.4e-15 & 11.0(44.5\%) & 35.8(36.9\%) & 10.4(10.2\%) \\
 & MIX$_{10^{-8}}$ & 100/0 & 134.9(99.6) & 0.366 & 6.27228 & 0.758 & 1.0e-08/3.4e-15 & 12.9(50.9\%) & 34.7(32.4\%) & 10.2(8.6\%) \\
\addlinespace
$15$ & RPNCGH & 100/0 & 483.7(102.1) & 4.167 & 10.0998 & 0.808 & 1.4e-11/1.4e-11 & 21.9(89.4\%) & 23.8(8.9\%) & -- \\
 & MIX$_0$ & 100/0 & 228.7(139.9) & 1.669 & 10.0998 & 0.808 & 1.2e-12/1.2e-12 & 17.7(74.7\%) & 18.7(13.5\%) & 13.8(8.0\%) \\
 & MIX$_{10^{-7}}$ & 100/0 & 220.4(134.5) & 1.524 & 10.1 & 0.791 & 1.0e-07/4.3e-13 & 16.5(70.5\%) & 18.0(15.4\%) & 14.8(8.5\%) \\
 & MIX$_{10^{-8}}$ & 100/0 & 226.6(138.6) & 1.645 & 10.0999 & 0.791 & 1.0e-08/1.4e-11 & 17.5(72.0\%) & 18.9(15.1\%) & 14.3(7.6\%) \\
\bottomrule
\end{tabular}}
\end{table*}

\begin{table*}[t]
\centering
\scriptsize
\caption{CM performance comparison when varying $\mu$, with fixed $n=512,r=5$.}
\label{tab:cms_varying_mu}
\resizebox{\textwidth}{!}{%
\begin{tabular}{@{}llccccccccc@{}}
\toprule
$\mu$ & Alg. & T/A & Iter(New) & Time & Obj & Sp & O/F & MProx & tCG & TAPR \\
\midrule
$0.01$ & RPNCGH & 100/0 & 901.5(675.0) & 1.460 & 0.661805 & 0.438 & 2.3e-15/2.3e-15 & 1.2(5.2\%) & 48.5(84.9\%) & -- \\
 & MIX$_0$ & 100/0 & 274.2(254.0) & 0.512 & 0.661805 & 0.438 & 1.7e-15/1.7e-15 & 3.4(8.4\%) & 54.6(75.7\%) & 5.0(8.1\%) \\
 & MIX$_{10^{-7}}$ & 100/0 & 268.4(248.0) & 0.540 & 0.661805 & 0.438 & 1.0e-07/2.3e-15 & 3.3(7.2\%) & 57.5(75.9\%) & 4.9(7.2\%) \\
 & MIX$_{10^{-8}}$ & 100/0 & 277.8(256.5) & 0.549 & 0.661805 & 0.438 & 1.0e-08/2.4e-15 & 3.3(7.3\%) & 57.7(75.6\%) & 4.9(7.5\%) \\
\addlinespace
$0.05$ & RPNCGH & 100/0 & 773.7(383.8) & 0.605 & 2.35915 & 0.822 & 5.4e-15/5.4e-15 & 1.7(14.0\%) & 44.1(71.2\%) & -- \\
 & MIX$_0$ & 100/0 & 197.3(158.5) & 0.349 & 2.35915 & 0.822 & 1.4e-15/1.4e-15 & 4.9(21.1\%) & 55.5(63.8\%) & 5.7(7.5\%) \\
 & MIX$_{10^{-7}}$ & 100/0 & 197.4(159.3) & 0.315 & 2.35915 & 0.816 & 1.0e-07/2.0e-15 & 4.2(13.5\%) & 55.3(67.9\%) & 5.6(7.8\%) \\
 & MIX$_{10^{-8}}$ & 100/0 & 195.6(157.3) & 0.333 & 2.35915 & 0.819 & 1.0e-08/1.9e-15 & 4.9(16.6\%) & 56.5(65.5\%) & 5.6(7.6\%) \\
\addlinespace
$0.10$ & RPNCGH & 100/0 & 630.5(221.3) & 0.308 & 4.10724 & 0.870 & 1.2e-14/1.2e-14 & 1.6(22.1\%) & 31.6(53.0\%) & -- \\
 & MIX$_0$ & 100/0 & 123.6(97.7) & 0.159 & 4.10724 & 0.869 & 9.8e-16/9.8e-16 & 4.6(24.2\%) & 40.0(55.3\%) & 6.7(10.1\%) \\
 & MIX$_{10^{-7}}$ & 100/0 & 122.5(97.8) & 0.161 & 4.10724 & 0.864 & 1.0e-07/2.0e-15 & 4.6(19.0\%) & 40.9(58.1\%) & 6.6(10.1\%) \\
 & MIX$_{10^{-8}}$ & 100/0 & 123.8(98.9) & 0.173 & 4.10724 & 0.865 & 1.0e-08/1.7e-15 & 5.5(24.4\%) & 39.7(53.0\%) & 6.6(9.4\%) \\
\addlinespace
$0.20$ & RPNCGH & 100/0 & 580.8(224.0) & 0.271 & 7.15027 & 0.904 & 1.2e-15/1.2e-15 & 1.8(27.1\%) & 24.9(45.7\%) & -- \\
 & MIX$_0$ & 100/0 & 104.9(85.1) & 0.102 & 7.15027 & 0.904 & 1.5e-15/1.5e-15 & 5.3(39.1\%) & 31.8(35.8\%) & 7.2(11.7\%) \\
 & MIX$_{10^{-7}}$ & 100/0 & 105.6(86.6) & 0.100 & 7.15027 & 0.897 & 1.0e-07/2.4e-15 & 5.6(33.8\%) & 31.6(38.3\%) & 6.9(12.0\%) \\
 & MIX$_{10^{-8}}$ & 100/0 & 104.1(84.5) & 0.113 & 7.15027 & 0.897 & 1.0e-08/2.2e-15 & 6.9(42.2\%) & 30.6(33.4\%) & 7.1(10.5\%) \\
\addlinespace
$0.50$ & RPNCGH & 100/0 & 830.6(590.6) & 0.467 & 14.878 & 0.937 & 1.4e-14/1.4e-14 & 2.0(23.8\%) & 22.5(57.2\%) & -- \\
 & MIX$_0$ & 100/0 & 146.7(114.2) & 0.092 & 14.878 & 0.937 & 1.1e-15/1.1e-15 & 3.8(41.6\%) & 23.1(28.2\%) & 9.1(12.2\%) \\
 & MIX$_{10^{-7}}$ & 100/0 & 149.5(116.5) & 0.110 & 14.878 & 0.924 & 1.0e-07/1.9e-15 & 4.6(43.0\%) & 24.3(27.3\%) & 6.3(12.3\%) \\
 & MIX$_{10^{-8}}$ & 96/4 & 147.8(115.3) & 0.141 & 14.878 & 0.924 & 1.0e-08/1.6e-15 & 6.3(55.1\%) & 24.3(20.9\%) & 6.3(10.0\%) \\
\bottomrule
\end{tabular}}
\end{table*}

\paragraph{Discussion.}
 Tables~\ref{tab:cms_varying_n}--\ref{tab:cms_varying_mu}
 show that all compared methods reach nearly the same objective values,
 sparsity levels, and feasibility residuals. This is expected, since both
 RPNCGH and MIX use a ManPG-type tangent proximal step and therefore tend to
 identify similar sparse patterns and converge to comparable critical points.

The main difference is computational efficiency. When the ambient dimension is
large but the identified intersection has small intrinsic dimension, MIX
computes Newton--CG corrections on \(\overline{\M}_k\), and the search is
performed in a much lower-dimensional manifold. This leads to
substantially fewer Newton iterations and lower wall-clock time, as shown in
Tables~\ref{tab:cms_varying_n} and~\ref{tab:cms_varying_mu}. For example,
in the \(\mu\)-sweep, the speedup of \(\mathrm{MIX}_0\) over RPNCGH increases
from \(0.605/0.349\approx1.7\) at \(\mu=0.05\) to
\(0.467/0.092\approx5.1\) at \(\mu=0.5\), as the solutions become sparser.

The three perturbation levels give broadly stable objective values and
runtimes. However, the perturbed cases may require slightly more time or more
outer iterations because the induced intersection problems can be more
ill-conditioned. This can lead to \texttt{ACC} rather than \texttt{TOL} in a
few cases. By contrast, \(\mathrm{MIX}_0\) reaches \texttt{TOL} in all CM
tests. The off-diagonal perturbation used to recover transversality may
slightly weaken sparsity, since \(\Delta_{ij}\ne0\) enforces
\(\langle X_i,X_j\rangle=\Delta_{ij}\) and hence requires support overlap
between columns \(i\) and \(j\), which explains the slightly denser solutions
observed for the perturbed variants.

The component timings show that TAPR is visible but not the dominant cost in
the Newton correction. Consistent with Table~\ref{tab:mix-complexity}, the
main cost of the Newton correction comes from the tCG solve together with the construction
and application of the tangent-projection and Riemannian-Hessian operators.
For the CM objective, a Euclidean Hessian--vector product is relatively
cheap; the more expensive part is forming and applying the geometry-dependent
operators on the identified intersection. In the rank sweep of
Table~\ref{tab:cms_varying_r}, the MProx/SSN cost and the cost of these
geometric ingredients become more prominent as \(r\) and the active set grow.
These costs partly offset the savings from fewer Newton steps and explain why
the rank-sweep speedups are not monotone in \(r\).

For the highest-rank CM case in Table~\ref{tab:cms_varying_r}
(\(r=15\)), the \(10^{-12}\)--\(10^{-11}\)-level feasibility residuals should
be interpreted as a subproblem-solve accuracy effect. In some cases, the
semismooth-Newton method for the \(\ell_1\) ManPG subproblem becomes severely
ill-conditioned, and the ManPG predictor is not reliably polished to the
\(10^{-15}\) level. Improving the robustness of the tangent-subproblem solver
for high-rank \(\ell_1\) instances is left for future work.

\paragraph{Geometry diagnostics.}
We next numerically check   Assumption~\ref{assump:row-cover} and the effects of perturbations on conditioning.
The geometry diagnostics are computed only on high-accuracy solved instances:
we require that \(\omega_k<10^{-8}\) and the feasibility
  \( \|Y^\top Y-(I_r+\Delta)\|_{\mathrm{F}}<10^{-14}\), where \(Y\) is the final ManPG predictor used for support
  identification. Runs that do not meet this accuracy
criterion are not counted;
instead, they are replaced by new random initializations until
50 high-accuracy samples are collected for each setting.

Table~\ref{tab:cms-geometry-diagnostics}
reports the unperturbed case \(\Delta=0\).
\input{shared/tables/table_cms_geometry_diagnostics.tex}

The columns in Table~\ref{tab:cms-geometry-diagnostics} are tied to the
support graph \(\mathcal G(Y)=(\mathcal V,\mathcal E)\) in
\eqref{eq:graph}. The column \({\rm Sp}\) is the zero fraction of the final
support mask. Let \(q_\ell=|J_\ell^Y|\) be the number of active columns in row
\(\ell\), with \(J_\ell^Y\) defined before \eqref{eq:graph}. Then
\(q_{\rm avg}=n^{-1}\sum_\ell q_\ell\), and \(p_2+p_3\) is the fraction of
rows for which \(q_\ell=2\) or \(q_\ell=3\). The edge density is
\(|\mathcal E_{\rm off}|/\binom{r}{2}\), where
\(\mathcal E_{\rm off}\) is the off-diagonal edge set in
Assumption~\ref{assump:row-cover}. The \(2/3\)-cov. column
reports the mean \(2/3\)-row-cover ratio associated with
Assumption~\ref{assump:row-cover}: the fraction of present off-diagonal edges
that are witnessed by rows with two or three active columns. Missing edges,
including edges incident to isolated columns, are not included in this ratio.

This explains why \(p_2+p_3\) can be small while the \(2/3\)-cover ratio is
close to one. The former uses all \(n\) rows as its denominator, and most rows
may be zero rows or singleton-support rows. The latter is conditional on the
much smaller set \(\mathcal E_{\rm off}\), whose size is at most
\(\binom{r}{2}\). Hence a small number of rows with \(q_\ell=2\) or
\(q_\ell=3\) can still witness nearly all present off-diagonal edges.

The final column reports the constant-rank test
\(\operatorname{rank}(A_{\gamma(Y)})=\gamma(Y)\), where
\(\gamma(Y)=|\mathcal E|\) is defined in \eqref{def:intersect_number} and
\(\dim(\mathbb S_{r,\gamma}(Y))=\gamma(Y)\) is given by
\eqref{eq:dim_sparse_set_symmetric}. {\color{blue}Numerically,
\(A_{\gamma(Y)}\) is obtained by removing from the matrix \(A_x\) defined
in \eqref{eq:matrix_A} the zero rows associated with missing support-graph
edges.}\footnote{A singular value \(s\) is counted as nonzero if
\(s>\epsilon_{\rm mach}\,\sigma_{\max}r(r+1)/2\), where
\(\epsilon_{\rm mach}\) is the double-precision machine epsilon
\((\epsilon_{\rm mach}\approx2.22\times10^{-16})\) and \(\sigma_{\max}\) is the
largest singular value of the tested matrix.} This condition is the constant-rank diagnostic used for the
clean-intersection construction in Theorem~\ref{thm:Mg-intersection-submanifold}.

Table~\ref{tab:cms-geometry-diagnostics} shows that the constant-rank test
\(\operatorname{rank}(A_{\gamma(Y)})=\gamma(Y)\) holds for all accepted
samples in the displayed CM settings.  The 2/3-cover ratios are also close to one in most settings, which is
qualitatively consistent with the high-probability result in
Proposition~\ref{prop:row-cover-bernoulli}.  
The main exception is the rank-sweep case \(n=256,r=10\), where the
2/3-cover ratio is $0.444$; nevertheless, the full-rank diagnostic in
Table~\ref{tab:cms-perturbation-conditioning} is satisfied for this case.
\input{shared/tables/table_cms_perturbation_conditioning.tex}

Table~\ref{tab:cms-perturbation-conditioning}
reports the full-rank transversality test
\(\operatorname{rank}(A_x)=r(r+1)/2\). The nonzero perturbations make this full-rank test hold in almost all
tested settings, while the high-rank case \(r=15\) remains sensitive at
\(\|\Delta\|_{\mathrm F}=10^{-8}\). The main numerical cost of the perturbation
is conditioning: the median \( L_\kappa=\log_{10}\kappa(A_\gamma)\) generally increases
when \(\Delta\ne0\).  The perturbation also affects the conditioning of the local intersection
geometry; see \eqref{eq:Rieman_Hessian}. In particular, a very small nonzero perturbation may restore
full rank while leaving the reduced constraint Jacobian $A_\gamma$
nearly rank deficient. This effect is particularly visible for $r=15$,
where $\|\Delta\|_{\mathrm F}=10^{-8}$ gives a substantially larger
median $L_\kappa=\log_{10}\kappa(A_\gamma)$ than
$\|\Delta\|_{\mathrm F}=10^{-7}$.

Interestingly, this severe geometric ill-conditioning does not lead to
a comparable increase in the iteration counts in Tables~\ref{tab:cms_varying_n}--\ref{tab:cms_varying_mu}. One possible
explanation is that both the RN-SLRA correction used in TAPR and the
Newton-CG step employ regularization, which stabilizes the corresponding
linear systems and prevents near-degeneracy of the intersection geometry
from directly translating into an ill-conditioned Newton correction.

\FloatBarrier

\subsection{Sparse PCA problem}
For sparse PCA, we use the row-sparse spiked covariance model
in~\cite{cai2013sparse}, where the target principal subspace is
row sparse. We therefore solve the \(\ell_{2,1}\)-regularized model
\[
\min_{X^\top X=I_r}
-\operatorname{Tr}(X^\top A X)+\mu\|X\|_{2,1}.
\]
The \(\ell_{2,1}\) norm is a standard row-sparsity penalty for
orthogonality-constrained sparse PCA and coordinate-independent sparse
estimation~\cite{xiao2020exact,chen2010coordinate}. In this case the identified
intersection is a row-support Stiefel manifold, so the Newton correction uses
the polar retraction  on the reduced Stiefel manifold and the Riemannian Hessian is computed
by \cite[Eq.(7.29)]{boumal2023introduction}. Since the row-support active
manifold is automatically transverse by Remark~\ref{rmk:identification}, we
solve the unperturbed problem \eqref{prob:nonsmooth_stiefel}.

The true row support is \(S_\star=\{1,\ldots,s\}\). We obtain
\(V_\star\) from the thin QR factor of a row-sparse matrix whose first \(s\)
rows are \(i^2g_i^\top\), sample
\(Y=H\operatorname{diag}(\sqrt{\lambda_1},\ldots,\sqrt{\lambda_r})V_\star^\top+E\),
and set \(A=Y^\top Y/N\), using the Gaussian model and spike strengths
from~\cite{cai2013sparse}. We set \(s=\max\{\lceil0.05n\rceil,5r\}\),
\(N=\lceil r(s-r)+s\log(\mathrm e n/s)\rceil\), and
\(\mu=\eta\|A\|_2/\sqrt{s}\). The dimension sweep varies \(n=500\), \(1000\),
\(2000\), and \(5000\) with \(r=5\), \(\eta=0.45\); the rank sweep varies
\(r=1\), \(5\), \(10\), and \(20\) with \(n=2000\), \(\eta=0.45\); and the
regularization sweep varies \(\eta=0.35,\ldots,0.55\) with spacing \(0.05\) and
fixed \(n=2000\), \(r=5\). Each setting has 50 trials.

For recovery diagnostics, we report
\(\|XX^\top-V_\star V_\star^\top\|_{\mathrm F}^2\) and its normalized value,
together with the precision, recall, and F1 score of
\(\widehat S=\{i:\|X_{i,:}\|_2>10^{-6}\}\). Figure~\ref{fig:spca-recovery-eta-path}
shows these metrics along the regularization path and the convergence history
for one representative instance. MIX achieves stable recovery along this path: the
F1 score and precision improve as \(\eta\) increases, recall remains stable,
and the subspace loss stays small. The convergence plot also shows that ManPG
stagnates above the tolerance and does not reach high accuracy on this
instance. Both MIX and RPNCGH exhibit local superlinear convergence.

\begin{figure}[!t]
\centering
\begin{minipage}[t]{0.46\textwidth}
\centering
\includegraphics[width=\linewidth]{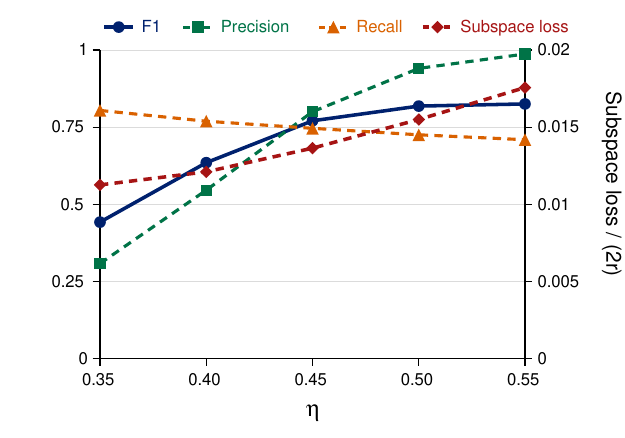}
\scriptsize (a) Recovery diagnostics.
\end{minipage}\hfill
\begin{minipage}[t]{0.48\textwidth}
\centering
\includegraphics[width=\linewidth]{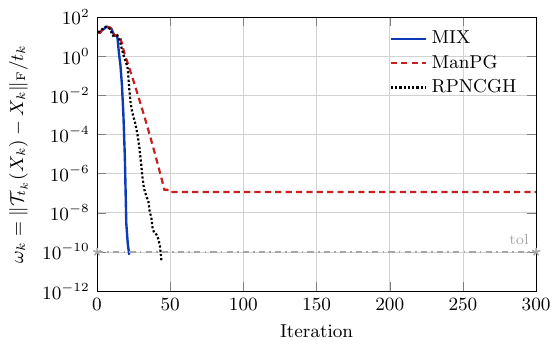}
\scriptsize (b) Convergence for \(n=2000\), \(r=5\), and \(\eta=0.45\).
\end{minipage}
\caption{Recovery and convergence diagnostics for SPCA.}
\label{fig:spca-recovery-eta-path}
\end{figure}

Table~\ref{tab:spca-combined} summarizes the three SPCA sweeps
using the notation of Tables~\ref{tab:cms_varying_n}--\ref{tab:cms_varying_mu}.
Only RPNCGH and MIX$_0$ are reported; all runs terminate at TOL, so the omitted
T/A column is 100/0 throughout.
\begin{table}[!t]
\centering
\scriptsize
\caption{SPCA performance comparison for the dimension $n$, rank $r$, and regularization $\mu=\eta\|A\|_2/\sqrt{s}$ sweeps.}
\label{tab:spca-combined}
\renewcommand{\arraystretch}{0.92}
\resizebox{\textwidth}{!}{%
\begin{tabular}{@{}lllccccccc@{}}
\toprule
\begin{tabular}{@{}c@{}}Var.\\fixed\end{tabular} & Val & Alg. & Iter(New) & Time & Obj & Sp & Orth & MProx & tCG \\
\midrule
\multirow{8}{*}{\begin{tabular}{@{}c@{}}$n$\\$r=5$\\$\eta=0.45$\end{tabular}} & $500$ & RPNCGH & 43.3(25.0) & 0.022 & -61.7027 & 0.95112 & 1.3e-15 & 1.4(17.6\%) & 2.4(44.0\%) \\
 &  & MIX$_0$ & 19.3(6.6) & 0.008 & -61.7027 & 0.95112 & 8.5e-16 & 2.1(29.7\%) & 10.7(12.4\%) \\
\addlinespace
 & $1000$ & RPNCGH & 38.9(21.1) & 0.034 & -63.3454 & 0.95136 & 1.3e-15 & 1.4(16.3\%) & 2.6(44.6\%) \\
 &  & MIX$_0$ & 19.4(6.3) & 0.015 & -63.3454 & 0.95136 & 8.4e-16 & 2.0(23.6\%) & 17.0(24.0\%) \\
\addlinespace
 & $2000$ & RPNCGH & 35.6(17.3) & 0.058 & -64.2027 & 0.953 & 1.6e-15 & 1.5(17.1\%) & 2.6(43.8\%) \\
 &  & MIX$_0$ & 20.2(6.2) & 0.041 & -64.2027 & 0.953 & 1.0e-15 & 2.0(25.8\%) & 21.1(34.7\%) \\
\addlinespace
 & $5000$ & RPNCGH & 35.6(15.4) & 0.308 & -64.0849 & 0.953372 & 2.0e-15 & 1.5(8.6\%) & 2.9(49.5\%) \\
 &  & MIX$_0$ & 22.2(6.7) & 0.141 & -64.0849 & 0.953372 & 1.3e-15 & 2.0(12.8\%) & 32.1(34.1\%) \\
\addlinespace[1.0ex]
\hline
\addlinespace[0.6ex]
\multirow{8}{*}{\begin{tabular}{@{}c@{}}$r$\\$n=2000$\\$\eta=0.45$\end{tabular}} & $1$ & RPNCGH & 41.9(3.7) & 0.022 & -8.09796 & 0.9839 & 2.1e-11 & 1.7(12.1\%) & 1.1(11.9\%) \\
 &  & MIX$_0$ & 22.2(12.6) & 0.014 & -8.09796 & 0.9839 & 2.7e-16 & 2.6(12.2\%) & 1.1(6.7\%) \\
\addlinespace
 & $5$ & RPNCGH & 35.6(17.3) & 0.057 & -64.2027 & 0.953 & 1.6e-15 & 1.5(18.2\%) & 2.6(43.2\%) \\
 &  & MIX$_0$ & 20.2(6.2) & 0.041 & -64.2027 & 0.953 & 1.0e-15 & 2.0(23.6\%) & 21.1(37.2\%) \\
\addlinespace
 & $10$ & RPNCGH & 35.4(17.2) & 0.151 & -135.841 & 0.94477 & 2.7e-15 & 1.5(20.1\%) & 2.6(43.8\%) \\
 &  & MIX$_0$ & 19.8(6.2) & 0.079 & -135.841 & 0.94477 & 1.4e-15 & 2.6(36.2\%) & 25.6(8.9\%) \\
\addlinespace
 & $20$ & RPNCGH & 35.3(16.3) & 0.526 & -284.257 & 0.93138 & 4.6e-15 & 1.5(34.8\%) & 2.6(45.0\%) \\
 &  & MIX$_0$ & 20.1(6.0) & 0.240 & -284.257 & 0.93138 & 2.3e-15 & 2.0(57.8\%) & 45.4(12.3\%) \\
\addlinespace[1.0ex]
\hline
\addlinespace[0.6ex]
\multirow{10}{*}{\begin{tabular}{@{}c@{}}$\eta$\\$n=2000$\\$r=5$\end{tabular}} & $0.35$ & RPNCGH & 39.5(22.1) & 0.077 & -67.883 & 0.86549 & 2.3e-15 & 1.4(16.9\%) & 2.6(46.2\%) \\
 &  & MIX$_0$ & 19.7(6.6) & 0.102 & -67.883 & 0.86549 & 1.5e-15 & 2.5(10.1\%) & 43.5(65.5\%) \\
\addlinespace
 & $0.4$ & RPNCGH & 37.6(19.8) & 0.068 & -66.0174 & 0.92811 & 1.7e-15 & 1.4(16.5\%) & 2.7(46.4\%) \\
 &  & MIX$_0$ & 20.0(6.6) & 0.073 & -66.0174 & 0.92811 & 1.2e-15 & 2.6(13.0\%) & 29.5(56.0\%) \\
\addlinespace
 & $0.45$ & RPNCGH & 35.6(17.3) & 0.063 & -64.2027 & 0.953 & 1.6e-15 & 1.5(16.7\%) & 2.6(43.4\%) \\
 &  & MIX$_0$ & 20.2(6.2) & 0.042 & -64.2027 & 0.953 & 1.0e-15 & 2.0(22.4\%) & 21.1(37.1\%) \\
\addlinespace
 & $0.5$ & RPNCGH & 35.7(16.9) & 0.055 & -62.414 & 0.96146 & 1.4e-15 & 1.5(17.8\%) & 2.5(41.6\%) \\
 &  & MIX$_0$ & 20.3(6.2) & 0.030 & -62.414 & 0.96146 & 9.2e-16 & 2.0(23.7\%) & 14.1(28.5\%) \\
\addlinespace
 & $0.55$ & RPNCGH & 37.1(17.8) & 0.057 & -60.6432 & 0.9642 & 1.5e-15 & 1.5(17.7\%) & 2.5(42.1\%) \\
 &  & MIX$_0$ & 20.8(6.4) & 0.029 & -60.6432 & 0.9642 & 8.3e-16 & 2.1(31.6\%) & 9.2(19.1\%) \\
\bottomrule
\end{tabular}}
\end{table}
\paragraph{Discussion.}
The table shows the same qualitative pattern as the CM experiments, with a
particularly clear iteration reduction for MIX$_0$. RPNCGH and MIX$_0$ reach
essentially identical objective values, sparsity ratios, and orthogonality
residuals, so the comparison is mainly about computational efficiency. Across
the dimension $n$ and rank $r$ sweeps, MIX$_0$ reduces the outer iterations to
about one half and the Newton-CG calls to about one third of those of RPNCGH, which
usually translates into faster runtimes, with speedups from about \(1.4\) to
\(2.9\) in these tests. Along the regularization path
\(\mu=\eta\|A\|_2/\sqrt{s}\), the gain becomes clearer as larger \(\eta\)
identifies sparser row supports: MIX$_0$ then performs Newton corrections on a
smaller active Stiefel manifold and its speedup grows to about \(2.0\) at
\(\eta=0.55\). For the less regularized cases \(\eta=0.35\) and \(0.40\), the
runtime is comparable or slower because the identified row support is larger
and the tCG cost dominates the MIX$_0$ runtime.

\section{Conclusion}

{\color{blue}
We developed a manifold-identification approach to Newton acceleration
for sparse composite optimization on the Stiefel manifold. After the
ManPG tangent proximal step identifies the active manifold, the Newton
correction is performed on its intersection with a nearby Stiefel
manifold. Generic off-diagonal perturbations provide transversality for
identification, while the clean-intersection analysis characterizes the
moving manifolds used by the Newton step.

The resulting MIX method combines a globally safeguarded ManPG step
with Newton-CG corrections on the identified intersections. Globally,
the objective values decrease monotonically and the ManPG residual
converges to zero.
Locally, if an accumulation point $X^*$ is nondegenerate, $g$ is
$C^2$-partly smooth at $X^*$, $\M_g(X^*)$ and
$\St_\Delta(n,r)$ intersect transversely at $X^*$, and the SOSC holds,
then the whole sequence converges to $X^*$, the ManPG predictors identify
$\M_g(X^*)$ in finitely many iterations, and $\{X_k\}$ converges
Q-superlinearly.

Numerical experiments on compressed modes and sparse PCA demonstrate
the computational advantage of performing the Newton correction on the
identified low-dimensional intersection. Future work includes reducing
the use of perturbations, relaxing the SOSC, and extending the framework
to broader equality-constrained manifold problems.
}

%% file: shared/tables/table_cms_geometry_diagnostics.tex
\begin{table}[!ht]
\centering
\small
\setlength{\tabcolsep}{7pt}
\caption{Geometry diagnostics for the CM problem. Percentages are computed over 50 instances. The \(2/3\)-cov. column reports the mean \(2/3\)-row-cover ratio associated with Assumption~\ref{assump:row-cover}. The CR column reports the percentage of accepted samples satisfying the constant-rank test \(\operatorname{rank}(A_{\gamma(Y)})=\gamma(Y)\).}
\label{tab:cms-geometry-diagnostics}
\begin{tabular}{@{}llcccccc@{}}
\toprule
\begin{tabular}{@{}c@{}}Var.\\fixed\end{tabular} & Val & Sp & \(q_{\rm avg}\) & \(p_2+p_3\) & dens. & \(2/3\)-cov. & CR (\%) \\
\midrule
\multirow{5}{*}{\begin{tabular}{@{}c@{}}\(n\)\\\(r=5\)\\\(\mu=0.1\)\end{tabular}} & \(128\) & 0.827 & 0.86 & 0.074 & 0.364 & 0.961 & 100 \\
 & \(256\) & 0.850 & 0.75 & 0.052 & 0.330 & 0.991 & 100 \\
 & \(512\) & 0.869 & 0.65 & 0.039 & 0.290 & 0.962 & 100 \\
 & \(1024\) & 0.888 & 0.56 & 0.021 & 0.214 & 0.978 & 100 \\
 & \(2048\) & 0.902 & 0.49 & 0.015 & 0.174 & 1.000 & 100 \\
\addlinespace[2pt]
\multirow{4}{*}{\begin{tabular}{@{}c@{}}\(r\)\\\(n=256\)\\\(\mu=0.1\)\end{tabular}} & \(3\) & 0.856 & 0.43 & 0.014 & 0.240 & 1.000 & 100 \\
 & \(5\) & 0.850 & 0.75 & 0.052 & 0.330 & 0.991 & 100 \\
 & \(10\) & 0.758 & 2.42 & 0.559 & 1.000 & 0.444 & 100 \\
 & \(15\) & 0.809 & 2.85 & 0.999 & 0.287 & 0.995 & 100 \\
\addlinespace[2pt]
\multirow{5}{*}{\begin{tabular}{@{}c@{}}\(\mu\)\\\(n=512\)\\\(r=5\)\end{tabular}} & \(0.01\) & 0.438 & 2.81 & 0.707 & 1.000 & 1.000 & 100 \\
 & \(0.05\) & 0.822 & 0.89 & 0.069 & 0.458 & 0.958 & 100 \\
 & \(0.10\) & 0.869 & 0.65 & 0.039 & 0.290 & 0.962 & 100 \\
 & \(0.20\) & 0.904 & 0.48 & 0.015 & 0.192 & 1.000 & 100 \\
 & \(0.50\) & 0.937 & 0.31 & 0.000 & 0.004 & 1.000 & 100 \\
\bottomrule
\end{tabular}
\end{table}

%% file: shared/tables/table_cms_perturbation_conditioning.tex
\begin{table}[!htb]
\centering
\footnotesize
\setlength{\tabcolsep}{7pt}
\caption{Effect of off-diagonal perturbations on CM transversality and conditioning. Each entry is computed over 50 instances. Rows use the same block format as Table~\ref{tab:cms-geometry-diagnostics}. The FR column reports the percentage of accepted samples satisfying the full-rank transversality test \(\operatorname{rank}(A_x)=r(r+1)/2\). Here \(L_\kappa=\log_{10}\kappa(A_\gamma)\), computed over samples satisfying \(\operatorname{rank}(A_{\gamma(Y)})=\gamma(Y)\).}
\label{tab:cms-perturbation-conditioning}
\begin{tabular}{@{}llcccccc@{}}
\toprule
& & \multicolumn{3}{c}{FR (\%)} & \multicolumn{3}{c}{median \(L_\kappa\)} \\
\cmidrule(lr){3-5}\cmidrule(l){6-8}
\begin{tabular}{@{}c@{}}Var.\\fixed\end{tabular} & Val & \(0\) & \(10^{-8}\) & \(10^{-7}\) & \(0\) & \(10^{-8}\) & \(10^{-7}\) \\
\midrule
\multirow{5}{*}{\begin{tabular}{@{}c@{}}\(n\)\\\(r=5\)\\\(\mu=0.1\)\end{tabular}} & \(128\) & 6 & 96 & 100 & 0.60 & 0.75 & 0.75 \\
 & \(256\) & 0 & 100 & 100 & 0.72 & 0.87 & 0.87 \\
 & \(512\) & 0 & 100 & 100 & 0.84 & 0.99 & 0.99 \\
 & \(1024\) & 0 & 100 & 100 & 0.61 & 0.96 & 0.96 \\
 & \(2048\) & 0 & 100 & 100 & 0.61 & 1.08 & 0.89 \\
\addlinespace[2pt]
\multirow{3}{*}{\begin{tabular}{@{}c@{}}\(r\)\\\(n=256\)\\\(\mu=0.1\)\end{tabular}} & \(3\) & 10 & 100 & 100 & 0.57 & 0.80 & 0.72 \\
 & \(10\) & 100 & 100 & 100 & 0.87 & 0.87 & 2.89 \\
 & \(15\) & 0 & 80 & 100 & 1.20 & 9.11 & 2.40 \\
\addlinespace[2pt]
\multirow{4}{*}{\begin{tabular}{@{}c@{}}\(\mu\)\\\(n=512\)\\\(r=5\)\end{tabular}} & \(0.01\) & 100 & 100 & 100 & 0.39 & 0.39 & 0.39 \\
 & \(0.05\) & 2 & 100 & 100 & 1.05 & 1.05 & 1.05 \\
 & \(0.20\) & 0 & 100 & 100 & 0.63 & 0.93 & 0.93 \\
 & \(0.50\) & 0 & 100 & 100 & 0.00 & 0.85 & 0.85 \\
\bottomrule
\end{tabular}
\end{table}

%% file: shared/appendix.tex
\appendix
\section*{Appendix}
\addcontentsline{toc}{section}{Appendix}
\input{shared/appendix_preliminaries}
\input{shared/appendix_identification}
\input{shared/appendix_geometry_examples}
\input{shared/appendix_optimization_intersection}
\section{Geometry and local corrections for $\St_\Delta(n,r)$}
\label{app:st-delta-geometry}

\subsection{Normal space and tangent projection of $\St_\Delta$}
\label{sec:append:normal}

Let
\(\St_\Delta(n,r)=\{X\in\mathbb R^{n\times r}:X^\top X=D\}\), where
\(D=I_r+\Delta\succ0\).
The tangent space at $X\in\St_\Delta(n,r)$ is
\[
    \T_X\St_\Delta(n,r)
    =
    \{Z\in\mathbb R^{n\times r}:X^\top Z+Z^\top X=0\}.
\]
Moreover, its Euclidean normal space is
\[
    \rmN_X\St_\Delta(n,r)
    =
    \{XS:S\in\mathbb S^r\}.
\]

Indeed, define $F(X):=X^\top X-D$. Then
\(\rmD F(X)[Z]=X^\top Z+Z^\top X\).
For any $\Lambda\in\mathbb S^r$,
\[
\begin{aligned}
    \langle \rmD F(X)[Z],\Lambda\rangle
    &=
    \langle X^\top Z+Z^\top X,\Lambda\rangle  \\
    &=
    2\langle Z,X\Lambda\rangle .
\end{aligned}
\]
Hence $\rmD F(X)^*[\Lambda]=2X\Lambda$, and the normal space is
$\range(\rmD F(X)^*)=\{XS:S\in\mathbb S^r\}$.

\paragraph{Orthogonal projection onto the perturbed tangent space.}
Given $Z\in\mathbb R^{n\times r}$, write its tangent projection as
\(H=Z-XS\), where \(S\in\mathbb S^r\).
The tangency condition $H\in\T_X\St_\Delta(n,r)$ gives
\(X^\top(Z-XS)+(Z-XS)^\top X=0\),
or equivalently,
\begin{equation}\label{eq:lyap-compact}
    DS+SD=X^\top Z+Z^\top X .
\end{equation}
Since $D\succ0$, the Lyapunov operator
$\mathcal L_D(S):=DS+SD$ is invertible on $\mathbb S^r$. Hence
\eqref{eq:lyap-compact} has a unique symmetric solution $S(Z)$, and
\begin{equation}\label{eq:proj-TX-StDelta-final}
    \Proj_{\T_X\St_\Delta}(Z)=Z-XS(Z).
\end{equation}

If $D=Q\Lambda Q^\top$ with
$\Lambda=\operatorname{diag}(\lambda_1,\ldots,\lambda_r)$, and if
\(\widetilde A:=Q^\top(X^\top Z+Z^\top X)Q\) and
\(\widetilde S:=Q^\top S(Z)Q\),
then
\begin{equation}\label{eq:lypunov_eigen}
    \widetilde S_{ij}
    =
    \frac{\widetilde A_{ij}}{\lambda_i+\lambda_j},
    \qquad
    S(Z)=Q\widetilde S Q^\top .
\end{equation}
In the unperturbed case $\Delta=0$, one has $D=I_r$ and therefore
\(S(Z)=\frac12(X^\top Z+Z^\top X)=\operatorname{sym}(X^\top Z)\),
so that
\(\Proj_{\T_X\St}(Z)=Z-X\operatorname{sym}(X^\top Z)\),
which is the classical Stiefel tangent projection.

\subsection{Euclidean projection onto $\St_\Delta(n,r)$}
\label{sec:append_projection_to_St}

We derive the Euclidean projection onto
\(\St_\Delta(n,r):=\{X\in\mathbb R^{n\times r}: X^\top X=I_r+\Delta\}\).
Let \(D:=I_r+\Delta\).
Throughout this subsection, we assume that $D\succ0$ and that $Y$ has full column
rank. The projection is defined by
\[
    \Proj_{\St_\Delta}(Y)
    :=
    \arg\min_{X\in\St_\Delta(n,r)}
    \frac12\|X-Y\|_{\mathrm F}^2 .
\]

Set $X=QD^{1/2}$, where $Q\in\St(n,r)$. Then
\[
\|X-Y\|_{\mathrm F}^2
=
\|QD^{1/2}-Y\|_{\mathrm F}^2
=
\operatorname{tr}(D)+\|Y\|_{\mathrm F}^2
-2\operatorname{tr}\bigl(Q^\top YD^{1/2}\bigr).
\]
Thus the projection problem is equivalent to the orthogonal Procrustes problem
\[
    \max_{Q\in\St(n,r)}
    \operatorname{tr}\bigl(Q^\top A\bigr),
    \qquad
    A:=YD^{1/2}.
\]
Since $Y$ has full column rank, so does $A$. Hence the unique maximizer is the
orthogonal polar factor of $A$, namely \(Q^*=A(A^\top A)^{-1/2}\).
Consequently,
\[
\begin{aligned}
    \Proj_{\St_\Delta}(Y)
    &=
    Q^*D^{1/2}  \\
    &=
    YD^{1/2}
    \left(D^{1/2}Y^\top YD^{1/2}\right)^{-1/2}
    D^{1/2}.
\end{aligned}
\]
Therefore,
\begin{equation}\label{eq:proj_closed_StDelta_app}
    \Proj_{\St_\Delta}(Y)
    =
    YD^{1/2}
    \left(D^{1/2}Y^\top YD^{1/2}\right)^{-1/2}
    D^{1/2}.
\end{equation}

Equivalently, if $YD^{1/2}=U\Sigma V^\top$ is the thin singular value
decomposition, then
\[
    \Proj_{\St_\Delta}(Y)=UV^\top D^{1/2}.
\]
When $D=I_r$, this reduces to the classical polar projection onto the standard
Stiefel manifold.

\subsection{A Sylvester-corrected one-step correction for $\St_\Delta(n,r)$}
\label{app:sylvester_corrected_phi2}

We seek a one-step correction that is first-order matched to the projection onto
\(\St_\Delta(n,r)=\{X\in\mathbb R^{n\times r}:X^\top X=D\}\), \(D\succ0\).
Recall from Subsection~\ref{sec:append_projection_to_St} that the exact Euclidean projection of $Y$ has the form
\(X=Y(I+\Gamma)^{-1}\),
where $\Gamma\in\mathbb{S}^r$ is the Lagrangian multiplier. 
  Motivated by this, we consider a one-step local approximation of the form
\[
\phi_2(Y)=Y(I-S(Y)),
\qquad S(Y)\in\mathbb S^r,
\]
where $S(Y)$ is a small symmetric matrix approximating the multiplier-induced factor
$(I+\Gamma)^{-1}-I$.

To determine $S(Y)$, we enforce first-order cancellation of the feasibility residual.
Let
\(M:=Y^\top Y\) and \(E:=M-D=Y^\top Y-D\).
For a symmetric matrix $S$, define \(X^+ := Y(I-S)\).
Then
\begin{align*}
(X^+)^\top X^+ - D
&= (I-S)M(I-S)-D \\
&= M-D - MS - SM + SMS \\
&= E - MS - SM + SMS.
\end{align*}
Near $\St_\Delta(n,r)$, both $E$ and $S$ are small. Replacing $M=D+E$ in the linear terms yields
\[
E-MS-SM
=
E-(DS+SD) - (ES+SE),
\]
and therefore
\[
(X^+)^\top X^+ - D
=
\underbrace{E-(DS+SD)}_{\text{first-order term}}
+\underbrace{\bigl[-(ES+SE)+SMS\bigr]}_{\text{higher-order terms}}.
\]
Thus, the first-order feasibility defect is cancelled by choosing $S$ as the unique symmetric solution of
\begin{equation}\label{eq:Sylvester_eq}
    DS+SD=E=Y^\top Y-D.
\end{equation}

This leads to the following map
\begin{equation}\label{eq:inexact_St_phi_2}
\phi_2(Y)=Y(I-S(Y)),
\qquad
DS(Y)+S(Y)D=Y^\top Y-D.
\end{equation}

\subsubsection*{Efficient evaluation for fixed $D$}

Although \eqref{eq:inexact_St_phi_2} is a Sylvester equation, in our alternating-projection setting
the matrix $D$ is fixed across all iterations. We therefore preprocess $D$ once.

Let \(D=U\Lambda U^\top\), where
\(\Lambda=\mathrm{Diag}(\lambda_1,\ldots,\lambda_r)\) and \(\lambda_i>0\), and define
\(K\in\mathbb R^{r\times r}\) by \(K_{ij}:=1/(\lambda_i+\lambda_j)\).
For any input $Y$, set \(Z:=YU\in\mathbb R^{n\times r}\). Then
\(U^\top(Y^\top Y-D)U = Z^\top Z-\Lambda\).
If we denote $\widehat S:=U^\top S(Y)U$, then \eqref{eq:inexact_St_phi_2} is equivalent to
\(\Lambda\widehat S+\widehat S\Lambda = Z^\top Z-\Lambda\),
whose entrywise solution is
\(\widehat S = K\odot (Z^\top Z-\Lambda)\).
Substituting $S(Y)=U\widehat S U^\top$ into \eqref{eq:inexact_St_phi_2} gives
\(\phi_2(Y)=Y-Z\widehat S U^\top\).

\subsubsection{Expansion of the Sylvester-corrected map}\label{append_subsec:proof_lemma}
\begin{lemma}\label{lem:expansion-phi}
Let $\bar X\in \St_\Delta(n,r)$, with $D=I_r+\Delta\succ0$, and define $\phi_2$ by \eqref{eq:Sylvester}.
Let $\lambda_{\min}:=\lambda_{\min}(D)$ and $\lambda_{\max}:=\lambda_{\max}(D)$, and define
\[
\rho_0:=\sqrt{\lambda_{\max}+\lambda_{\min}}-\sqrt{\lambda_{\max}}.
\]
Then for any $w\in\R^{n\times r}$ satisfying $\|w\|_{\mathrm{F}}\le \rho_0$, we have
\begin{equation}\label{eq:expansion_phi-inexact}
\phi_2(\bar{X}+w)=\bar{X}+\Proj_{\T_{\bar{X}}\St_{\Delta}}(w) +\bar{R}_{\T}(w)+\bar{R}_{\rmN}(w),
\end{equation}
where $\bar{R}_{\T}(w)=O(\|w\|_{\mathrm{F}}^2)$ and $\bar{R}_{\rmN}(w)=O(\|w\|_{\mathrm{F}}^2)$ are smooth mappings on the Frobenius ball $\{w:\|w\|_{\mathrm{F}}\le \rho_0\}$.
Furthermore, if $\Proj_{\rmN_{\bar{X}}\St_{\Delta}}(w)=O(\|w\|_{\mathrm{F}}^2)$, then $\bar{R}_{\T}(w)=O(\|w\|_{\mathrm{F}}^3)$.
\end{lemma}
\begin{proof}[Proof of Lemma~\ref{lem:expansion-phi}]
Fix $\bar X\in\St_\Delta(n,r)$ and let $w$ satisfy $\|w\|_F\le \rho_0$.
Set
\(E(w):=(\bar X+w)^\top(\bar X+w)-D\).
Since $\bar X^\top\bar X=D$, we can write
\(E(w)=B(w)+Q(w)\), where
\(B(w):=\bar X^\top w+w^\top\bar X\in\mathbb S^r\) and
\(Q(w):=w^\top w\in\mathbb S^r\).

Let $\mathcal L_D:\mathbb S^r\to\mathbb S^r$ denote the inverse Sylvester operator
\(\mathcal L_D(E)=S \Longleftrightarrow DS+SD=E\), \(S\in\mathbb S^r\).
Since $D\succ0$, $\mathcal L_D$ is bounded and satisfies
\begin{equation}\label{eq:LDbound}
\|\mathcal L_D(E)\|_F\le \frac{1}{2\lambda_{\min}}\|E\|_F,\qquad \forall E\in\mathbb S^r.
\end{equation}
Define
\(S_1(w):=\mathcal L_D(B(w))\) and \(S_2(w):=\mathcal L_D(Q(w))\).
By linearity, we obtain
\(S(\bar X+w)=\mathcal L_D(E(w))=S_1(w)+S_2(w)\).

We expand
\begin{align*}
\phi_2(\bar X+w)
&=(\bar X+w)\bigl(I-S_1(w)-S_2(w)\bigr) \\
&=\bar X+w-\bar X S_1(w)-\bar X S_2(w)-wS_1(w)-wS_2(w).
\end{align*}
On the other hand, the Euclidean projection onto the tangent space has the form
\[
    \Proj_{\T_{\bar X}\St_\Delta}(w)=w-\bar X S_1(w),
\]
because $S_1(w)\in\mathbb S^r$ solves
\(DS_1(w)+S_1(w)D=B(w)=\bar X^\top w+w^\top\bar X\).
Therefore,
\(\phi_2(\bar X+w)=\bar X+\Proj_{\T_{\bar X}\St_\Delta}(w)+R(w)\),
where
\(R(w):=-\bar X S_2(w)-wS_1(w)-wS_2(w)\).
Define
\(\bar R_\T(w):=\Proj_{\T_{\bar X}\St_\Delta}(R(w))\) and
\(\bar R_\rmN(w):=\Proj_{\rmN_{\bar X}\St_\Delta}(R(w))\).
Then \eqref{eq:expansion_phi-inexact} holds by construction.

We now estimate the remainder orders uniformly on $\{w:\|w\|_F\le \rho_0\}$.
Since $\|\bar X\|_2=\sqrt{\lambda_{\max}}$ and $\|w\|_2\le \|w\|_F$, we have
\[
\|B(w)\|_F \le 2\|\bar X\|_2\|w\|_F \le 2\sqrt{\lambda_{\max}}\|w\|_F,
\qquad
\|Q(w)\|_F \le \|w\|_F^2.
\]
Hence, by \eqref{eq:LDbound},
\[
\|S_1(w)\|_F\le \frac{1}{2\lambda_{\min}}\|B(w)\|_F=O(\|w\|_F),\qquad
\|S_2(w)\|_F\le \frac{1}{2\lambda_{\min}}\|Q(w)\|_F=O(\|w\|_F^2),
\]
and thus
\[
R(w)=-\bar X S_2(w)-wS_1(w)-wS_2(w)=O(\|w\|_F^2).
\]
Since $\Proj_{\T_{\bar X}\St_\Delta}$ and $\Proj_{\rmN_{\bar X}\St_\Delta}$ are fixed linear maps,
\[
\bar R_\T(w)=O(\|w\|_F^2),\qquad \bar R_\rmN(w)=O(\|w\|_F^2),
\]
uniformly for $\|w\|_F\le \rho_0$. Smoothness of $\bar R_\T,\bar R_\rmN$ on the ball follows because
$S(Y)$ depends smoothly on $Y$.

Moreover, for $\|w\|_F\le \rho_0$ we have
\[
\|Y^\top Y-D\|_2 = \|E(w)\|_2 \le \|B(w)\|_2+\|Q(w)\|_2
\le 2\sqrt{\lambda_{\max}}\|w\|_F+\|w\|_F^2 \le \lambda_{\min},
\]
so \eqref{eq:LDbound} also yields $\|S(\bar X+w)\|_2\le \|S(\bar X+w)\|_F\le \tfrac12$,
and hence $I-S(\bar X+w)$ is uniformly invertible on the ball.

For the final claim, assume $\Proj_{\rmN_{\bar X}\St_\Delta}(w)=O(\|w\|_F^2)$ and write
\(w=\eta+\nu\), where \(\eta\in \T_{\bar X}\St_\Delta\) and
\(\nu\in \rmN_{\bar X}\St_\Delta\),
so $\nu=O(\|w\|_F^2)$. Since $\eta\in \T_{\bar X}\St_\Delta$, we have
$\bar X^\top\eta+\eta^\top\bar X=0$, hence
\(B(w)=\bar X^\top\nu+\nu^\top\bar X=O(\|w\|_F^2)\), and therefore
\(S_1(w)=\mathcal L_D(B(w))=O(\|w\|_F^2)\).
Together with $S_2(w)=O(\|w\|_F^2)$, we obtain
\(wS_1(w)=O(\|w\|_F^3)\) and \(wS_2(w)=O(\|w\|_F^3)\).
Finally, since $S_2(w)\in\mathbb S^r$ implies $\bar X S_2(w)\in \rmN_{\bar X}\St_\Delta$, we have
$\Proj_{\T_{\bar X}\St_\Delta}(\bar X S_2(w))=0$, and thus
\[
\bar R_\T(w)=\Proj_{\T_{\bar X}\St_\Delta}\bigl(-wS_1(w)-wS_2(w)\bigr)=O(\|w\|_F^3).
\]
This completes the proof.
\end{proof}

\section{Technical lemmas for Section~\ref{sec:alg_on_intersection}}\label{append:technical_last}

\subsection{Proof of Lemma~\ref{lem:approx_KKT_residual}}
\label{app:proof-approx-KKT-residual}

\begin{proof}[Proof of Lemma~\ref{lem:approx_KKT_residual}]
Let $D:=I_r+\Delta$. Since $X_k\in\St_\Delta$, one has $X_k^\top X_k=D\succ0$.
By the optimality condition of the tangent subproblem~\eqref{eq:tangent_problem},
there exist $G_k\in\partial g(Y_k)$ and $\Lambda_k\in\mathbb S^r$ such that
\[
    \nabla f(X_k)+G_k+\frac1{t_k}V_k+X_k\Lambda_k=0 .
\]
We first show that $\{\Lambda_k\}$ is bounded. Since $V=0$ is feasible for
\eqref{eq:tangent_problem}, the optimality of $V_k$ gives
\[
    \langle \nabla f(X_k),V_k\rangle+\frac{1}{2t_k}\normfro{V_k}^2+g(Y_k)
    \le g(X_k).
\]
The boundedness of $\nabla f$ and the Lipschitz continuity of $g$ then yield
$\normfro{V_k}/t_k\le C$ for some constant $C>0$. Moreover, the Lipschitz
continuity of $g$ implies that   the subgradients $G_k\in\partial g(Y_k)$
are uniformly bounded. Multiplying the optimality condition by $X_k^\top$ gives
\[
    D\Lambda_k
    =
    -X_k^\top\left(\nabla f(X_k)+G_k+\frac1{t_k}V_k\right).
\]
Since $D^{-1}$ and $\{X_k\}$ are bounded, the preceding bounds imply
\begin{equation}\label{eq:bound_of_dual}
    \normfro{\Lambda_k}\le M_\Lambda
\end{equation}
for some constant $M_\Lambda>0$.

We now estimate the KKT residual at $Y_k$. Evaluating the stationarity part of
\eqref{eq:R0-def} at $(G,\Lambda)=(G_k,\Lambda_k)$, the tangent subproblem
optimality condition gives
\[
\begin{aligned}
\nabla f(Y_k)+G_k+Y_k\Lambda_k
&=
\nabla f(Y_k)-\nabla f(X_k)-\frac1{t_k}V_k+V_k\Lambda_k .
\end{aligned}
\]
Hence, by the Lipschitz continuity of $\nabla f$, the boundedness of
$\Lambda_k$, and $t_k\le t_{\max}$, one has
\[
\normfro{\nabla f(Y_k)+G_k+Y_k\Lambda_k}
\le
C_{\rm R}\frac{\normfro{V_k}}{t_k}
\]
for  some $C_{\rm R}>0$.

It remains to control the feasibility residual. One has
\(
    \normfro{Y_k^\top Y_k-I_r}
    \le
    \normfro{\Delta}+\normfro{V_k}^2 .
\)
Combining the stationarity and feasibility estimates with
Definition~\ref{def:KKT-residual-P0} proves the first claim. Finally, if
$\normfro{V_k}\le1$, then $\normfro{V_k}^2\le \normfro{V_k}\le
t_{\max}\normfro{V_k}/t_k$, and the second estimate follows by increasing the
constant.
\end{proof}

\subsection{Proof of Lemma~\ref{lem:local-solution-branch}}
\label{app:proof-local-solution-branch}

\begin{proof}[Proof of Lemma~\ref{lem:local-solution-branch}]
Define
\(\Gamma(X,\Delta'):=X^\top X-(I_r+\Delta')\), for
\((X,\Delta')\in\M_g(X^*)\times\mathbb S^r\).
Its differential is given by
\(\rmD\Gamma(X,\Delta')[\xi,H]=X^\top\xi+\xi^\top X-H\).
Since the derivative with respect to $\Delta'$ is $-I$ on $\mathbb S^r$, the
set
\[
    \mathcal S
    :=
    \{(X,\Delta'):\ X\in\M_g(X^*),\ \Gamma(X,\Delta')=0\}
\]
is a smooth embedded submanifold around $(X^*,\Delta)$.

By the transversality assumption in Lemma~\ref{lem:local-solution-branch}, the map
\(\xi\mapsto X^{*\top}\xi+\xi^\top X^*\)
from $\T_{X^*}\M_g(X^*)$ to $\mathbb S^r$ is surjective. This property is open.
Hence, after restricting to neighborhoods $\mathcal U$ of $X^*$ and
$\mathcal V$ of $\Delta$, the projection
\[
    \pi:\mathcal S\to\mathbb S^r,\qquad \pi(X,\Delta')=\Delta',
\]
is a submersion on $\mathcal S\cap(\mathcal U\times\mathcal V)$.

According to the submersion theorem, there exist a neighborhood
$\mathcal W\subset\mathbb R^q$ of the origin and a smooth mapping
\(\phi:\mathcal W\times\mathcal V\to\mathcal U\)
such that $\phi(0,\Delta)=X^*$ and, locally,
\[
    \phi(\mathcal W,\Delta')
    =
    \overline{\M}_{\Delta'}^*\cap\mathcal U,
    \qquad \forall\,\Delta'\in\mathcal V.
\]
Define
\(h(u,\Delta'):=F(\phi(u,\Delta'))\).
Since $F|_{\M_g(X^*)}$ is $C^2$, the function $h$ is $C^2$. Furthermore,
{\color{blue} since $X^*$ is critical,  $\nabla_u h(0,\Delta)=0$, while the SOSC
implies that $\nabla^2_{uu}h(0,\Delta)$ is positive definite and hence
invertible.}

Applying the implicit function theorem to
\(
    \nabla_u h(u,\Delta')=0
\)
gives, after possibly shrinking $\mathcal W$ and $\mathcal V$, a unique
$C^1$ mapping $u:\mathcal V\to\mathcal W$ satisfying
$u(\Delta)=0$ and
\[
    \nabla_u h(u(\Delta'),\Delta')=0,
    \qquad \forall\,\Delta'\in\mathcal V.
\]
Set
\(\mathcal X(\Delta'):=\phi(u(\Delta'),\Delta')\).
Then $\mathcal X(\Delta)=X^*$ and $\mathcal X(\Delta')$ is the unique critical
point of $F$ on $\overline{\M}_{\Delta'}^*\cap\mathcal U$.

{\color{blue}By the SOSC and continuity, after shrinking $\mathcal U$ and
$\mathcal V$, the Hessian bound in
Lemma~\ref{lem:local-solution-branch} holds on every nearby manifold $\M_{\Delta'}$; hence
$\mathcal X(\Delta')$ is the unique local minimizer of $F$ on
$\overline{\M}_{\Delta'}^*\cap\mathcal U$.} Finally, the $C^1$ smoothness of
$\mathcal X$ implies its local Lipschitz continuity, and hence
\(    \|\mathcal X(\Delta')-X^*\|_{\mathrm F}
    \le
    C_\Delta\|\Delta'-\Delta\|_{\mathrm F},
 \forall\,\Delta'\in\mathcal V.
\)
\end{proof}

\subsection{Proof of Lemma~\ref{lem:alpha_one_eventually}}
\label{app:proof-alpha-one-eventually}

\begin{proof}
{\color{blue}
The proof of Theorem~\ref{thm:global_convergence_mix} provides a constant
$c_{\rm dec}>0$ such that
\[
    F(X_{k+1})
    \leq
    F(X_k)-c_{\rm dec}\|V_k\|_{\mathrm F}^2,
    \qquad
    \sum_{k=0}^{\infty}\|V_k\|_{\mathrm F}^2<\infty.
\]
Thus $V_k\to0$, while $\{F(X_k)\}$ is nonincreasing. Let
$X_{k_j}\to X^*$ be a subsequence given by
Assumption~\ref{assump:mix-seq-converge}. Continuity gives
$F(X_{k_j})\to F(X^*)$, and monotonicity then yields
$F(X_k)\to F(X^*)$.

{\color{blue}
	Set
	\[
	\Phi_\Delta(X)
	:=
	f(X)+g(X)+\delta_{\St_\Delta(n,r)}(X).
	\]
	Under the transversality condition,
	$\Phi_\Delta$ is $C^2$-partly smooth at $X^*$ relative to
	\[
	\M_\Delta^*
	=
	\M_g(X^*)\cap\St_\Delta(n,r),
	\]
	and
	\[
	\partial\Phi_\Delta(X^*)
	=
	\nabla f(X^*)+\partial g(X^*)
	+\rmN_{X^*}\St_\Delta(n,r).
	\]
	Since $\operatorname{ri} \Proj_{\T_{X^*}\St_\Delta}\partial g = \Proj_{\T_{X^*}\St_\Delta} (\operatorname{ri} \partial g) $, the nondegeneracy condition gives
	\(
	0\in\operatorname{ri}\partial\Phi_\Delta(X^*).
	\)
	Moreover, $\Phi_\Delta$ is prox-regular at $X^*$ by the sum rule.
	
	Since $\Phi_\Delta|_{\M_\Delta^*}=F|_{\M_\Delta^*}$,
	the criticality of $X^*$ and the SOSC~\eqref{eq:mix-sosc} imply,
	by Proposition~\ref{lem:QG-implies-PDHess}, that $X^*$ is a strong
	local minimizer of $\Phi_\Delta|_{\M_\Delta^*}$. Hence $X^*$ is a
	strong critical point of $\Phi_\Delta$ relative to $\M_\Delta^*$ in
	the sense of Definition~\ref{def:strong-critical}.
	
	Since $\Phi_\Delta$ is $C^2$-partly smooth and prox-regular at $X^*$,
	\cite[Theorem~6.2(ii)]{hareLewis2004} yields constants $\mu>0$ and
	$\mathfrak{r}>0$ such that
	\begin{equation}\label{eq:mix-local-growth}
		F(X)
		\geq
		F(X^*)+\mu\|X-X^*\|_{\mathrm F}^2,
		\qquad
		X\in\St_\Delta(n,r),\quad
		\|X-X^*\|_{\mathrm F}<\mathfrak{r} .
	\end{equation}
}

It remains to bound one MIX displacement near $X^*$.
After decreasing $\mathfrak{r}$ if necessary, Theorem~\ref{thm:id-strong-crit}
and Lemma~\ref{lem:local-solution-branch} apply throughout the
corresponding neighborhood. Since $V_k\to0$ and
$\Delta_k-\Delta=V_k^\top V_k\to0$,
for all sufficiently large $k$, whenever
$\|X_k-X^*\|_{\mathrm F}<\mathfrak{r}$, one has
$\M_g(Y_k)=\M_g(X^*)$ and
$\overline{\M}_k=\overline{\M}_{\Delta_k}^*$.
The fixed-phase TAPR map in
\cite[Theorem~5]{chen2026alternating} is a smooth local retraction;
after decreasing the neighborhood if necessary, there is a constant
$C_{\rm R}>0$ such that
$\|\Retr^{\rm TAPR}_{Y}(w)-Y\|_{\mathrm F}
\leq C_{\rm R}\|w\|_{\mathrm F}$ on the nearby manifolds considered below.

On the ManPG branch,
Lemma~\ref{lem:orthogonalization_increase_general_Delta} gives
$\|X_{k+1}-X_k\|_{\mathrm F}
\leq C_{\rm M}\|V_k\|_{\mathrm F}$ for some $C_{\rm M}>0$.

On an accepted Newton-CG branch, set
$A_k:=\Hess\bar F_k(Y_k)$, $H_k:=A_k+\varrho_k I$, and
$g_k:=\grad\bar F_k(Y_k)$. Lemma~\ref{lem:local-solution-branch}
ensures $A_k\succeq\mathfrak m I$. Residual orthogonality for the
returned CG iterate yields
$\langle g_k,w_k\rangle=-\langle H_k w_k,w_k\rangle$. Consequently,
\[
    \mathfrak m\|w_k\|_{\mathrm F}^2
    \leq
    -\langle g_k,w_k\rangle
    \leq
    \|g_k\|_{\mathrm F}\|w_k\|_{\mathrm F},
\]
which, combined with Lemma~\ref{lem:grad_bound_by_manpg} and
$t_k\geq t_{\min}$, gives
$\|w_k\|_{\mathrm F}\leq C_w\|V_k\|_{\mathrm F}$.
The preceding retraction estimate and
Lemma~\ref{lem:orthogonalization_increase_general_Delta} then imply
$\|X_{k+1}-X_k\|_{\mathrm F}
\leq C_{\rm N}\|V_k\|_{\mathrm F}$.
Every unsuccessful Newton-CG trial reverts to the ManPG step.
Thus, for some $C>0$ and all sufficiently large $k$,
\begin{equation}\label{eq:mix-step-bound}
    \|X_{k+1}-X_k\|_{\mathrm F}
    \leq C\|V_k\|_{\mathrm F}
\end{equation}
whenever $\|X_k-X^*\|_{\mathrm F}<\mathfrak{r}$.

Choose $K$ so that $C\|V_k\|_{\mathrm F}<\mathfrak{r}/4$ for all $k\geq K$,
and then use $X_{k_j}\to X^*$ and $F(X_k)\to F(X^*)$ to select
$k_0\geq K$ such that
\[
    \|X_{k_0}-X^*\|_{\mathrm F}<\frac {\mathfrak{r}}{4},
    \qquad
    F(X_{k_0})-F(X^*)<\frac{\mu \mathfrak{r}^2}{4}.
\]
We claim that the sequence remains in the smaller neighborhood:
\[
    \|X_k-X^*\|_{\mathrm F}<\frac {\mathfrak{r}}{2},
    \qquad k\geq k_0.
\]
Suppose otherwise, and let $j>k_0$ be the smallest index for which
$\|X_j-X^*\|_{\mathrm F}\geq \mathfrak{r}/2$.
Then
$\|X_{j-1}-X^*\|_{\mathrm F}<\mathfrak{r}/2$, and
\eqref{eq:mix-step-bound} implies
$\|X_j-X^*\|_{\mathrm F}<\mathfrak{r}$. Equation~\eqref{eq:mix-local-growth}
and monotonicity would then give
\[
    \frac{\mu \mathfrak{r}^2}{4}
    \leq F(X_j)-F(X^*)
    \leq F(X_{k_0})-F(X^*)
    <\frac{\mu \mathfrak{r}^2}{4},
\]
which is impossible. The claim therefore holds. Combining
\eqref{eq:mix-local-growth} with $F(X_k)\to F(X^*)$ now yields
\[
    X_k\to X^*.
\]

Having proved convergence of the whole sequence, we may apply
Theorem~\ref{thm:id-strong-crit}, since
$t_k\in[t_{\min},t_{\max}]$. It furnishes an index $K_{\rm id}$ such that
\[
    \M_g(Y_k)=\M_g(X^*),
    \qquad
    \overline{\M}_k=\overline{\M}_{\Delta_k}^*,
    \qquad
    k\geq K_{\rm id}.
\]
In particular, $Y_k\to X^*$ and $\Delta_k\to\Delta$.

It remains to consider the Newton-CG step. For all sufficiently large
$k$, Lemma~\ref{lem:local-solution-branch} yields
$A_k\succeq\mathfrak m I$. Since
\[
H_k\succeq(\mathfrak m+\varrho_k)I
\]
and $\vartheta_k<\mathfrak m+\varrho_k$, the curvature test in
Algorithm~\ref{alg:tcg} cannot be triggered for all sufficiently large $k$.  
In addition, Lemma~\ref{lem:grad_bound_by_manpg} and $V_k\to0$ imply
$g_k\to0$. Because $\varpi>0$, the superlinear tCG stopping criterion
is active for all sufficiently large $k$. Positive definiteness of
$H_k$ ensures that CG reaches this criterion in finitely many iterations;
thus $\mathrm{cg\_status}=\texttt{superlinear}$ eventually. The estimate
above also gives $w_k\to0$.

Since $Y_k\to X^*$ and $\Delta_k\to\Delta$, the common local
parameterization in Lemma~\ref{lem:local-solution-branch}, together with
the smooth dependence of the fixed-phase TAPR construction, allows the
second-order pullback expansion to be taken uniformly for all sufficiently
large $k$.
For the full step, the standard pullback expansion for a second-order
retraction~\cite[Proposition~5.44]{boumal2023introduction} gives
\[
\begin{aligned}
    \bar F_k\bigl(\Retr^{\rm TAPR}_{Y_k}(w_k)\bigr)
    &=
    \bar F_k(Y_k)
    +\langle g_k,w_k\rangle
    +\frac12\langle A_kw_k,w_k\rangle
    +o(\|w_k\|_{\mathrm F}^2).
\end{aligned}
\]
Substituting
$\langle g_k,w_k\rangle
=-\langle (A_k+\varrho_kI)w_k,w_k\rangle$ into this expansion gives
\[
\begin{aligned}
&
    \bar F_k\bigl(\Retr^{\rm TAPR}_{Y_k}(w_k)\bigr)
    -\bar F_k(Y_k)
    -c_{\rm N}\langle g_k,w_k\rangle
\\
&\leq
    -\left(\frac12-c_{\rm N}\right)\mathfrak m\|w_k\|_{\mathrm F}^2
    +o(\|w_k\|_{\mathrm F}^2).
\end{aligned}
\]
Since $c_{\rm N}<1/2$, the Newton Armijo condition is satisfied with
$\alpha=1$ for all sufficiently large $k$.

Only the projected sufficient-decrease test can then reject the unit
Newton step. Once $t_k\leq\bar t_{\rm N}$,
Lemma~\ref{lem:suff_decrease} guarantees acceptance. Otherwise, each
failure reduces $t_k$ by the factor $\beta_t^\downarrow$, whereas an
accepted full superlinear step leaves $t_k$ unchanged. Thus only
finitely many failures occur. Consequently,
Algorithm~\ref{alg:newton-cg} eventually accepts $\alpha=1$ at every
iteration, $t_k$ remains constant, and the update
rule~\eqref{eq:mix-update-rule} reduces to
$\lambda_{k+1}=\beta_\lambda^\downarrow\lambda_k$. Therefore
$\lambda_k\to0$.
}
\end{proof}

\begin{lemma}\label{lem:orthogonalization_increase_general_Delta}
Let $\Delta\in\mathbb{S}^r$ and $D=I_r+\Delta\succ0$. 
Let $m:=\lambda_{\min}(D)$ and $M:=\lambda_{\max}(D)$.
For any $X\in\mathbb{R}^{n\times r}$ satisfying
\[
\normfro{X^\top X-D} \le \frac{m}{4},
\]
it holds that
\begin{equation}\label{eq:orthogonalization_increase_general_Delta}
\normfro{ \Proj_{\St_\Delta}(X)-X}\leq 
c_\Delta \normfro{X^\top X-D},
\end{equation}
where  
\[
c_\Delta \ :=\ \frac{2\sqrt5}{3\sqrt3}\cdot \frac{\sqrt{M}}{m}.
\]
In particular, when $\Delta=0$, \eqref{eq:orthogonalization_increase_general_Delta} holds with
$c_0=\frac{2\sqrt5}{3\sqrt3}<1$.
\end{lemma}

\begin{proof}
Define
\(\bar X := X D^{-1/2}\) and
\(G:=\bar X^\top \bar X = D^{-1/2}X^\top X D^{-1/2}\).
Then, we have 
\(G-I_r = D^{-1/2}(X^\top X-D)D^{-1/2}\).
Using $\|D^{-1/2}\|_2=1/\sqrt{m}$ and the assumption $\|X^\top X-D\|_F\le m/4$, we have
\begin{equation}\label{eq:G_close_I}
\|G-I_r\|_F \le \|D^{-1/2}\|_2^2\,\|X^\top X-D\|_F \le \frac{1}{m}\cdot \frac{m}{4}=\frac14.
\end{equation}
Hence $G\succ0$ and its eigenvalues lie in $[3/4,\,5/4]$.

Now set
\[
    \bar Y := \bar X\,G^{-1/2},\qquad
    Y := \bar Y\,D^{1/2}.
\]
Then $\bar Y^\top \bar Y = I_r$ and
$Y^\top Y = D^{1/2}\bar Y^\top \bar Y D^{1/2}=D$, hence $Y\in\St_\Delta$.
Since $\widehat X =\Proj_{\St_\Delta}(X)$, it follows that 
\begin{equation}\label{eq:proj_is_best}
\|\widehat X-X\|_F \le \|Y-X\|_F.
\end{equation}

Note that
\(Y-X = \bar Y D^{1/2}-\bar X D^{1/2} = (\bar Y-\bar X)D^{1/2}\),
so
\begin{equation*}
\|Y-X\|_F \le \|D^{1/2}\|_2\,\|\bar Y-\bar X\|_F = \sqrt{M}\,\|\bar Y-\bar X\|_F.
\end{equation*}
Moreover,
\[
\bar Y-\bar X = \bar X(G^{-1/2}-I_r),
\quad\Rightarrow\quad
\|\bar Y-\bar X\|_F \le \|\bar X\|_2\,\|G^{-1/2}-I_r\|_F.
\]
Since $\|\bar X\|_2^2=\lambda_{\max}(G)\le 5/4$, we have
\(\|\bar X\|_2 \le \sqrt{5/4}\).
By now, we obtain
\begin{equation}\label{eq:barX_norm}
    \normfro{\hat{X}-X}\leq \sqrt{\frac{5M}{4}} \normfro{G^{-1/2}-I_r}.
\end{equation}
 
Since $G-I_r=D^{-1/2}(X^\top X-D)D^{-1/2}$, we have 
\begin{equation}\label{eq:Y_minus_X_1}
\|G-I_r\|_F \le \|D^{-1/2}\|_2^2\,\|X^\top X-D\|_F = \frac{1}{m}\|X^\top X-D\|_F.
\end{equation}

It remains to bound $\|G^{-1/2}-I_r\|_F$ in terms of $\|G-I_r\|_F$.
Let $E:=G-I_r$. For $t\in[0,1]$, $I_r+tE$ has eigenvalues in $[3/4,\,5/4]$, hence
\[
\|(I_r+tE)^{-3/2}\|_2 \le (3/4)^{-3/2}=\left(\frac{4}{3}\right)^{3/2}=\frac{8}{3\sqrt3}.
\]
Using the integral identity
\[
G^{-1/2}-I_r = (I_r+E)^{-1/2}-I_r
=\int_0^1 \frac{d}{dt}(I_r+tE)^{-1/2}\,dt
=-\frac12\int_0^1 (I_r+tE)^{-3/2}E\,dt,
\]
we obtain
\begin{equation}\label{eq:G_inv_sqrt_bound}
\|G^{-1/2}-I_r\|_F
\le \frac12\sup_{t\in[0,1]}\|(I_r+tE)^{-3/2}\|_2\,\|E\|_F
\le \frac12\cdot\frac{8}{3\sqrt3}\,\|G-I_r\|_F
=\frac{4}{3\sqrt3}\,\|G-I_r\|_F.
\end{equation}
Combining \eqref{eq:barX_norm} and \eqref{eq:G_inv_sqrt_bound} yields
\begin{equation}\label{eq:barY_minus_barX_bound}
\|\hat X - X\|_F
\le \sqrt{\frac{5M}{4}}\cdot \frac{4}{3\sqrt3}\,\|G-I_r\|_F
= \frac{2\sqrt{5M}}{3\sqrt3}\,\|G-I_r\|_F.
\end{equation}
Plugging \eqref{eq:barY_minus_barX_bound} into \eqref{eq:Y_minus_X_1} gives
\begin{equation*}\label{eq:proj_dist_bound_final}
\|\widehat X-X\|_F
\le
\frac{2\sqrt5}{3\sqrt3}\cdot \frac{\sqrt{M}}{m}\,\|X^\top X-D\|_F.
\end{equation*}
\end{proof}

\subsection{tCG algorithm and Newton-CG properties}

\begin{algorithm}[ht]
\caption{tCG interface for the regularized Newton equation}
\label{alg:tcg}
\begin{algorithmic}[1]
\Require $\overline{\M}_k$, $Y\in\overline{\M}_k$, smooth objective $\bar F_k$,
regularization scale $\lambda_k>0$, exponent $\sigma\in[0,1]$,
and tCG parameters $\varpi\ge0$, $\theta\in(0,1]$, $\vartheta_k>0$.
\Ensure A vector $w\in\T_Y\overline{\M}_k$ and a status flag.

\State Set $d_k=\dim(\overline{\M}_k)$, $g=\grad \bar F_k(Y)$, and
$\varrho_k=\lambda_k\normfro{g}^{\sigma}$.
\State Define
\[
    H_k[w]=\Hess \bar F_k(Y)[w]+\varrho_k w,
    \qquad b=-g.
\]
\If{$g=0$}
    \State \Return $w=0$ and $\mathrm{status}=\texttt{superlinear}$.
\EndIf
\State Set $w_0=0$, $r_0=b$, and $p_0=r_0$.
\For{$i=0,1,\ldots,d_k-1$}
    \State Compute $H_kp_i$.
    \If{$\langle p_i,H_kp_i\rangle\le\vartheta_k\normfro{p_i}^2$}
        \State \Return $w=\mathrm{null}$ and
        $\mathrm{status}=\texttt{negative-curvature}$.
    \EndIf

    \State $\alpha_i\gets
    \dfrac{\normfro{r_i}^2}{\langle p_i,H_kp_i\rangle}$.
    \State $w_{i+1}\gets w_i+\alpha_i p_i$.
    \State $r_{i+1}\gets r_i-\alpha_i H_kp_i$.

\If{$\normfro{r_{i+1}} \le \normfro{r_0}\min\{\normfro{r_0}^{\theta},\varpi\}$} \If{$\normfro{r_0}^{\theta}\le\varpi$} \State \Return $w=w_{i+1}$ and $\mathrm{status}=\texttt{superlinear}$. \Else \State \Return $w=w_{i+1}$ and $\mathrm{status}=\texttt{linear}$. \EndIf \EndIf

    \State $\beta_i\gets
    \dfrac{\normfro{r_{i+1}}^2}{\normfro{r_i}^2}$.
    \State $p_{i+1}\gets r_{i+1}+\beta_i p_i$.
\EndFor
\State \Return $w=w_{d_k}$ and $\mathrm{status}=\texttt{maxit}$.
\end{algorithmic}
\end{algorithm}

\begin{lemma}[Local Lipschitz estimate for the ManPG step]\label{lem:Vk_dist_bound}
Suppose that the assumptions of Theorem~\ref{thm:id-strong-crit} hold on a neighborhood $\mathcal U$ of $X^*$ and on the compact stepsize interval $\mathcal I=[t_{\min},t_{\max}]$. Then, after possibly shrinking $\mathcal U$, there exists a constant $C_V>0$ such that
\[
    \|\mathcal T_t(X)-X\|_{\mathrm F}
    \le C_V\|X-X^*\|_{\mathrm F},
    \qquad
    \forall\,X\in\mathcal U,\quad \forall\,t\in\mathcal I.
\]
In particular, if $X_k\in\mathcal U$ and $t_k\in\mathcal I$, then
\[
    \|V_k\|_{\mathrm F}
    \le C_V\|X_k-X^*\|_{\mathrm F}.
\]
\end{lemma}
\begin{proof}
By Theorem~\ref{thm:id-strong-crit} and Lemma~\ref{lem:Ysharp-smooth},
$\mathcal T_t(X)=Y_t^\sharp(X)$ is $C^1$ on $\mathcal U\times\mathcal I$ after
shrinking $\mathcal U$ if necessary. Moreover, $\mathcal T_t(X^*)=X^*$ for all
$t\in\mathcal I$. Define $V_t(X):=\mathcal T_t(X)-X$. Then $V_t(X^*)=0$ for all
$t\in\mathcal I$, and $V_t$ is $C^1$ in $(X,t)$. Since $\mathcal I$ is compact,
the derivative of $V_t$ with respect to $X$ is uniformly bounded on a smaller
neighborhood $\mathcal U\times\mathcal I$. The mean-value theorem gives the
desired bound.
\end{proof}

\begin{lemma}[Control of the relative gradient by the ManPG residual]
\label{lem:grad_bound_by_manpg}
Suppose that Assumption~\ref{ass:working_set} holds and that
$t_k\le t_{\max}$. Let $Y_k=X_k+V_k$ be generated by the tangent subproblem
\eqref{eq:tangent_problem}. Whenever the local manifold
$\overline{\M}_k$ is well defined and $\bar F_k:=F|_{\overline{\M}_k}$ is
smooth, there exists a constant $C_{\rm grad}>0$, independent of $k$, such that
\[
    \normfro{\grad \bar F_k(Y_k)}
    \le
    C_{\rm grad}\frac{\normfro{V_k}}{t_k}.
\]
\end{lemma}

\begin{proof}
Let $D=I_r+\Delta\succ0$. By the optimality condition of
\eqref{eq:tangent_problem}, there exist
$G_k\in\partial g(Y_k)$ and $\Lambda_k\in\mathbb S^r$ such that
\begin{equation}\label{eq:manpg-opt-appendix}
    \nabla f(X_k)+G_k+\frac{1}{t_k}V_k+X_k\Lambda_k=0 .
\end{equation}
It follows from \eqref{eq:bound_of_dual} that
\(\normfro{\Lambda_k}\le M_\Lambda\)
for some constant $M_\Lambda>0$ independent of $k$.

Let
\(P_k:=\Proj_{\T_{Y_k}\overline{\M}_k}\).
Since $\overline{\M}_k\subseteq\M_g(Y_k)$ and $g$ is smooth relative to
$\M_g(Y_k)$, for any $\xi\in\T_{Y_k}\overline{\M}_k$ and any
$G_k\in\partial g(Y_k)$, one has
\[
    \mathrm D\bigl(g|_{\overline{\M}_k}\bigr)(Y_k)[\xi]
    =
    \langle G_k,\xi\rangle .
\]
Therefore,
\(\mathrm D\bar F_k(Y_k)[\xi]
=\langle \nabla f(Y_k)+G_k,\xi\rangle\) for all
\(\xi\in\T_{Y_k}\overline{\M}_k\),
and hence
\(\grad\bar F_k(Y_k)=P_k(\nabla f(Y_k)+G_k)\).
Since $Y_k\in\St_{\Delta_k}$ and $\Lambda_k\in\mathbb S^r$,
$Y_k\Lambda_k$ is normal to $\St_{\Delta_k}$ at $Y_k$, and hence
$P_k(Y_k\Lambda_k)=0$. It follows from
\eqref{eq:manpg-opt-appendix} that
\[
\begin{aligned}
    \grad\bar F_k(Y_k)
    &=
    P_k\bigl(\nabla f(Y_k)+G_k+Y_k\Lambda_k\bigr) \\
    &=
    P_k\left(
        \nabla f(Y_k)-\nabla f(X_k)
        -\frac{1}{t_k}V_k
        +V_k\Lambda_k
    \right).
\end{aligned}
\]
Using the $L_f$-Lipschitz continuity of $\nabla f$ and the bound on
$\Lambda_k$, we have
\[
\begin{aligned}
    \normfro{\grad\bar F_k(Y_k)}
    &\le
    \left(
        L_f+M_\Lambda+\frac{1}{t_k}
    \right)\normfro{V_k} \\
    &\le
    \left(
        1+t_{\max}(L_f+M_\Lambda)
    \right)
    \frac{\normfro{V_k}}{t_k}.
\end{aligned}
\]
Thus the claim holds with
$C_{\rm grad}:=1+t_{\max}(L_f+M_\Lambda)$.
\end{proof}

The following local estimate follows the regularized Riemannian Newton analysis
in \cite[Theorem~16]{hu2018adaptive}. We provide its proof because, in MIX, the
local intersection manifold varies with $\Delta_k$ and the accepted Newton point
is subsequently projected back to $\St_\Delta$.
\begin{lemma}\label{lem:local_regularized_newton_estimate}
Under the setting of Lemma~\ref{lem:alpha_one_eventually}, let
$X_k^*:=\mathcal X(\Delta_k)$ be given by
Lemma~\ref{lem:local-solution-branch}. Then there exists some constant 
$K<\infty$ such that, for all $k\ge K$,
\[
    \normfro{Z_k-X_k^*}
   = o\!\left(\|Y_k-X_k^*\|_{\mathrm F}\right)
   +
    O\left(
        \varrho_k\normfro{Y_k-X_k^*}
        +
        \normfro{Y_k-X_k^*}^{1+\theta}
    \right).
\]
\end{lemma}
\begin{proof}
By Lemma~\ref{lem:alpha_one_eventually}, for all sufficiently large $k$,
Algorithm~\ref{alg:tcg} returns
$\mathrm{status}=\texttt{superlinear}$ and
Algorithm~\ref{alg:newton-cg} accepts the full trial step. Hence
\(Z_k=\Retr_{Y_k}^{\rm TAPR}(w_k)\).
Set
\(g_k:=\grad\bar F_k(Y_k)\) and
\(B_k:=\Hess\bar F_k(Y_k)+\varrho_k I\),
and define the residual of the regularized Newton equation by
\begin{equation}\label{eq:qk-superlinear}
    q_k:=B_k[w_k]+g_k .
\end{equation}
Since tCG returns the superlinear status, one has
\begin{equation}\label{eq:qk-forcing}
    \normfro{q_k}
    \le
    \normfro{g_k}^{1+\theta}.
\end{equation}

{\color{blue}For all sufficiently large $k$, both $Y_k$ and $X_k^*$ belong to
$\overline{\M}_k=\overline{\M}_{\Delta_k}^*$ and converge to $X^*$. Let
\[
    s_k:=\Exp_{Y_k}^{-1}(X_k^*)\in \T_{Y_k}\overline{\M}_k .
\]
After restricting to a sufficiently small neighborhood, the exponential
coordinates may be chosen uniformly for $\Delta_k$ near $\Delta$, and
there exists $C_s>0$ such that
\begin{equation}\label{eq:sk-distance}
    \normfro{s_k}
    \le
    C_s\normfro{Y_k-X_k^*}
\end{equation}
for all sufficiently large $k$.}

{\color{blue}Define
\[
    \widehat F_k(s)
    :=
    \bar F_k\bigl(\Exp_{Y_k}(s)\bigr).
\]
Since $X_k^*$ is a critical point of $\bar F_k$,
$\nabla\widehat F_k(s_k)=0$. Moreover,
\[
    \nabla\widehat F_k(0)=g_k,
    \qquad
    \nabla^2\widehat F_k(0)=\Hess\bar F_k(Y_k).
\]
Since $F|_{\M_g(X^*)}$ is $C^2$ and the local manifolds
$\M_{\Delta'}^*$ depend smoothly on $\Delta'$, the corresponding
exponential pullbacks have locally continuous Hessians.  Hence,
uniformly for all sufficiently large $k$,
\begin{equation}\label{eq:sk-newton-expansion}
    g_k+\Hess\bar F_k(Y_k)[s_k]
    =
    o(\normfro{s_k}).
\end{equation}
}

Combining \eqref{eq:qk-superlinear} and
\eqref{eq:sk-newton-expansion}, one has
\[
\begin{aligned}
    q_k
    &=
    B_k[w_k]+g_k \\
    &=
    B_k[w_k-s_k]
    +
    \varrho_k s_k
    +
    o(\normfro{s_k}).
\end{aligned}
\]
Equivalently,
\(B_k[w_k-s_k]=q_k-\varrho_k s_k+o(\normfro{s_k})\).
{\color{blue}By the Hessian bound in
Lemma~\ref{lem:local-solution-branch} and $\varrho_k\ge0$, one has
$B_k\succeq\mathfrak m I$, so the inverses of $B_k$ are uniformly bounded.}
It follows that
\begin{equation}\label{eq:wk-sk-bound}
    \normfro{w_k-s_k}
    \le
    C_t\left(
        \normfro{q_k}
        +
        \varrho_k\normfro{s_k}
        +
      o(  \normfro{s_k})
    \right),
\end{equation}
for some $C_t>0.$
{\color{blue}Since the fixed-phase TAPR map is a $C^2$ retraction,
\[
    \normfro{\Retr_{Y_k}^{\rm TAPR}(u)-\Retr_{Y_k}^{\rm TAPR}(v)}
    \leq C_{\rm R}\normfro{u-v}
\]
locally, with $C_{\rm R}$ independent of sufficiently large $k$. Moreover,
\[
    \normfro{\Retr_{Y_k}^{\rm TAPR}(s_k)-\Exp_{Y_k}(s_k)}
    =
    O(\normfro{s_k}^2).
\]
Therefore,
\[
\begin{aligned}
    \normfro{Z_k-X_k^*}
    &=
    \normfro{
        \Retr_{Y_k}^{\rm TAPR}(w_k)
        -
        \Exp_{Y_k}(s_k)
    } \\
    &\le
    C\left(
        \normfro{q_k}
        +
        \varrho_k\normfro{s_k}
        +
       o( \normfro{s_k})
    \right),
\end{aligned}
\]
for some $C>0$.}

Moreover, \eqref{eq:sk-newton-expansion} and the local boundedness of
$\Hess\bar F_k(Y_k)$ give
\[
    \normfro{g_k}
    =
    O(\normfro{s_k})
    =
    O(\normfro{Y_k-X_k^*}).
\]
Together with \eqref{eq:qk-forcing} and \eqref{eq:sk-distance}, this yields
\[
    \normfro{Z_k-X_k^*}
    \le
    o(  \normfro{Y_k-X_k^*})+
    C\left(
        \varrho_k\normfro{Y_k-X_k^*}
        +
        \normfro{Y_k-X_k^*}^{1+\theta}
    \right).
\]
\end{proof}

\input{shared/appendix_extension}

%% file: shared/appendix_preliminaries.tex
\section{\texorpdfstring{Subdifferential continuity in partial smoothness}{Subdifferential continuity in partial smoothness}}
\label{app:preliminaries}
\label{app:prelim-subcontinuity}

\begin{remark}\label{rem:subcont-meaning}
	In Definition~\ref{def:partial_smoothness}(4), the \emph{sub-continuity} refers to the
		continuity of the subdifferential mapping along $\M_g$. More precisely, define the
	set-valued mapping $G:\M_g \rightrightarrows \E$ by
	\(G(x):=\partial g(x)\), \(x\in\M_g\).
	We say that $G$ is continuous at $\bar x$ (relative to $\M_g$) if it is both outer and inner
	semicontinuous:

	\smallskip
	\noindent\emph{(outer semicontinuity)} for any sequence $x_r\in \M_g$ with $x_r\to \bar x$,
	and any $y_r\in G(x_r)$ with $y_r\to \bar y$, it holds that $\bar y\in G(\bar x)$; equivalently,
	\(\limsup_{r\to\infty} G(x_r)\subseteq G(\bar x)\).

	\smallskip
	\noindent\emph{(inner semicontinuity)} for any sequence $x_r\in \M_g$ with $x_r\to \bar x$,
	and any $\bar y\in G(\bar x)$, there exist $y_r\in G(x_r)$ such that $y_r\to \bar y$; equivalently,
	\(G(\bar x)\subseteq\liminf_{r\to\infty}G(x_r)\).
\end{remark}

%% file: shared/appendix_identification.tex
\section{\texorpdfstring{Proofs for the identification analysis in Section~\ref{sec:id_manpg}}{Proofs for the identification analysis in Section 3}}
\label{app:section3-identification-proofs}

\subsection{\texorpdfstring{Proof of Proposition~\ref{lem:generic-transv}}{Proof of Proposition 3.1}}
\label{app:proof-generic-transv}

\begin{proof}[Proof of Proposition~\ref{lem:generic-transv}]
Let $m:=r(r-1)/2$ and fix a linear isomorphism
$\overline\vet_{\rm off}:\mathbb S_0^r\to\R^m$.
Write \(h_{\rm diag}(X):=\diag(X^\top X-I_r)\) and
\(h_{\rm off}(X):=\overline\vet_{\rm off}(\mathrm{off}(X^\top X))\), so that
for every $\Delta\in\mathcal P$,
\[
\St_\Delta(n,r)=\{X:\ h_{\rm diag}(X)=0,\ h_{\rm off}(X)=d\},
\qquad d:=\overline\vet_{\rm off}(\Delta).
\]
Let $D:=\overline\vet_{\rm off}(\mathcal P)\subset\R^m$, which is open since $\mathcal P$ is open in $\mathbb S_0^r$.

For each $i$, let $h_i:\R^{n\times r}\to\R^{p_i}$ be the defining map of $\mathcal M_i$ from Assumption~\ref{assump:finite-active-manifolds}, and set
\(g_i(X):=(h_i(X),\,h_{\rm diag}(X))\).
By \eqref{eq:oblique-transv-condition}, we first get that $g_i$ satisfies LICQ at every point of $\mathcal M_i\cap\Ob(n,r)$.
Since $\rank \rmD h_i(X)=p_i$, it follows that
\(\rank \rmD g_i(X)=p_i+r\) for all
\(X\in\mathcal M_i\cap\Ob(n,r)\).

We now apply \cite[Theorem~4]{tang2024feasible} to the pair $(g_i,h_{\rm off})$. It yields a set $D_i\subset D$ of Lebesgue measure zero such that for every $d\in D\setminus D_i$, every solution of
$g_i(X)=0$ and $h_{\rm off}(X)=d$ satisfies LICQ for the joint system $(g_i,h_{\rm off})$.

Now define $D_e:=\bigcup_{i=1}^N D_i$. Since $N<\infty$ and each $D_i$ has measure zero, so does $D_e$. Set
\[
\mathcal P_{\rm gen}=\{\Delta\in\mathcal P:\ \overline\vet_{\rm off}(\Delta)\notin D_e\}.
\]
Then $\mathcal P\setminus\mathcal P_{\rm gen}$ has Lebesgue measure zero in $\mathbb S_0^r$, and \eqref{eq:transv-Mg-StDelta} holds for every $\Delta\in\mathcal P_{\rm gen}$ and every $i$.

The final submanifold statement follows from the
transverse-intersection theorem whenever the intersection is nonempty. Since
\(\mathrm{codim}\,\St_\Delta(n,r)= r(r+1)/2\), the dimension of
\(\M_i\cap \St_\Delta(n,r)\) is
\(\dim(\mathcal M_i)-r(r+1)/2\); nonemptiness therefore forces this number to
	be nonnegative.
\end{proof}

\subsection{\texorpdfstring{Proof of Lemma~\ref{lem:incidence-Tcap}}{Proof of Lemma 3.3}}
\label{app:proof-incidence-Tcap}

\begin{proof}[Proof of Lemma~\ref{lem:incidence-Tcap}]
Let $p := \mathrm{codim} \M_g(Y^*)$ and let $m:=\mathrm{codim}\St_\Delta(n,r)=\dim \rmN_{X^*}\St_\Delta(n,r)$.
Let   $h_g$ be the local defining function
\(\M_g(Y^*)=\{Y \mid h_g(Y)=0\}\), with
\(\mathrm{D}h_g(Y^*)\) surjective.

Since $\St_\Delta(n,r)$ is a smooth embedded submanifold, there exists a neighborhood
$\mathcal U_0\subset \St_\Delta(n,r)$ of $X^*$ and a smooth map
$Q_N:\mathcal U_0\to \mathbb R^{(nr)\times m}$ whose columns form an orthonormal basis
of $\rmN_X\St_\Delta(n,r)$ for every $X\in\mathcal U_0$.
Define the normal-coordinate operator
\[
\mathcal P_N(X)[Z]:=Q_N(X)^\top \vet(Z)\in\mathbb R^m,\qquad Z\in\mathbb R^{n\times r}.
\]
Then,	define the smooth map
\[
\Phi:\ \mathcal U_0\times \mathcal V\ \to\ \mathbb R^{p}\times \mathbb R^{m},
\qquad
\Phi(X,Y):=\begin{pmatrix}
	h_g(Y)\\
	\mathcal P_N(X)[\,Y-X\,]
\end{pmatrix}.
\]
By construction, for $(X,Y)\in \mathcal U_0\times \mathcal V$,
\(\Phi(X,Y)=0\) if and only if \(Y\in\T_\cap(X)\).
The derivative of $\Phi$ with respect to $Y$ is
\[
\mathrm{D}_Y\Phi(X,Y)[H]=\begin{pmatrix}
	\mathrm{D} h_g(Y)[H]\\
	\mathcal P_N(X)[H]
\end{pmatrix},\qquad H\in\mathbb R^{n\times r}.
\]
Since $\mathrm{range}(\mathrm{D} h_g(Y^*)^\top) = \rmN_Y \M_g$ and $\mathrm{range}( \Proj_N(X^*)^*) = \rmN_{X^*}\St_{\Delta},$
transversality of $\M_g(Y^*)$ and $X^*+\T_{X^*}\St_\Delta(n,r)$ at $Y^*$ means
\begin{equation}\label{eq:transversality_X_Y}
	\rmN_{Y^*}\M_g(Y^*)\ \cap\ \rmN_{X^*}\St_\Delta(n,r)=\{0\},
\end{equation}
which is equivalent to surjectivity of $\mathrm{D}_Y\Phi(X^*,Y^*)$ onto
$\mathbb R^{p}\times\mathbb R^{m}$. By continuity, shrinking $\mathcal U_0$ and choosing
a neighborhood $\mathcal V_0\subset \mathcal V$ of $Y^*$ if necessary, we get that
$\mathrm{D}_Y\Phi(X,Y)$ is surjective for all $(X,Y)\in \mathcal U_0\times \mathcal V_0$.
Hence the full Jacobian
\(\mathrm{D}\Phi(X,Y)=\bigl[\,\mathrm{D}_X\Phi(X,Y)\ \ \ \mathrm{D}_Y\Phi(X,Y)\,\bigr]\)
is surjective on $\mathcal U_0\times \mathcal V_0$.

Therefore, by Proposition~\ref{prop:submanifold}, the zero set
\(\{(X,Y)\in \mathcal U_0\times \mathcal V_0:\ \Phi(X,Y)=0\}\)
is a smooth embedded submanifold of $\St_\Delta(n,r)\times \mathbb R^{n\times r}$,
and it coincides with $\mathcal I_0$.

Finally, for any fixed $X\in\mathcal U_0$, the restricted map $Y\mapsto \Phi(X,Y)$
has surjective derivative $\rmD_Y\Phi(X,Y)$ on $\mathcal V_0$, so its zero set
$\{Y\in\mathcal V_0:\Phi(X,Y)=0\}=\T_\cap(X)\cap\mathcal V_0$ is a smooth embedded
submanifold of $\mathbb R^{n\times r}$.
\end{proof}

\subsection{\texorpdfstring{Proof of Proposition~\ref{lem:QG-implies-PDHess}}{Proof of Proposition 3.4}}
\label{app:proof-QG-implies-PDHess}

\begin{proof}[Proof of Proposition~\ref{lem:QG-implies-PDHess}]
Since $h$ is $C^2$ and $\bar y$ is a strict local minimizer of $h|_{\M}$, one has
$\grad_{\M} h(\bar y)=0$. Fix any $\eta\in \T_{\bar y}\M$ and consider the geodesic
\(\gamma(s)=\Exp_{\bar y}(s\eta)\).
Because $\M$ is $C^2$, the exponential map is well-defined and smooth on a neighborhood of
$(\bar y,0)$; hence there exists $\varepsilon>0$ such that $\gamma(s)\in U$ for all
$|s|<\varepsilon$. Moreover, it follows from $\gamma(0)=\bar y,\qquad \dot\gamma(0)=\eta$ that
\(\gamma(s)-\bar y=s\eta+O(s^2)\).
In particular, after possibly shrinking $\varepsilon$, we have
\[
\|\gamma(s)-\bar y\|\ge \frac12 |s|\,\|\eta\|,
\qquad \forall\,|s|<\varepsilon.
\]
Applying \eqref{eq:QG-on-M} to $x=\gamma(s)$ yields
\[
h(\gamma(s))
\ge h(\bar y)+\alpha \|\gamma(s)-\bar y\|^2
\ge h(\bar y)+\frac{\alpha}{4}s^2\|\eta\|^2,
\qquad |s|<\varepsilon.
\]
On the other hand, since $h$ is $C^2$ and $\gamma$ is a geodesic with $\dot\gamma(0)=\eta$,
the second-order expansion gives
\[
h(\gamma(s))=h(\bar y)+s\langle \grad_{\M} h(\bar y),\eta\rangle
+\frac{s^2}{2}\langle \eta,\Hess_{\M} h(\bar y)[\eta]\rangle+o(s^2).
\]
Using $\grad_{\M} h(\bar y)=0$ and dividing by $s^2>0$ then letting $s\to 0$ yields
\[
\langle \eta,\Hess_{\M} h(\bar y)[\eta]\rangle\ge \frac{\alpha}{2}\|\eta\|^2.
\]
Taking $\beta=\alpha/2$ completes the first part of the proof.

Conversely, suppose that $\grad h(\bar y)=0$ and that the Riemannian
Hessian is positive definite at $\bar y$. Let $\phi$ be a local
coordinate chart with $\phi(\bar y)=0$, and set
$\widehat h=h\circ\phi^{-1}$. Since $\bar y$ is critical, the Hessian
of $\widehat h$ at the origin represents the Riemannian Hessian of $h$
at $\bar y$ and is therefore positive definite. Taylor's theorem gives,
after restricting the neighborhood if necessary,
\[
\widehat h(u)
\geq
\widehat h(0)+c_1\|u\|^2
\]
for some $c_1>0$. Since $\phi^{-1}$ is locally bi-Lipschitz, there
exists $c_2>0$ such that
\[
\|u\|\geq c_2\|\phi^{-1}(u)-\bar y\|.
\]
Hence
\[
h(y)\geq h(\bar y)+c_1c_2^2\|y-\bar y\|^2
\]
for all $y\in\M$ sufficiently close to $\bar y$, proving that
$\bar y$ is a strong local minimizer.
\end{proof}

%% file: shared/appendix_geometry_examples.tex
\section{Examples and proofs for the intersection geometry in Section~\ref{sec:intersection_property}}
\label{app:section4-examples}
\subsection{\texorpdfstring{Vectorized representation of the differential}
{Vectorized representation of the differential}}

Let $\mathcal H_Y$ be the map defined in
Section~\ref{sec:intersection_property}. For
$X\in\M_g(Y)$ and $\eta\in\T_X\M_g(Y)$,
\[
    \mathrm D\mathcal H_Y(X)[\eta]
    =
    \eta^\top X+X^\top\eta.
\]
If $\{i,j\}\notin\mathcal E(Y)$, then the $i$th and $j$th column
supports are disjoint. Since both $X$ and $\eta$ have the support
pattern prescribed by $\M_g(Y)$,
\[
    \bigl(\eta^\top X+X^\top\eta\bigr)_{ij}=0.
\]
Consequently,
\[
    \operatorname{range}\bigl(\mathrm D\mathcal H_Y(X)\bigr)
    \subseteq \mathbb S_{r,\gamma}(Y).
\]

We now give the coordinate representation used in
Section~\ref{sec:opt_on_intersection}. Fix
\[
    X\in\overline{\M}_\Delta(Y),
    \qquad
    x:=\vvec(X),
    \qquad
    s:=\dim\M_g(Y).
\]
For \(1\leq i\leq j\leq r\), let \(S_{ij}\in\mathbb S^r\) be the
elementary symmetric matrices
\begin{equation}\label{eq:Sij}
    S_{ij}
    =
    \begin{cases}
        e_i e_i^\top, & i=j,\\
        e_i e_j^\top+e_j e_i^\top, & i<j,
    \end{cases}
\end{equation}
where \(e_i\) is the \(i\)th standard basis vector of \(\mathbb R^r\).

Since
\[
    \rmN_X\St_\Delta(n,r)
    =
    \{X\Lambda:\Lambda\in\mathbb S^r\}
\]
and \(X\) has full column rank,
\(\{\vvec(XS_{ij})\}_{1\leq i\leq j\leq r}\) is a basis of
\(\vvec(\rmN_X\St_\Delta(n,r))\). Hence
\begin{equation}\label{eq:basis B}
    B_x=(I_r\otimes X)U_r,
\end{equation}
where
\(U_r\in\mathbb R^{r^2\times r(r+1)/2}\) is the duplication matrix
satisfying
\[
    U_r\svec(\Lambda)=\vvec(\Lambda),
    \qquad
    \Lambda\in\mathbb S^r.
\]

Let
\begin{equation}\label{eq:tangent_M_g_basis}
    E_x\in\mathbb R^{nr\times s}
\end{equation}
be an orthonormal basis matrix of
\(\vvec(\T_X\M_g(Y))\). Thus every
\(\eta\in\T_X\M_g(Y)\) can be represented as
\[
    \vvec(\eta)=E_x\widehat{\vvec}(\eta).
\]
Because \(\M_g(Y)\) is a fixed coordinate subspace, \(E_x\) may be
chosen independently of \(x\); the subscript is retained here only
to match the notation for \(B_x\).

Order the coordinates of \(\svec\) by the pairs \(1\leq i\leq j\leq r\),
and define
\[
    C_r:=\operatorname{diag}(c_{ij}),
    \qquad
    c_{ij}
    =
    \begin{cases}
        2, & i=j,\\
        1, & i<j.
    \end{cases}
\]
Then \(C_r\) is invertible and
\[
    \svec\bigl(\eta^\top X+X^\top\eta\bigr)
    =
    C_rB_x^\top\vvec(\eta)
    =
    C_rB_x^\top E_x\widehat{\vvec}(\eta).
\]
{\color{blue}Consequently, the coordinate matrix of
$\mathrm D\mathcal H_Y(X)$ in these bases is $C_rA_x$, where $A_x$ is
defined in \eqref{eq:matrix_A}. Since $C_r$ is invertible, this proves
the rank identity \eqref{eq:rank_DH_rankA}.}

Let
\(P_x\in\mathbb R^{nr\times(nr-s)}\) be an orthonormal basis matrix
of \(\vvec(\rmN_X\M_g(Y))\). Then
\([E_x\ \ P_x]\) is orthogonal and \(P_x^\top E_x=0\).

\begin{lemma}\label{lem:simple_dim_lem}
It holds that
\[
    \rank
    \begin{pmatrix}
        B_x^\top\\[2pt]
        P_x^\top
    \end{pmatrix}
    =
    (nr-s)+\rank(A_x).
\]
\end{lemma}

\begin{proof}
Right multiplication by the orthogonal matrix
\([E_x\ \ P_x]\) preserves rank. Therefore,
\[
\begin{aligned}
    \rank
    \begin{pmatrix}
        B_x^\top\\
        P_x^\top
    \end{pmatrix}
    &=
    \rank
    \begin{pmatrix}
        B_x^\top E_x & B_x^\top P_x\\
        0             & I_{nr-s}
    \end{pmatrix} \\
    &=
    \rank(B_x^\top E_x)+(nr-s) \\
    &=
    \rank(A_x)+(nr-s).
\end{aligned}
\]
\end{proof}

\subsection{\texorpdfstring{Counterexamples for Assumption~\ref{assump:intersection}}{Counterexamples for Assumption}}

\begin{example}\label{ex:nonzero-intersection-r3-denseDelta}
	Let \(n=r=3\) and set
	\[
	Y=\frac{1}{\sqrt{2}}
	\begin{pmatrix}
		1&1&0\\
		1&0&1\\
		0&1&1
	\end{pmatrix}.
	\]
	Then
	\[
	Y^\top Y=
	\begin{pmatrix}
		1&\frac12&\frac12\\
		\frac12&1&\frac12\\
		\frac12&\frac12&1
	\end{pmatrix}
	=I_3+\Delta,
	\qquad
	\Delta=\frac12
	\begin{pmatrix}
		0&1&1\\
		1&0&1\\
		1&1&0
	\end{pmatrix}.
	\]
	Thus \(\Delta\in\mathbb S^3_0\) and
	\(I_3+\Delta\succ0\), with eigenvalues \(2,\frac12,\frac12\).
	
	Define
	\[
	\Lambda=
	\begin{pmatrix}
		-1&1&1\\
		1&-1&1\\
		1&1&-1
	\end{pmatrix}\in\mathbb S_3,
	\qquad
	Z=Y\Lambda.
	\]
	A direct calculation gives
	\[
	Z=\frac{1}{\sqrt{2}}
	\begin{pmatrix}
		0&0&2\\
		0&2&0\\
		2&0&0
	\end{pmatrix}\neq 0.
	\]
	Since \(Z=Y\Lambda\) with \(\Lambda\in\mathbb S_3\), we have
	\(Z\in\rmN_Y\St_\Delta(3,3)\). Moreover,
	\(Z\in\rmN_Y\M_g(Y)\) because \(Z\) vanishes on the support of \(Y\).
	
	Therefore,
	\[
	0\neq Z\in  \rmN_Y\St_\Delta(3,3)\cap \rmN_Y\M_g(Y),
	\]
	and the support graph is complete. This shows that
	Assumption~\ref{assump:intersection} may fail for \(r=3\) under  
	nonzero zero-diagonal perturbations if the one-column-witness condition is absent.
\end{example}

\begin{example}\label{example:r_4_wrong}
	Fix $n=8$, $r=4$,   consider matrices $X\in\mathbb R^{8\times 4}$ in the following form
	\begin{equation}\label{eq:block-pattern-8x4}
		X=
		\begin{bmatrix}
			A & B\\
			C & 0\\
			0 & D
		\end{bmatrix},
		\qquad
		A,B\in\mathbb R^{4\times 2},\ \ C,D\in\mathbb R^{2\times 2},
	\end{equation}
	 where all   entries in $A,B,C,D$ are nonzero.
	Let $\M_g$ denote the active manifold associated with this pattern,
	\[
	\M_g:=\{X:\ \Proj_{\mathcal S^c}X=0\},
	\]
	where $\mathcal S$ denotes the support.
	Assume $C$ and $D$ are invertible.
	\medskip
	For any $X$ of the form \eqref{eq:block-pattern-8x4} with $\det(C)\det(D)\neq 0$,
	an element $X\Lambda$ belongs to $\rmN_X\St(8,4)\cap \rmN_X\M_g$ iff $\Lambda\in\mathbb S^4$ and $(X\Lambda)_{ij}=0$ for all $(i,j)\in\mathcal S$.
	Write
	\[
	\Lambda=\begin{bmatrix}\Lambda_{11}&K\\K^\top&\Lambda_{22}\end{bmatrix},\qquad
	\Lambda_{11},\Lambda_{22}\in\mathbb S^2,\ \ K\in\mathbb R^{2\times2}.
	\]
 It follows from	the support constraints on $C$ and $D$ that
	$C\Lambda_{11}=0$ and $D\Lambda_{22}=0$. One has $\Lambda_{11}=\Lambda_{22}=0$. Thus
	\[
	\Lambda=\begin{bmatrix}0&K\\K^\top&0\end{bmatrix},
	\qquad
	X\Lambda\in \rmN_X\St\cap \rmN_X\M_g
	\iff
	AK=0,\ \ BK^\top=0.
	\]
	Consequently,
	\begin{equation}\label{eq:dim-reduction}
		\dim\bigl(\rmN_X\St(8,4)\cap \rmN_X\M_g\bigr)
		=
		\dim\{K\in\mathbb R^{2\times2}:\ AK=0,\ BK^\top=0\}.
	\end{equation}
	
	\medskip
	\medskip
	\noindent{(i) A point $\tilde X$ with $\rank(A)=\rank(B)=2$ (hence $\dim=0$).}
	Take
	\[
	\tilde A=\frac{1}{10}\begin{bmatrix}
		1&-1\\
		1&-2\\
		1&-3\\
		1&-4
	\end{bmatrix},
	\qquad
	\tilde B=\frac{1}{20}\begin{bmatrix}
		3&-4\\
		-5&7\\
		1&-2\\
		1&-1
	\end{bmatrix}.
	\]
	Then $\rank(\tilde A)=\rank(\tilde B)=2$ and $\tilde A^\top \tilde B=0$.
	Moreover, we get two full rank matrices:
	\[
	I_2-\tilde A^\top \tilde A=
	\begin{bmatrix}
		\frac{24}{25}&\frac{1}{10}\\[2pt]
		\frac{1}{10}&\frac{7}{10}
	\end{bmatrix}=: \tilde M,
	\qquad
	I_2-\tilde B^\top \tilde B=
	\begin{bmatrix}
		\frac{91}{100}&\frac{1}{8}\\[2pt]
		\frac{1}{8}&\frac{33}{40}
	\end{bmatrix}=: \tilde N.
	\]
	Define invertible $\tilde C,\tilde D$ by the   factorizations $\tilde C^\top \tilde C=\tilde M$ and
	$\tilde D^\top \tilde D=\tilde N$:
	\[
	\tilde C=
	\begin{bmatrix}
		a_C & \tfrac{s_C+t_C}{2}\\[2pt]
		a_C & \tfrac{s_C-t_C}{2}
	\end{bmatrix},
	\qquad
	\begin{aligned}
		a_C&=\sqrt{\tfrac{\tilde M_{11}}{2}}=\tfrac{2\sqrt3}{5},\\
		s_C&=\tfrac{\tilde M_{12}}{a_C}=\tfrac{\sqrt3}{12},\\
		t_C&=\sqrt{2\tilde M_{22}-s_C^2}=\sqrt{\tfrac{331}{240}}.
	\end{aligned}
	\]
	\[
	\tilde D=
	\begin{bmatrix}
		a_D & \tfrac{s_D+t_D}{2}\\[2pt]
		a_D & \tfrac{s_D-t_D}{2}
	\end{bmatrix},
	\qquad
	\begin{aligned}
		a_D&=\sqrt{\tfrac{\tilde N_{11}}{2}}=\sqrt{\tfrac{91}{200}},\\
		s_D&=\tfrac{\tilde N_{12}}{a_D}=\tfrac{5\sqrt2}{4\sqrt{91}},\\
		t_D&=\sqrt{2\tilde N_{22}-s_D^2}=\sqrt{\tfrac{5881}{3640}}.
	\end{aligned}
	\]
	
	Take $\tilde X$ in \eqref{eq:block-pattern-8x4} with $(A,B,C,D)=(\tilde A,\tilde B,\tilde C,\tilde D)$.
	One has $\tilde X^\top \tilde X=I_4$. It follows from $\rank(\tilde A)=\rank(\tilde B)=2$ and 
	\eqref{eq:dim-reduction} that
	\[
	\dim\bigl(\rmN_{\tilde X}\St(8,4)\cap \rmN_{\tilde X}\M_g\bigr)=0.
	\]
	
	\medskip
	\noindent{(ii) A point $\hat X$ with $\rank(A)=\rank(B)=1$ (hence $\dim=1$) and the same signed-support as $\tilde{X}$.}
	Take
	\[
	\hat A=\frac{1}{10}\begin{bmatrix}
		1&-1\\
		1&-1\\
		1&-1\\
		1&-1
	\end{bmatrix},
	\qquad
	\hat B=\frac{1}{20}\begin{bmatrix}
		3&-3\\
		-5&5\\
		1&-1\\
		1&-1
	\end{bmatrix}.
	\]
	Then $\rank(\hat A)=\rank(\hat B)=1$, $\hat A^\top \hat B=0$, and
	\[
	I_2-\hat A^\top \hat A=
	\begin{bmatrix}
		\frac{24}{25}&\frac{1}{25}\\[2pt]
		\frac{1}{25}&\frac{24}{25}
	\end{bmatrix}=: \hat M,
	\qquad
	I_2-\hat B^\top \hat B=
	\begin{bmatrix}
		\frac{91}{100}&\frac{9}{100}\\[2pt]
		\frac{9}{100}&\frac{91}{100}
	\end{bmatrix}=: \hat N.
	\]
	Define $\hat C^\top \hat C=\hat M$ and $\hat D^\top \hat D=\hat N$ by
	\[
	\hat C=
	\begin{bmatrix}
		\bar a_C & \tfrac{\bar s_C+\bar t_C}{2}\\[2pt]
		\bar a_C & \tfrac{\bar s_C-\bar t_C}{2}
	\end{bmatrix},
	\qquad
	\begin{aligned}
		\bar a_C&=\sqrt{\tfrac{\hat M_{11}}{2}}=\tfrac{2\sqrt3}{5},\\
		\bar s_C&=\tfrac{\hat M_{12}}{\bar a_C}=\tfrac{\sqrt3}{30},\\
		\bar t_C&=\sqrt{2\hat M_{22}-\bar s_C^2}=\sqrt{\tfrac{23}{12}}.
	\end{aligned}
	\]
	\[
	\hat D=
	\begin{bmatrix}
		\bar a_D & \tfrac{\bar s_D+\bar t_D}{2}\\[2pt]
		\bar a_D & \tfrac{\bar s_D-\bar t_D}{2}
	\end{bmatrix},
	\qquad
	\begin{aligned}
		\bar a_D&=\sqrt{\tfrac{\hat N_{11}}{2}}=\sqrt{\tfrac{91}{200}},\\
		\bar s_D&=\tfrac{\hat N_{12}}{\bar a_D}=\tfrac{9\sqrt2}{10\sqrt{91}},\\
		\bar t_D&=\sqrt{2\hat N_{22}-\bar s_D^2}=\sqrt{\tfrac{164}{91}}.
	\end{aligned}
	\]
	
	Set $\hat X$  in \eqref{eq:block-pattern-8x4} with $(A,B,C,D)=(\hat A,\hat B,\hat C,\hat D)$.
	Then $\hat X^\top \hat X=I_4$, and $\ker(\hat A)=\ker(\hat B)=\mathrm{span}\{(1,1)^\top\}$, so
	we have $\dim(\rmN_{\hat X}\St\cap \rmN_{\hat X}\M_g)=1$.

	Even if $\tilde X$ and $\hat X$ have the same  sign and support, $\dim(\rmN_X\St\cap \rmN_X\M_g)$ takes different values $0$ and $1$ at $\tilde X$ and $\hat X$. Hence, Assumption~\ref{assump:intersection} fails.
\end{example}

\subsection{\texorpdfstring{Proof of the $r=2$ verification}{Proof of the r=2 verification}}
\label{app:proof-intersection-r2}

\begin{proof}[Proof of Proposition~\ref{prop:intersection-for_r=2}]
{\color{blue}
Let
\[
    D:=I_2+\Delta
    =
    \begin{pmatrix}
        d_1 & c\\
        c   & d_2
    \end{pmatrix}\succ0,
    \qquad
    \Lambda=
    \begin{pmatrix}
        a & u\\
        u & b
    \end{pmatrix},
\]
and take
\[
    Z=X\Lambda\in
    \rmN_X\St_\Delta(n,2)\cap \rmN_X\M_g(Y).
\]
Since $X$ and $Y$ have the same support in $\mathcal U_s(Y)$,
$Z$ vanishes on every active entry of $X$.
By Lemma~\ref{lem:diag-vanish},
\[
    d_1a+cu=0,
    \qquad
    cu+d_2b=0.
    \tag{*}
\]

Suppose first that $c\neq0$. Then the two columns of $X$ have
overlapping supports, so $\{1,2\}\in\mathcal E(Y)$. If some nonzero row of
$X$ has only one active entry, say
$X_{\ell1}\neq0$ and $X_{\ell2}=0$, then
\(
    0=Z_{\ell1}=aX_{\ell1},
\) and hence $a=0$. The first equation in $(*)$ gives $u=0$, and the
second gives $b=0$. The case in which only the second entry is active
is analogous.

If no nonzero row has only one active entry, then every nonzero row
of $X$ has both entries active. Since $Z$ vanishes on the active
entries, this gives $Z=X\Lambda=0$. As $X^\top X=D\succ0$, $X$ has
full column rank, and hence $\Lambda=0$. Therefore
\(
    \rmN_X\St_\Delta(n,2)\cap \rmN_X\M_g(Y)=\{0\}
\)
whenever $c\neq0$.

Now suppose that $c=0$. Since $d_1,d_2>0$, $(*)$ gives
$a=b=0$, and hence
\(
    Z=uXS_{12}.
\)
If $\{1,2\}\in\mathcal E(Y)$, there exists a row $\ell$ with
$X_{\ell1}X_{\ell2}\neq0$. Since both entries are active,
\(
    0=Z_{\ell1}=uX_{\ell2}, 
    0=Z_{\ell2}=uX_{\ell1},
\)
so $u=0$.

If $\{1,2\}\notin\mathcal E(Y)$, the column supports are disjoint.
Then $XS_{12}$ vanishes on the support of $X$, and therefore
\[
    XS_{12}\in \rmN_X\M_g(Y).
\]
Since $S_{12}$ is symmetric,
$XS_{12}\in \rmN_X\St_\Delta(n,2)$ as well. Consequently,
\[
    \rmN_X\St_\Delta(n,2)\cap \rmN_X\M_g(Y)
    =
    \operatorname{span}\{XS_{12}\}.
\]
Combining the two cases proves the result.
}
\end{proof}

\subsection{\texorpdfstring{Proof of the $r=3$ verification}{Proof of the r=3 verification}}
\label{app:proof-intersection-r3-clean}

\begin{proof}[Proof of Proposition~\ref{prop:intersection-r3-clean}]
	Take $Z\in \rmN_X\St(n,3)\cap \rmN_X\M_g(Y)$ and write $Z=X\Lambda$ with
	$\Lambda\in\mathbb S_3$.

	\emph{Step 1.}
	It follows from Lemma~\ref{lem:diag-vanish} that $\Lambda_{11}=\Lambda_{22}=\Lambda_{33}=0$. Therefore
	\begin{equation}\label{eq:Z_alpha}
			\Lambda=\alpha_{12}S_{12}+\alpha_{13}S_{13}+\alpha_{23}S_{23},
			\qquad
			Z=\alpha_{12}XS_{12}+\alpha_{13}XS_{13}+\alpha_{23}XS_{23}.
		\end{equation}

	We distinguish two cases.

	\smallskip
		\noindent{Case A: there exists a row $\ell$ with $X_{\ell1}X_{\ell2}X_{\ell3}\neq 0$.}
		Let $(x_1,x_2,x_3):=(X_{\ell1},X_{\ell2},X_{\ell3})$. Since all three entries are active,
		$Z_{\ell1}=Z_{\ell2}=Z_{\ell3}=0$, and using
		\[
		(XS_{12})_{\ell\cdot}=(x_2,x_1,0),\quad
		(XS_{13})_{\ell\cdot}=(x_3,0,x_1),\quad
		(XS_{23})_{\ell\cdot}=(0,x_3,x_2),
		\]
	we obtain
	\[
	\begin{bmatrix}
			x_2&x_3&0\\
			x_1&0&x_3\\
			0&x_1&x_2
		\end{bmatrix}
	\begin{bmatrix}\alpha_{12}\\ \alpha_{13}\\ \alpha_{23}\end{bmatrix}=0.
	\]
	Since  $\det	\begin{bmatrix}
			x_2&x_3&0\\
			x_1&0&x_3\\
			0&x_1&x_2
		\end{bmatrix}=-2x_1x_2x_3\neq 0$ , it follows that $\alpha_{12}=\alpha_{13}=\alpha_{23}=0$,
		hence $Z=0$.  Therefore, $\rmN_X\St(n,3)\cap \rmN_X\M_g(Y)=\{0\}$.

	\smallskip
	\noindent{Case B: no row has three nonzeros among columns $\{1,2,3\}$.}
	Then every row activates at most two columns. In particular, for any edge
		$\{i,j\}\in\mathcal E(Y)$ there exists a row $\ell$ such that
		$X_{\ell i}\neq 0$, $X_{\ell j}\neq 0$, and necessarily $X_{\ell k}=0$ for the remaining
		$k\in\{1,2,3\}\setminus\{i,j\}$.

		Evaluating $Z$ on this row, the terms in \eqref{eq:Z_alpha} involving $S_{ik}$ or $S_{jk}$ vanish because they
		carry a factor $X_{\ell k}=0$. Hence the constraints $Z_{\ell i}=Z_{\ell j}=0$ reduce to
		the $r=2$-type relations
		\[
		Z_{\ell i}=\alpha_{ij}X_{\ell j}=0,\qquad
		Z_{\ell j}=\alpha_{ij}X_{\ell i}=0,
		\]
		which force $\alpha_{ij}=0$ since $X_{\ell i},X_{\ell j}\neq 0$.
			Therefore $\alpha_{ij}=0$ for every present edge $\{i,j\}\in\mathcal E(Y)$,
			and $Z$ can only involve coefficients on missing edges:
		\[
		Z\in \mathrm{span}\{\,XS_{ij}: i<j,\ \{i,j\}\notin\mathcal E(Y)\,\}.
		\]

		Finally, the reverse inclusion follows from Proposition~\ref{lem:M_1 and M_2 are not transverse}:
		if $\{i,j\}\notin\mathcal E(Y)$ then the supports of the $i$th and $j$th columns are disjoint, and $X\in\mathcal U_s(Y)$ gives the same support pattern for $X$ and $Y$.
		Thus $XS_{ij}\in \rmN_X\St(n,3)\cap \rmN_X\M_g(Y)$. This completes the proof.
\end{proof}

\subsection{\texorpdfstring{Proof of the $2/3$-row cover verification}{Proof of the 2/3-row cover verification}}
\label{app:proof-verify-intersection}

\begin{proof}[Proof of Proposition~\ref{prop:verify-intersection}]
	Throughout the proof, we work on the support-invariant manifold $\M_g(Y)$, hence
	$\mathcal G(X)=\mathcal G(Y)$ for all $X\in\M_g(Y)\cap\mathcal U_s(Y)$.

	Let $\mathcal L\subseteq[n]$ be the set of rows in Assumption~\ref{assump:row-cover}, so that
	$| J_\ell^Y|\in\{2,3\}$ for all $\ell\in\mathcal L$ and these rows cover all present
	off-diagonal edges in $\mathcal E_{\rm off}(Y)$.

	\smallskip
	\noindent\emph{Step 1 (the ``$\supseteq$'' inclusion).}
	This follows from Proposition~\ref{lem:M_1 and M_2 are not transverse}:
	if $i<j$ and $\{i,j\}\notin\mathcal E_{\rm off}(Y)$, then the supports of
	columns $i$ and $j$ are disjoint, hence
	$XS_{ij}\in \rmN_X\St_\Delta(n,r)\cap \rmN_X\M_g(Y)$.

	\smallskip
	\noindent\emph{Step 2 (the ``$\subseteq$'' inclusion).}
	Take any $Z\in \rmN_X\St_\Delta(n,r)\cap \rmN_X\M_g(Y)$. Write $Z=X\Lambda$ with
	$\Lambda\in{\color{blue}\mathbb S^r}$.
	By Lemma~\ref{lem:diag-vanish}, we have
	\begin{equation}\label{eq:diag_IplusDelta}
		\diag\bigl((I_r+\Delta)\Lambda\bigr)=0.
	\end{equation}

	\smallskip
	\noindent\emph{Step 2a (vanishing diagonal of $\Lambda$).}
	If $\Delta=0$, then Lemma~\ref{lem:diag-vanish} already yields $\Lambda_{ii}=0$ for all $i$.
	Assume now $\Delta\neq 0$. By the condition~\eqref{assump:one_colm_witness} of Assumption~\ref{assump:row-cover}, for each $i\in[r]$ there exists
	a row $\ell_i$ such that  $J_{\ell_i}^X=\{i\}$ and hence the row vector
	$x:=X_{\ell_i,:}$ satisfies $x_i\neq 0$ and $x_j=0$ for $j\neq i$. Because
	$Z\in\rmN_X\M_g(Y)$, we have $Z_{\ell_i i}=0$, and thus
	\[
	0=Z_{\ell_i i}=(X\Lambda)_{\ell_i i}=x\Lambda e_i=x_i\Lambda_{ii},
	\]
	which implies $\Lambda_{ii}=0$. Therefore, $\diag(\Lambda)=0$ in both cases $\Delta=0$ and
	$\Delta\neq 0$.
	\smallskip
	\noindent\emph{Step 2b (vanishing on present edges).}
	Let $\{i,j\}\in\mathcal E_{\rm off}(Y)$. By Assumption~\ref{assump:row-cover}, there exists
	$\ell\in\mathcal L$ such that $\{i,j\}\subseteq  J_\ell^Y$ and
	$| J_\ell^Y|\in\{2,3\}$. Set $J:= J_\ell^Y$ and denote $x:=X_{\ell,:}$, so
	that $x_k=0$ for $k\notin J$ and $x_k\neq 0$ for $k\in J$.
	Since $Z\in \rmN_X\M_g(Y)$, we have $Z_{\ell k}=0$ for all $k\in J$, i.e.,
	\begin{equation}\label{eq:row_constraint}
		0=Z_{\ell J}=(X\Lambda)_{\ell J}=x_J\,\Lambda_{J,J}.
	\end{equation}
	\begin{itemize}
		\item If $|J|=2$, say $J=\{i,j\}$, then using $\Lambda_{ii}=\Lambda_{jj}=0$, the two coordinates
		of \eqref{eq:row_constraint} give
		\[
		0=x_i\Lambda_{ij},\qquad 0=x_j\Lambda_{ij},
		\]
		hence $\Lambda_{ij}=0$.

		\item If $|J|=3$, say $J=\{i,j,k\}$, then using $\Lambda_{ii}=\Lambda_{jj}=\Lambda_{kk}=0$,
		the three equations in \eqref{eq:row_constraint} reduce to
		\[
		\begin{bmatrix}
			x_j & x_k & 0\\
			x_i & 0   & x_k\\
			0   & x_i & x_j
		\end{bmatrix}
		\begin{bmatrix}
			\Lambda_{ij}\\ \Lambda_{ik}\\ \Lambda_{jk}
		\end{bmatrix}
		=0,
		\]
		whose determinant equals $-2x_i x_j x_k\neq 0$. Hence
		$\Lambda_{ij}=\Lambda_{ik}=\Lambda_{jk}=0$, in particular $\Lambda_{ij}=0$.
	\end{itemize}
	Since $\mathcal L$ covers all present off-diagonal edges, we conclude that
	$\Lambda_{ij}=0$ for every $\{i,j\}\in\mathcal E_{\rm off}(Y)$.

	\smallskip
	\noindent
	Thus $\Lambda$ has zero diagonal and is supported only on missing off-diagonal pairs
	$\{i,j\}\notin\mathcal E_{\rm off}(Y)$, so
	\[
	\Lambda=\sum_{\substack{i<j\\ \{i,j\}\notin\mathcal E_{\rm off}(Y)}} \Lambda_{ij} S_{ij},
	\qquad
	Z=X\Lambda=\sum_{\substack{i<j\\ \{i,j\}\notin\mathcal E_{\rm off}(Y)}} \Lambda_{ij}\,XS_{ij},
	\]
	which proves
	$Z\in \mathrm{span}\{XS_{ij}:i<j,\ \{i,j\}\notin\mathcal E_{\rm off}(Y)\}$ and completes the proof.
\end{proof}

\subsection{\texorpdfstring{Proof of the Bernoulli row-cover bound}{Proof of the Bernoulli row-cover bound}}

\begin{proof}[Proof of Proposition~\ref{prop:row-cover-bernoulli}]
	Let the active set of row $\ell$ be
	\(J_\ell=\{j\in[r]:\delta_{\ell j}=1\}\), and set \(K_\ell:=|J_\ell|\).

	\smallskip
	\noindent{(a) the probability of 2/3 row cover.}
	Fix a pair $\{i,j\}\subseteq[r]$ with $i\neq j$. For each row $\ell\in[n]$, define the events
	\[
	\mathsf B^{ij}_\ell:=\{\delta_{\ell i}=\delta_{\ell j}=1\}=\{\{i,j\}\subseteq J_\ell\},
	\qquad
	\mathsf A^{ij}_\ell:=\{\{i,j\}\subseteq J_\ell,\ |J_\ell|\in\{2,3\}\}.
	\]
	Clearly $\mathsf A^{ij}_\ell\subseteq \mathsf B^{ij}_\ell$. Under Assumption~\ref{assump:bernoulli-support},
	the row indicator vectors $(\delta_{\ell 1},\dots,\delta_{\ell r})$ are i.i.d.\ across $\ell$,
	hence for fixed $(i,j)$ the events $\{\mathsf A^{ij}_\ell\}_{\ell=1}^n$ are i.i.d.

	Moreover, $\mathsf A^{ij}_\ell$ occurs iff $\delta_{\ell i}=\delta_{\ell j}=1$ and either
	(i) all other $r-2$ indicators are zero (so $|J_\ell|=2$), or
	(ii) exactly one among the other $r-2$ indicators equals one (so $|J_\ell|=3$). Therefore
	\[
	q_{\rm rc}:=\Pr(\mathsf A^{ij}_\ell)
	=p^2(1-p)^{r-2}+(r-2)p^3(1-p)^{r-3}.
	\]
	Let $\mathsf{Bad}_{ij}$ be the event that $\{i,j\}$ is an off-diagonal edge, i.e., some
	row has $\delta_{\ell i}=\delta_{\ell j}=1$, but it is not covered by any $2/3$-cover row:
	\[
	\mathsf{Bad}_{ij}
	:=\Bigl(\exists\,\ell\in[n]:\mathsf B^{ij}_\ell\Bigr)\ \cap\
	\Bigl(\forall\,\ell\in[n]:(\mathsf A^{ij}_\ell)^c\Bigr).
	\]
	Since $\mathsf A^{ij}_\ell\subseteq \mathsf B^{ij}_\ell$, we have
	\[
	\begin{aligned}
		\Pr(\mathsf{Bad}_{ij})
		&=\Pr\Bigl(\forall\,\ell:(\mathsf A^{ij}_\ell)^c\Bigr)
		-\Pr\Bigl(\forall\,\ell:(\mathsf B^{ij}_\ell)^c\Bigr)\\
		&\le \Pr\Bigl(\forall\,\ell:(\mathsf A^{ij}_\ell)^c\Bigr)
		=(1-q_{\rm rc})^n
		\le e^{-nq_{\rm rc}},
	\end{aligned}
	\]
	where we use $1-x\le e^{-x}$. By the union bound over $1\le i<j\le r$,
	\[
	\Pr\bigl(\text{$2/3$-row cover fails}\bigr)
	\le \sum_{1\le i<j\le r}\Pr(\mathsf{Bad}_{ij})
	\le \binom{r}{2}e^{-nq_{\rm rc}}.
	\]

	\smallskip
	\noindent{(b) the probability One-column witness when $\Delta\neq 0$.}
	For $i\in[r]$ and $\ell\in[n]$, define
	\[
	\mathsf W^i_\ell:=\{J_\ell=\{i\}\}
	=\{\delta_{\ell i}=1\}\cap\bigcap_{k\neq i}\{\delta_{\ell k}=0\}.
	\]
	Then $\{\mathsf W^i_\ell\}_{\ell=1}^n$ are i.i.d.\ and
	\(q_{1}:=\Pr(\mathsf W^i_\ell)=p(1-p)^{r-1}\).
	Let $\mathsf{Bad}_i:=\bigcap_{\ell=1}^n (\mathsf W^i_\ell)^c$ be the event that column $i$
	has no witness row. Then \(\Pr(\mathsf{Bad}_i)=(1-q_{1})^n\le e^{-nq_{1}}\).
	A union bound over $i\in[r]$ yields
	\[
	\Pr\bigl(\text{one-column witness fails}\bigr)
	\le \sum_{i=1}^r \Pr(\mathsf{Bad}_i)
	\le r e^{-nq_{1}}.
	\]

	\smallskip
	\noindent{(c) Combine the two cases.}
	Let $\mathsf{Bad}$ denote the failure event of Assumption~\ref{assump:row-cover}.
	By part (a) and (b),
	\[
	\Pr(\mathsf{Bad})
	\le \binom{r}{2}e^{-nq_{\rm rc}} + r\,\mathbf 1_{\{\Delta\neq 0\}} e^{-nq_{1}}
	\le \Bigl(\binom{r}{2}+r\,\mathbf 1_{\{\Delta\neq 0\}}\Bigr)e^{-n q_w},
	\]
	where $q_w=q_{\rm rc}$ when $\Delta=0$ and
	$q_w=\min\{q_{\rm rc},q_1\}$ when $\Delta\neq 0$, as in \eqref{eq:qw}.
	This proves the stated probability bound.

	Finally, if
	$
	n \ge \bigl(\log(\binom{r}{2}+r\,\mathbf 1_{\{\Delta\neq 0\}})+\log(1/\delta)\bigr)/q_w,
	$
	then
	$
	(\binom{r}{2}+r\,\mathbf 1_{\{\Delta\neq 0\}})e^{-nq_w}\le \delta,
	$
	and hence Assumption~\ref{assump:row-cover} holds with probability at least $1-\delta$.
\end{proof}

%% file: shared/appendix_optimization_intersection.tex
\section{Derivations for optimization on the intersection manifold}
\label{app:section5-optimization-derivations}

\subsection{Derivation of the tangent-intersection projection}
\label{app:proof-projection-intersection}

We use the notation of Section~\ref{sec:projection_to_inter_tangent}.
Every vector $u\in \vvec(\T_X\M_g(\bar Y))$ can be written as
$u=Ea$ for some $a\in\R^s$. The condition
\(u\in b+\vvec(\T_X\St_\Delta)\) is equivalent to
\(B_x^\top(u-b)=0\), that is,
\(A_x a=B_x^\top b\), where \(A_x:=B_x^\top E\).
Therefore the projection of $c$ onto $\T_\cap^b(x)$ is obtained by solving
\[
	\begin{aligned}
		&\min_{a\in\R^s}\frac12\|Ea-c\|^2,\\
		&\st \quad A_x a=B_x^\top b.
	\end{aligned}
\]
Since $E^\top E=I_s$, this problem is equivalent to
\eqref{prob:projection_eq}.

Under Assumption~\ref{assump:intersection}, the zero rows of $A_x$ correspond
to missing support-graph edges, and the remaining submatrix
$A_{\gamma(\bar Y)}$ has full row rank. Hence the equality constraint can be
reduced to \(A_{\gamma(\bar Y)}a=B_{\gamma(\bar Y)}^\top b\),
and the constrained least-squares problem becomes
\[
\min_{a\in\R^s}\frac12\|a-E^\top c\|^2
\qquad \text{s.t.}\qquad
A_{\gamma(\bar Y)}a=B_{\gamma(\bar Y)}^\top b.
\]
The standard Moore--Penrose formula for this equality-constrained projection
gives
\[
a
=
E^\top c-
A_{\gamma(\bar Y)}^\dagger
\bigl(A_{\gamma(\bar Y)}E^\top c-B_{\gamma(\bar Y)}^\top b\bigr),
\]
where
\[
A_{\gamma(\bar Y)}^\dagger
=
A_{\gamma(\bar Y)}^\top
\bigl(A_{\gamma(\bar Y)}A_{\gamma(\bar Y)}^\top\bigr)^{-1}.
\]
Multiplying by $E$ gives the projection formula
\eqref{eq:projection_vector_g_old}.

\subsection{Derivation of the Riemannian Hessian formula}
\label{app:proof-intersection-hessian}

We start from \eqref{hess_1}. In vectorized coordinates, let
\(g_x:=\vvec(\nabla f(X))\).
By \eqref{eq:projection_vector_g},
\[
\Proj_{\T_X\overline{\M}_\Delta(\bar Y)}g_x=(M_x-\Gamma_x)g_x,
\qquad
\Gamma_x:=G_x^\top(G_xG_x^\top)^{-1}G_x.
\]
Since $M_x=EE^\top$ is constant on the fixed-support manifold,
differentiating with respect to $x$ in the direction $\eta$ yields
\[
\Proj_\eta(g_x)
=
-\dot G_x[\eta]^\top(G_xG_x^\top)^{-1}G_xg_x
+G_x^\top\omega,
\]
where $\dot G_x[\eta]:=\mathrm D G_x(x)[\eta]$ and
$\omega\in\R^{\gamma(\bar Y)}$ is a vector whose explicit form is not needed.
Since
\(G_x=B_{\gamma(\bar Y)}^\top M_x\) and
\(B_{\gamma(\bar Y)}^\top=U_{\gamma(\bar Y)}^\top(I_r\otimes X^\top)\),
we have
\(\dot G_x[\eta]=U_{\gamma(\bar Y)}^\top(I_r\otimes \eta_m^\top)M_x\),
where $\eta_m\in\R^{n\times r}$ is the matrix representation of the vector
$\eta\in\R^{nr}$.

For any $\xi\in \T_X\overline{\M}_\Delta(\bar Y)$, $M_x\xi=\xi$ and
$G_x \xi = B_{\gamma(\bar Y)}^\top \xi =0$. Hence
\(\langle G_x^\top\omega,\xi\rangle=\langle \omega,G_x\xi\rangle=0\).
Therefore the term $G_x^\top\omega$ is normal to
$\overline{\M}_\Delta(\bar Y)$ and vanishes after projection onto the tangent
space. Combining this identity with \eqref{hess_1} gives
\eqref{eq:Rieman_Hessian}.

%% file: shared/appendix_extension.tex
\section{Extension of MIX to general equality-constrained manifolds}
\label{app:mix-general-equality}

We briefly explain how the global-convergence mechanism of MIX extends beyond
Stiefel-type constraints. This discussion is intended as a template rather than
a complete local superlinear theory.

Let \(\E\simeq\mathbb R^N\), and consider
\begin{equation}\label{prob:general-equality}
	\min_{x\in\mathcal C_0} F(x):=f(x)+g(x),
	\qquad
	\mathcal C_0:=\{x\in\E:\ h(x)=0\},
\end{equation}
where \(f\) is smooth, \(g\) is convex and Lipschitz continuous, and
\(h:\E\to\mathbb R^m\) is \(C^2\). Assume that \(\mathrm Dh(x)\) has full row
rank on the region of interest, so that \(\mathcal C_0\) is a smooth embedded
manifold with \(\T_x\mathcal C_0=\ker \mathrm Dh(x)\).

At \(x_k\in\mathcal C_0\), the ManPG predictor is obtained from the tangent
subproblem
\begin{equation}\label{eq:general-tangent-prox}
	v_k\in\argmin_{v\in\T_{x_k}\mathcal C_0}
	\left\{
	\langle \nabla f(x_k),v\rangle
	+\frac{1}{2t_k}\|v\|^2
	+g(x_k+v)
	\right\},
	\qquad
	y_k:=x_k+v_k .
\end{equation}
Since \(v_k\in\ker\mathrm Dh(x_k)\), Taylor expansion gives
\begin{equation}\label{eq:general-feasibility-shift}
	h(y_k)=h(x_k)+\mathrm Dh(x_k)v_k+O(\|v_k\|^2)
	=O(\|v_k\|^2).
\end{equation}
Thus \(y_k\) lies exactly on the nearby level-set manifold
\(\mathcal C_{\theta_k}:=\{x:\ h(x)=\theta_k\}\), where
\(\theta_k:=h(y_k)\).
If \(\mathrm Dh\) remains full rank near \(\mathcal C_0\), then
\(\mathcal C_{\theta_k}\) is a smooth embedded manifold for all sufficiently
small \(\theta_k\).

The analogue of the intersection manifold used in MIX is
\begin{equation}\label{eq:general-intersection-model}
	\mathcal I_k
	:=
	\mathcal M_g(y_k)\cap \mathcal C_{\theta_k}.
\end{equation}
By construction, \(y_k\in\mathcal I_k\). To compute a Newton-type correction on
\(\mathcal I_k\), one needs this intersection to be a smooth embedded manifold
and to admit a computable local retraction. For example, this follows under a
clean-intersection condition between \(\mathcal M_g(y_k)\) and
\(\mathcal C_{\theta_k}\), together with a retraction construction such as a
projection-type or alternating-projection-type limiting map.
Smoothness of \(\mathcal I_k\) alone is not sufficient for the
global argument; the safeguard also requires a restoration map to \(\mathcal C_0\)
and an acceptance test of the form \eqref{eq:general-projected-decrease}.

The trial point produced on \(\mathcal I_k\) is not necessarily feasible for
the original constraint \(\mathcal C_0\). Therefore, as in MIX, the trial point
is restored to \(\mathcal C_0\). Let
\(\Pi_0\) denote a local restoration map onto \(\mathcal C_0\), such as the local
Euclidean projection or another smooth retraction-like correction. We require
the restoration estimate
\begin{equation}\label{eq:general-restoration-estimate}
	\|\Pi_0(z)-z\|
	\le c_{\rm res}\|h(z)\|,
	\qquad
	z\ \text{near}\ \mathcal C_0 .
\end{equation}
For \(z\in\mathcal C_{\theta_k}\), this gives
\(\|\Pi_0(z)-z\|\le c_{\rm res}\|\theta_k\|=O(\|v_k\|^2)\),
by \eqref{eq:general-feasibility-shift}.

The global-convergence argument only uses the Newton branch through an
acceptance safeguard. Specifically, a restored Newton trial
\(x_k^{\rm trial}:=\Pi_0(z_k)\) is accepted only if
	\begin{equation}\label{eq:general-projected-decrease}
	F(x_k^{\rm trial})
	\le
	F(x_k)-\bar\tau\frac{\|v_k\|^2}{t_k}
	\end{equation}
	for some \(0<\bar\tau\le1/2\). If this condition fails, the algorithm falls back to a
ManPG step on \(\mathcal C_0\) satisfying a standard descent estimate
\begin{equation}\label{eq:general-manpg-descent}
	F(x_{k+1})
	\le
	F(x_k)-c_{\rm M}\|v_k\|^2
\end{equation}
with \(c_{\rm M}>0\). Under the usual boundedness and Lipschitz assumptions,
\eqref{eq:general-projected-decrease} and \eqref{eq:general-manpg-descent}
imply
\[
F(x_{k+1})
\le
	F(x_k)-c_{\rm dec}\|v_k\|^2,
	\qquad
	c_{\rm dec}:=\min\{\bar\tau/t_{\max},c_{\rm M}\}>0.
	\]
If \(F\) is bounded below on \(\mathcal C_0\), then
\[
\sum_{k=0}^\infty \|v_k\|^2<\infty,
\qquad
\min_{0\le j\le K-1}\|v_j\|^2
\le
\frac{F(x_0)-F_{\inf}}{c_{\rm dec}K}.
\]
Thus \(v_k\to0\). If the stepsizes satisfy
\(0<t_{\min}\le t_k\le t_{\max}<\infty\), then
\(\|v_k\|/t_k\to0\).

Finally, define the equality-constrained KKT residual by
\begin{equation}\label{eq:general-kkt-residual}
	\mathcal R_{\mathcal C}(x)
	:=
	\inf_{\substack{G\in\partial g(x)\\ \lambda\in\mathbb R^m}}
	\left(
	\|\nabla f(x)+G+\mathrm Dh(x)^*\lambda\|
	+
	\|h(x)\|
	\right).
\end{equation}
The optimality condition for \eqref{eq:general-tangent-prox} gives
multipliers \(\lambda_k\) such that
\[
\nabla f(x_k)+G_k+\frac1{t_k}v_k+\mathrm Dh(x_k)^*\lambda_k=0,
\qquad
G_k\in\partial g(y_k).
\]
Under the full-rank and boundedness assumptions above, the multipliers
\(\lambda_k\) are locally bounded. Using the Lipschitz continuity of
\(\nabla f\), the \(C^1\) smoothness of \(\mathrm Dh\), and
\eqref{eq:general-feasibility-shift}, one obtains
\begin{equation}\label{eq:general-residual-bound}
	\mathcal R_{\mathcal C}(y_k)
	\le
	C_{\rm R}\frac{\|v_k\|}{t_k}
	+
	O(\|v_k\|^2).
\end{equation}
Consequently, the ManPG residual controls the KKT residual of the original
equality-constrained problem. In particular, if
\(\|v_k\|/t_k\to0\), then
\(\mathcal R_{\mathcal C}(y_k)\to0\).

The preceding argument shows that the global-convergence mechanism of MIX is
not specific to the Stiefel constraint. What is specific to the sparse Stiefel
setting is the explicit construction and verification of the intersection
manifolds, tangent projections, Hessian formulas, and retractions. For a
general equality-constrained manifold, the same global safeguard applies once
the following ingredients are available: a ManPG descent step on
\(\mathcal C_0\), smooth nearby level sets \(\mathcal C_{\theta}\), a smooth
intersection set \(\mathcal M_g(y_k)\cap\mathcal C_{\theta_k}\), a local
retraction on this intersection, and a restoration map satisfying
\eqref{eq:general-restoration-estimate}. Local superlinear convergence would
require additional identification and {\color{blue}an SOSC analogous to that
used for the Stiefel case}.